\documentclass[11pt,reqno]{amsart}
\usepackage[utf8]{inputenc}
\usepackage{amsmath,amssymb,amsthm}
\usepackage[left=1in, right=1in, top=1.2in, bottom=1.2in]{geometry}
\usepackage{framed}
\usepackage{graphicx}
\usepackage{bm}
\usepackage{theoremref}
\usepackage{enumitem}
\usepackage{color}
\usepackage{float}
\usepackage{dsfont}
\usepackage{mathtools}
\usepackage{breqn}
\usepackage{bbm}
\usepackage{mathrsfs}
\usepackage{amstext}
\usepackage{hyperref}
\hypersetup{
   colorlinks  = true,
   citecolor  = blue,
   linkcolor  = blue
}
\usepackage{pdfsync}
\usepackage{amsthm}
\usepackage{amssymb}
\usepackage{mathrsfs}

\usepackage{color}

\usepackage{verbatim}
\usepackage[utf8]{inputenc}
\usepackage[english]{babel}

\DeclarePairedDelimiter\floor{\lfloor}{\rfloor}

\numberwithin{equation}{section}
\newtheorem{theorem}{Theorem}[section]
\newtheorem{corollary}[theorem]{Corollary}
\newtheorem{lemma}[theorem]{Lemma}
\newtheorem{proposition}[theorem]{Proposition}

\theoremstyle{remark}
\newtheorem{definition}[theorem]{Definition}
\newtheorem{remark}[theorem]{Remark}
\newtheorem{example}[theorem]{Example}

\let\oldproofname=\proofname
\renewcommand{\proofname}{\rm\bf{\oldproofname}}

\newcommand{\KR}{\text{KR}}
\newcommand{\lip}{\text{lip}}
\newcommand{\Lip}{\text{Lip}}

\newcommand{\supp}{\text{supp}}

\newcommand{\cB}{\mathcal{B}}
\newcommand{\cC}{\mathcal{C}}
\newcommand{\cD}{\mathcal{D}}
\newcommand{\cE}{\mathcal{E}}
\newcommand{\cF}{\mathcal{F}}

\newcommand{\cH}{\mathcal{H}}
\newcommand{\cI}{\mathcal{I}}

\newcommand{\cL}{\mathcal{L}}
\newcommand{\cM}{\mathcal{M}}
\newcommand{\cN}{\mathcal{N}}

\newcommand{\cP}{\mathcal{P}}

\newcommand{\cS}{\mathcal{S}}

\newcommand{\cW}{\mathcal{W}}

\newcommand{\EE}{\mathbb{E}}

\newcommand{\bN}{\mathbb{N}}

\newcommand{\PP}{\mathbb{P}}

\newcommand{\bZ}{\mathbb{Z}}

\renewcommand{\bN}{\mathbf{N}}

\newcommand{\bR}{\mathbf{R}}
\newcommand{\bS}{\mathbf{S}}

\renewcommand{\bZ}{\mathbf{Z}}

\newcommand{\fB}{\mathfrak{B}}

\newcommand{\fD}{\mathfrak{D}}

\newcommand{\fG}{\mathfrak{G}}

\newcommand{\dd}{\mathop{}\!\mathrm{d}}

\newcommand{\vertiii}[1]{{\left\vert\kern-0.25ex\left\vert\kern-0.25ex\left\vert #1 
		\right\vert\kern-0.25ex\right\vert\kern-0.25ex\right\vert}}
\title{Pathwise Regularisation of Singular Interacting Particle Systems and their Mean Field Limits}
\author{Fabian A. Harang \and Avi Mayorcas}
\date{November 2021}

\address{Fabian A. Harang: 
 Department of Economics, BI Norwegian Business School,N-0442, Oslo, Norway 
 \\
 and Department of Mathematics, University of Oslo, Norway}
\email{fabian.a.harang@bi.no} 
\address{Avi J. Mayorcas: DPMMS, University of Cambridge, CMS, Wilberforce Road, CB3 0WB, UK.}
\email{am3015@cam.ac.uk} 

\thanks{\emph{Acknowledgments.} 
We are grateful to the anonymous referees for their helpful comments which have greatly improved the manuscript. FH gratefully acknowledges financial support from the STORM project 274410, funded by the Research Council of Norway. AM was supported by the EPSRC Centre For Doctoral Training
in Partial Differential Equations: Analysis and Applications [grant number EP/L015811/1].}

\begin{document}

\maketitle

\begin{abstract}
  We investigate the regularizing effect of certain perturbations by noise in singular interacting particle systems under the mean field scaling. In particular, we show that the addition of a suitably irregular path can regularise these dynamics and we recover the McKean--Vlasov limit under very broad assumptions on the interaction kernel; only requiring it to be controlled in a possibly distributional Besov space. In the particle system we include two sources of randomness, a common noise path $Z$ which regularises the dynamics and a family of idiosyncratic noises, which we only assume to converge in mean field scaling to a representative noise in the McKean--Vlasov equation.
\end{abstract}

{
\hypersetup{linkcolor=black}
 \tableofcontents 
}
\section{Introduction}
Interacting particle systems of the form
\begin{equation}\label{eq:IntroPartSystem}
  \begin{cases}
  \displaystyle
  \dd X_t^i = \frac{1}{N}\sum_{j=1}^N K(X^i_t-X^j_t)\,\dd t + \dd B_t^i,\\
  X^i_0 = \xi^i,
  \end{cases}
  \quad \text{ for }i=1,\ldots,N,
\end{equation}
are objects of great interest in topics as diverse as mathematical physics, biology, the social sciences, computer science and random matrix theory, cf. \cite{jabin_wang_17,hambly_ledger_sojmark_19, fournier_jourdain_17,gomes_stuart_wolfram_19,friesen_kutoviy_20,tzen_raginsky_20,chafai_15,serfaty_20}. A typical setting is to take $(B^i)_{i=1}^N$ a family of i.i.d standard Brownian motions, $(\xi^i)_{i=1}^N$ a family of i.i.d $\bR^d$ valued random variables and $K:\bR^d \setminus\{0\}\rightarrow \bR^d$ a possibly singular interaction potential. If $K$ is ill-defined at the origin then it is either necessary to modify the sum to exclude the diagonal $i=j$ or work with an altered kernel defined to be zero at zero. In the rest of the text we will leave these modifications implicit since our regularised kernels will always be well defined and finite at the origin. When $K$ is Lipschitz continuous the classical existence and uniqueness results for ordinary or stochastic differential equations ($O/SDEs$) apply and global well-posedness of the ensemble $(X^i)_{i=1}^N$ is ensured. In this case it can also be established (see \cite{sznitman_89}) that the particle system converges in law to the solution of a McKean--Vlasov equation,
\begin{equation}\label{eq:IntroGenMKV2}
	\begin{cases}
		\dd X_t = (K \ast \mu_t)(X_t) \,\dd t +  \dd B_t,&\\
		\mu_t = \cL(X_t),&\\
		X_0 \sim \cL(\xi^1).
	\end{cases}
\end{equation}
In \eqref{eq:IntroGenMKV2} the complexity of $N$ equations has been replaced by a single, non-linear equation. In many cases this simplification is advantageous for both computation and analysis. It is therefore natural to ask whether the same approximation result holds for $K$ less than Lipschitz. This question has been addressed in a number of specific cases, cf. \cite{fournier_hauray_mischler_14,duerinckx_16,holding_16,jabin_wang_18,bresch_jabin_wang_19,serfaty_20}, we refer to \cite{bossy_05_particle,chaintron_diez_21_review} and the references therein for a general survey of such results. In this article we address this question in a pathwise regularisation by noise context, inspired by the recent works, \cite{catellier_gubinelli_16,harang_20_cinfinity, galeati_gubinelli_22_noiseless,galeati_harang_20} concerning pathwise regularisation for O/SDEs. Instead of treating a particular example we consider a wide class of problems, making very general assumptions on $K$, $B$ and $\xi$ and show that the same regularisation by noise approach can be applied across this class.\\ \par
More precisely, we show that given a $K$ in any Besov space and $\xi^{(N)}:=(\xi^i)_{i=1}^N$, $B^{(N)}:=(B^i)_{i=1}^N$ suitable $\bR^d$ and $C_T$ valued random variables, there exists a family of irregular paths, $t\mapsto Z_t\in C([0,\infty);\bR^d)$ such that the new particle system
\begin{equation}\label{eq:IntroRegularisedPartSystem}
  \begin{cases}
  \displaystyle
  \dd Y_t^i = \frac{1}{N}\sum_{j=1}^N K(Y^i_t-Y^j_t+Z_t)\,\dd t + \dd B_t^i,\\
  Y^i_0 = \xi^i,
  \end{cases}
  \quad \text{ for }i=1,\ldots,N,
\end{equation}
enjoys the same well-posedness and mean field convergence properties as hold for \eqref{eq:IntroPartSystem} when $K$ is Lipschitz. We emphasise that regarding the inputs $(\xi^{(N)}_0,B^{(N)})$ we only require that their empirical measure converges in the Wasserstein sense to $(\xi,B)$. The mean field limit of \eqref{eq:IntroRegularisedPartSystem} that we establish takes the form of a perturbed McKean--Vlasov problem.
  \begin{equation}\label{eq:IntroYEquation2}
	\begin{cases}
		\dd Y_t= \left(K(\,\cdot\,+Z_t)\ast \mu_t\right)(Y_t)\,\dd t+ \dd B_t,
		\\
		\mu_t=\cL(Y_t)
		\\
		Y_0 \sim \cL(\xi^1).
	\end{cases}
\end{equation}
For fixed $N>0$, and any $i=1,\ldots,N$, if we let $X^i:=Y^i+Z$, with $(Y^i)_{i=1}^N$ solving \eqref{eq:IntroRegularisedPartSystem}, formally we see that
\begin{equation}\label{eq:IntroRegularisedPartSystemX}
  \begin{cases}
  \displaystyle
  \dd X^i_t = \frac{1}{N}\sum_{j=1}^N K(X^i_t-X^j_t+Z_t)\,\dd t + \dd B^i_t+\dd Z_t,\\
  X^i_0 = \xi^i,
  \end{cases}
  \quad \text{ for }i=1,\ldots,N
\end{equation}
and the process $X:= Y+Z$, with $Y$ solving \eqref{eq:IntroYEquation2}, formally solves
\begin{equation}\label{eq:IntroXEquation2}
	\begin{cases}
		\dd X_t= \left(K \ast \mu_t\right)(X_t)\,\dd t+ \dd B_t+ \dd Z_t,
		\\
		\mu_t=\cL(X_t-Z_t)
		\\
		X_0 \sim \cL(\xi^1).
	\end{cases}
\end{equation}
Comparing \eqref{eq:IntroXEquation2} with \eqref{eq:IntroGenMKV2} we see that we have introduced both an additive noise into the dynamics and a shift inside the law. Considering the particle system, \eqref{eq:IntroRegularisedPartSystemX} we see that this shift in the law is necessary if we hope to obtain a regularising effect due to the additive noise. Without shifting the law, the pairwise interaction becomes $K(X_t^i-X_t^j) = K(Y^i_t-Y^j_t$), which is no better behaved. This relates to the well known idea that additive common noise cannot have a regularising effect on singular particle systems, see \cite{delarue_flandoli_vincenzi_13} for a discussion in one dimension. Since in this work we are primarily concerned with regularising the particle system and proving its mean field limit we focus on the formulation \eqref{eq:IntroRegularisedPartSystem}, \eqref{eq:IntroYEquation2}.\\ \par
The idea that noise can regularise certain ill-posed dynamics stems back to the early works of A. K. Zvonkin and A. J. Veretennikov, \cite{Zvonkin_74, veretennikov_81} which employed PDE theory and It\^o calculus to obtain strong well-posedness, in a probabilistic sense for the SDE
\begin{equation}\label{eq:RegularisedSDE}
	\dd X_t = f(X_t)\,\dd t + \dd B_t,
\end{equation}
with $f \in L^\infty$. This is in contrast to the Cauchy--Lipschitz theory for ODEs which requires $f\in W^{1,\infty}$ to ensure well-posedness. Since these early works the subject has received a much attention, with authors exploring similar results for time dependent drift, \cite{krylov_rockner_05}, methods based on Malliavin calculus, \cite{meyerbrandis_proske_10} and considering more general noise terms, \cite{banos_nilssen_proske_19}. In \cite{davie_07}, A. M. Davie presented a similar result in a pathwise setting, which in contrast to the other works mentioned does not involve averaging over the random perturbation. Davie established that for, \eqref{eq:RegularisedSDE} with $f\in L^\infty$ and almost every path $B$ of the Brownian motion, there exists a unique continuous path $X \in C([0,T];\bR^d)$ satisfying \eqref{eq:RegularisedSDE} in an integral sense. This notion of pathwise regularisation has recently been extensively developed, cf. \cite{catellier_gubinelli_16, harang_20_cinfinity,galeati_gubinelli_22_noiseless, galeati_gubinelli_20_Prevalence, galeati_harang_20, gerencser_20} for a selection of works. Our main result is presented from a pathwise perspective, and we give more background to the approach in Section \ref{sec:PathwiseRegularisation}. For now we only give a brief intuition behind the  regularising effect in our context.\\ \par 
Consider the particle system, \eqref{eq:IntroRegularisedPartSystem} and let us assume that $K$ is singular at the origin. Then, if $t\mapsto Z_t$ is a $C^1$ path we have $K(Y^i_t-Y^j_t+Z_t) \rightarrow \infty$ as $Y^i_t-Y^j_t +Z_t\rightarrow 0$. So the singularity in the equation remains. However, if $t\mapsto Z_t$ is a random path, whose trajectories oscillate sufficiently fast then we may imagine that the event $|Y^i_t-Y^j_t +Z_t|<\varepsilon $ occurs with a very small probability. If this probability decreases sufficiently fast as $\varepsilon\rightarrow 0$, compared with the blow-up rate of $K$ near zero, then we may hope to show that the drift terms of \eqref{eq:IntroRegularisedPartSystem} are suitably bounded. The theory of averaged fields and non-linear Young equations makes this idea rigorous and extends it to more general regularising process. This is discussed in Section \ref{sec:PathwiseRegularisation}.\\ \par
Regularisation by noise results for interacting particle systems and McKean--Vlasov equations have been obtained in a number of specific settings. In \cite{flandoli_gubinelli_priola_11} it is shown that the vortex dynamics associated to the Euler equation in two dimensions are globally well-posed when perturbed by suitable \textit{turbulent} noise, as opposed to the well-posedness from almost every initial data obtained by C. Marchioro, M. Pulvirenti, cf. \cite{marchioro_pulvirenti_94}. A selection by noise result for a system of Vlasov point charges in one dimensions was obtained in \cite{delarue_flandoli_vincenzi_13}. Similarly the works, \cite{duerinckx_16,holding_16,fournier_jourdain_17,chaudru_frikha_21_backward,chaudru_frikha_22_wasserstein_PDE,lacker_21_hierarchies,han_22_entropic} can be seen as regularisation by noise phenomena, since the same results are not known to hold for the same systems without Brownian noise. Finally we mention the work of V. Marx \cite{marx_20} which is in a very different direction, concerning regularisation by noise for McKean--Vlasov equations viewed as differential equations on the space of probability measures perturbed by the Wasserstein diffusion. 
To the best of our knowledge the result presented here is the first to consider a pathwise regularisation of interacting particle systems in the sense of Davie with corresponding convergence to the associated McKean--Vlasov problem. We mention also two very recent articles by L. Galeatti and both authors of the current article \cite{GalHarMay_22_benchmark,GalHarMay_22fBm} where a mixture of pathwise and probabilistic techniques are used to prove existence and uniqueness of general distribution dependent SDEs with singular interactions and driven by an additive fractional Brownian motion. However, the question of interacting particle systems and mean field approximation is not considered there. Finally, we mention that a pathwise approach to interacting particle systems and McKean--Vlasov equations has recently been considered in the setting of rough differential equations, \cite{cass_lyons_14,bailleul_catellier_delarue_20_solving,bailleul_catellier_delarue_21_propagation,coghi_nilssen_21_RoughNonlocal}.
\\ \par
Methodologically we focus primarily on studying \eqref{eq:IntroYEquation2}. We handle the random dynamics coming from $B_t$ in a pathwise setting, using the framework of \cite{coghi_deuschel_friz_maurelli_20} and the trick of Tanaka, \cite{tanaka_84}, also used in that paper, to cast the mean field approximation as a stability result for \eqref{eq:IntroYEquation2}. We highlight this is possible since we only consider $B^{(N)}$ and $B$ as additive perturbations. We handle the drift term, $(K(\,\cdot\,+Z_t )\ast \mu_t)(Y_t)$, as a non-linear Young integral with measure flow convolutions. We present this extension in Section \ref{subsec:NLYI}. To summarise our main contributions, we make very few assumptions on $K$, $B$ and $\xi$ and the approximating sequences $B^{(N)},\,\xi^{(N)}$. We may in fact allow $\xi$ and/or $B$ to be deterministic, in particular we cover the case $B^{(N)}\equiv B\equiv 0$. This generality is compensated by restricting our setting to including the $B$ terms additively and imposing restrictions on the regularising path $Z$. Still, for given data $(K,B)$, the class of applicable paths $Z$ is vast, with precise assumptions described in the next section. With regards the topic of non-linear Young integration we believe the extension to integrands involving measure flows is of independent interest.
%
%
\subsection{Main results}\label{sec:PWiseMainResults}
Before stating our main results, we give a rigorous definition of solution to the perturbed McKean--Vlasov equation \eqref{eq:IntroYEquation2}.  An explanation of the necessary notations used below can be found in Subsection \ref{subsec:Notation}.\\ \par
We fix $T>0$, $(\gamma,\eta)\in (1/2,1)\times (0,\infty)$ such that 
\begin{equation}\label{eq:EtaGammaMainResults}
    (\eta \wedge \gamma)+ \gamma>1
\end{equation}
and $(\Omega,\cF,\PP)$ an abstract probability space. Although \eqref{eq:EtaGammaMainResults} is equivalent to the requirement $\eta +\gamma >1$ for $(\gamma,\eta)\in (1/2,1)\times (0,\infty)$ we chose to state the constraint in this form as it makes clear the role that $\eta \wedge \gamma$ plays in our analysis. For a distribution $K\in \cS'(\bR^d;\bR^d)$ and $Z:C([0,T];\bR^d)$ we define the averaged field $\Gamma K:[0,T]\times \bR^d \rightarrow \bR^d$ formally by 
\begin{equation}\label{Formal avg field}
  \Gamma_t K(x)=\int_0^tK(x+Z_r)\dd r, 
\end{equation}
and we write $\Gamma_{s,t}K(x)=\Gamma_tK(x)-\Gamma_sK(x)$. It is this object that will play a central role in the regularization procedure, and a rigorous discussion of this object is given in Section \ref{sec:PathwiseRegularisation}.
\begin{definition}\label{def:solution}
	Let $q,p\geq 1$, $(\xi,B)\in L^q(\Omega;\bR^d)\times L^{p}(\Omega;\cC_T^\eta)$, with $B$ a zero at zero path and assume we are given $K\in \cS^\prime(\bR^d;\bR^d)$ and zero at zero path $Z\in C_T$ such that the associated averaged field, $\Gamma K$ (see \eqref{Formal avg field} and Def. \ref{def: avg operator}), is contained in $\cC^\gamma_T\cC^2(\bR^d;\bR^d)$. Then we say that $Y$ is a solution to \eqref{eq:IntroYEquation2} if  $Y\in L^{q\wedge p}(\Omega;\cC^{\gamma\wedge \eta}_T)$ and $Y$ solves the generalised McKean--Vlasov equation 
	\begin{equation}\label{eq:IntroNLYEquation}
		\begin{cases}
			\displaystyle Y_t=\xi+\int_0^t\left(\Gamma_{\dd r} K\ast \mu_r\right)(Y_r)+B_t,
			\\
			\mu_r=\cL(Y_r).
		\end{cases}
	\end{equation}
	where the drift term is properly defined as a measure dependent non-linear Young integral in Section \ref{subsec:NLYI}, see Lemma \ref{lem:YoungIntegration} for example and in particular for an explanation of the notation $\Gamma_{\dd s}K$.
\end{definition}
	Proofs of Theorems \ref{th:IntroGeneralMainTheorem} and \ref{th:IntroMeanFieldApprox} and Corollary \ref{cor:RegularizeClassicalMKV} stated below are completed in Section \ref{sec:MainResultsProofs}.
\begin{theorem}\label{th:IntroGeneralMainTheorem}
Let $\sigma\in \bR$, $q,r \in [1,\infty]$. Assume we are given a pair $(\xi,B)\in L^1(\Omega;\bR^d) \times L^{p}(\Omega;\cC^\eta_T)$ for all $p\geq 1$, such that $B_0=0$, a $K\in \cB^\sigma_{r,q}(\bR^d;\bR^d)$ and a zero at zero path $Z \in C_T$ such that the associated averaged field $\Gamma K\in \cC_T^\gamma\cC^2$. Then there exists a unique solution $Y\in L^1(\Omega;\cC^{\eta\wedge \gamma}_T)$ to the equation
	\begin{equation}\label{eq:intro thm equation}
		\begin{cases}
			\dd Y_t = (K(\cdot +Z_t) \ast \mu_t)(Y_t) \,\dd t + \dd B_t,&\\
			\mu_t=\cL(Y_t),&\\
			Y_0 = \xi,
		\end{cases}
	\end{equation}
	in sense of Definition \ref{def:solution}. %
	Furthermore, if $(\xi^1,B^1),\,(\xi^2,B^2) \in L^1(\Omega;\bR^d)\times L^{p}(\Omega;\cC^\eta_T)$ for all $p\geq 2$ with $(\xi^1-\xi^2)\perp (B^1,B^2)$ and $Y^1,\,Y^2 \in L^1(\Omega;\cC^{\eta\wedge \gamma}_T)$ are the corresponding solutions to \eqref{eq:intro thm equation}, then there exists a constant $C=C(T,\Gamma,\gamma,\eta)>0$ such that,
	\begin{equation}\label{eq:IntroYProcessStable}
		\cW_{1;\cC^{\eta \wedge \gamma}_T}(\cL(Y^1),\cL(Y^2)) \leq C \left(\cW_{1;\bR^d }(\cL(\xi^1),\cL(\xi^2))+ \cW_{2; \cC^{\eta}_T}(\cL(B^1),\cL(B^2))\right).
	\end{equation}
\end{theorem}
There exist a large number of continuous paths $Z$ providing the regularizing effect required by the above theorem, see \cite{harang_20_cinfinity,galeati_gubinelli_20_Prevalence,harang_ling_21_VoltLevy}. We will illustrate an example of a class of random paths with exactly such regularizing properties in Corollary \ref{cor:RegularizeClassicalMKV}. With the above general theorem at hand we turn to some specific applications. First we describe the implications of Theorem \ref{th:IntroGeneralMainTheorem} in the context of mean field approximations to \eqref{eq:intro thm equation}. We use the trick of Tanaka, \cite{tanaka_84}, and the fact that the noise terms $B_t^i$ are additive to obtain Theorem \ref{th:IntroMeanFieldApprox}.
\begin{theorem}[Mean Field Approximation]\label{th:IntroMeanFieldApprox}
Let $N\in \bN$ and $(\xi,B)\in L^1(\Omega;\bR^d)\times L^{p}(\Omega; \cC^\eta_T)$, $(\xi^{(N)}_0,B^{(N)})\in L^1(\Omega;\bR^{Nd})\times L^{p}(\Omega;(\cC^\eta_T)^N)$ for all $p\geq 1$ be such that $B_0=0$ and $B^{(N)}_0 =0$. In addition assume that for some $\tilde{p}\geq 1$
	\begin{equation*}
		\lim_{N\rightarrow \infty} \left(\cW_{1;\bR^d}\left(\cL_N(\xi^{(N)}(\omega)),\cL(\xi)\right) + \cW_{\tilde{p};\cC^{\eta}_T}\left(\cL_N(B^{(N)}(\omega)),\cL(B)\right) \right) = 0,\quad \PP\text{-a.s.}
	\end{equation*}
  Then with $Z$ and $\Gamma$ as in Theorem \ref{th:IntroGeneralMainTheorem}, for any $N\geq 1$ there exists a unique solution $Y^{(N)}:=(Y^{i})_{i=1}^N \in L^1(\Omega;(\cC^{\eta \wedge \gamma}_T)^N)$ to the particle system
\begin{equation}\label{eq:IntroNLYParticleEquation}
Y^i_t=\xi^i+\frac{1}{N}\sum_{j=1}^N\int_0^t \Gamma_{\dd r}K(Y_r^i-Y^j_r)\,+B^i_t,\quad \text{ for } i=1,\ldots,N.
\end{equation}
Furthermore, letting $Y$ be the unique solution to \eqref{eq:IntroNLYEquation} and assuming $\xi\perp B$,
\begin{equation*}
\lim_{N\rightarrow \infty}\cW_{1;\cC^{\eta \wedge \gamma}_T}\left(\cL_N(Y^{(N)}(\omega)),\cL(Y)\right) =0 \quad \PP\text{-a.s.}
\end{equation*}
\end{theorem}
\begin{remark}
    The assumption of convergence in Wasserstein of the data $(\xi^{(N)},B^{(N)})$ is satisfied if for example the empirical measure converges weakly $\PP$-a.s. and sufficiently high moments are uniformly integrable. In particular, if the sequences $(\xi^{i},B^i)_{i=1}^N$ are i.i.d then one obtains the required convergence. See \cite[Lem. 53 \& 54]{coghi_deuschel_friz_maurelli_20} for details.
\end{remark}
\begin{remark}
It is clear from the proofs of Theorems~\ref{th:IntroGeneralMainTheorem} \& \ref{th:IntroMeanFieldApprox} that if in addition one assumes the initial data to have finite $q$ moments for some $q>1$ then the respective solutions to the McKean--Vlasov equation and particle system will have finite $q$ moments in $\cC^{\eta\wedge \gamma}_T$. However, so as not to unnecessarily complicate the notation further we restrict ourselves to $L^1$ initial data. Furthermore, from the steps of the proofs one can see that we do not necessarily require all moments of $B$ to be finite, but only $p$ moments for $p$ that can be arbitrarily large depending on some parameters chosen in the proofs. Informally, there is a trade off between the regularising effect coming from $Z$ and the number of finite moments we require $B$ to have. Since we are mostly interested in leveraging the maximum possible regularising effect for simplicity we ask for $B$ to have all finite moments. For more details see Remark~\ref{rem:MomentRem2}.
 \end{remark}
Since Theorems \ref{th:IntroGeneralMainTheorem} and \ref{th:IntroMeanFieldApprox} are both quite general in nature, we specify our results to a particular class of McKean--Vlasov equations, that includes many physically relevant models. Recall that a distribution $K \in \cS^\prime(\bR^d;\bR^d)$ is said to be homogeneous of degree $\sigma \in \bR$ if for any $\psi \in \cS(\bR^d;\bR^d)$, defining $\psi_\lambda = \frac{1}{\lambda^d}\psi\left(\,\cdot\,/\lambda\right)$ one has
\begin{equation}
  \langle K,\psi_\lambda\rangle = \lambda^{\sigma}\langle K,\psi\rangle.
\end{equation}
A fuller discussion of such distributions is given in Appendix \ref{app:BesovSpaces}. To further simplify the setting, we only consider $B$ a random path taking values in $C^{1/2-}_T$ and we exhibit as an example a range of regularising paths $Z$ for which our results apply. We emphasise that Theorems \ref{th:IntroGeneralMainTheorem} and \ref{th:IntroMeanFieldApprox} involve more general assumptions.
\begin{corollary}\label{cor:RegularizeClassicalMKV}
Let $\sigma<0$, $\xi \in L^1(\Omega;\bR^d)$ $K$ be a homogeneous kernel of degree $\sigma$ and $B\in L^{p}(\Omega;\cC^{1/2-}_T)$ for all $p\geq 1$. Then let $(Z_t)_{t\in [0,T]}$ be a fractional Brownian motion with Hurst parameter $H\in (0,1)$ on a possibly different probability space $(\tilde{\Omega},\tilde{\cF},\tilde{\PP})$. If $H<\frac{1}{4-2\sigma}$, then there exists a set $\tilde{\cN}\subset \tilde{\Omega}$ of full measure, such that for all $\tilde{\omega}\in \tilde{\cN}$, $Z=Z(\tilde{\omega}):[0,T]\rightarrow \bR^d$ is a continuous path and under suitable remaining assumptions, the results of Theorems \ref{th:IntroGeneralMainTheorem} and \ref{th:IntroMeanFieldApprox} apply.
\end{corollary}
\begin{remark}
    Note that the assumption $B\in L^{p}(\Omega;\cC^{1/2-}_T)$ for all $p\geq 1$ covers a wide variety of stochastic processes, including but not limited to, the Brownian motion.
\end{remark}
\begin{remark}
	Regarding the threshold on the Hurst parameter, we note that in contrast to the results obtained in \cite{harang_20_cinfinity}, for a general distribution $f \in \cB^{\sigma}_{2,2}(\bR^d;\bR^d)$ which requires a Hurst parameter $H<\frac{1}{4+2d-2\sigma }$, the threshold of $H<\frac{1}{4-2\sigma}$ in Corollary~\ref{cor:RegularizeClassicalMKV} is dimension independent. This is due to the fact that $K$ is assumed to be a homogeneous distribution; see Definition \ref{eq:HomogeneousDefinition} and therefore one has $K\in \cB^{\sigma+d/p}_{p,\infty}(\bR^d;\bR^d)$ for any $d,\,p\geq 1$. See Appendix \ref{subsec:HomogeneousDist} for details. 
\end{remark}
\subsection{Structure and Outline}
In the remainder of this section we detail the additional notation used in this article. In Section~\ref{sec:PathwiseRegularisation} we recap some ideas from the theory of pathwise regularisation, averaged fields and non-linear Young integration. Readers familiar with this material should feel free to skip this section. Appendix \ref{app:BesovSpaces} contains a review of Besov spaces and homogeneous distributions. In Section~\ref{sec:MeasureFlowsRegularity} we recall the definitions and some properties of the Wasserstein distances on the space of probability measures. In Subsection~\ref{subsec:HolderMeasureFlow} we define a notion of H\"older continuity for time dependent measure valued flows which is a central tool in proving our results. The proofs of our three main results stated in Section \ref{sec:PWiseMainResults} are conducted in two major steps. Firstly we consider the abstract non-linear Young equation \eqref{eq:AbstractNLYEMKV} and then in Sections \ref{sec:AbstractMKVResults} and \ref{sec:AbstractMeanFieldLimit} we prove analogues of Theorems \ref{th:IntroGeneralMainTheorem} and \ref{th:IntroMeanFieldApprox} for the abstract equation, which makes no reference to $K$ or $Z$. Section \ref{sec:MainResultsProofs} contains the proof of our main results, in which we relate the abstract theorems proved for \eqref{eq:AbstractNLYEMKV} to the perturbed McKean--Vlasov problem \eqref{eq:intro thm equation}. Finally in Section \ref{sec:Applications} we discuss some specific models to which our results apply.
\subsection{Notations}\label{subsec:Notation}
Throughout the text we write $\bR$ for the real line, $\bR_+$ for the positive half line and $\bR^d$ for $d$-dimensional Euclidean space. We will often fix a $T \in \bR_+$ and work on intervals of the form $[0,T]\subset \bR_+$. We use the symbol $\lesssim$ 
to denote that the quantity on the left is less than or equal to 
the quantity on the right up to a constant that is either unimportant or depends on quantities that are fixed in the context. When we wish to make the dependence on these quantities more specific we either write $\lesssim_{d,\kappa}$ or $\leq C(K,p,d)$. These constants will always be allowed to change from line to line without further indication.

For mappings $f :\bR^d \rightarrow \bR$ we write $\nabla f = (\partial_1 f,\ldots, \partial_d f) \in \bR^d$ for the gradient and for mappings $g:\bR^d\rightarrow \bR^d$ we write $\nabla \cdot g = \sum_{i=1}^d \partial_ i g$ for the sum of partial derivatives which defines the divergence. For vectors $k \in \bN^d$ we use standard multi-index notation and for $a \geq 0$ and $n\in \bN_{\geq 1}$ we write $C^a(\bR^d;\bR^n)$ for the set of all functions $f:\bR^d\rightarrow \bR^n$ such that $D^{k}f := \prod_{i=1}^d\partial_{k_i} f$ is continuous for all $|k|\leq a$. When $a=\infty$ we write $C^\infty(\bR^d;\bR^n) = \cup_{a \geq 0} C^a(\bR^d;\bR^n)$. We write $C^a_c(\bR^d;\bR^n)$ for the set of $C^a$ functions with compact support on $\bR^d$ and $C^a_b(\bR^d;\bR^n)$ for the set of bounded $C^a$ functions. We write $\cS(\bR^d;\bR^n)$ for the space of Schwarz functions $f:\bR^d \rightarrow \bR^n$ such that for any $m \geq 0$ and $a\geq 0$ we have
	\begin{equation*}
		\sup_{|k|=a}\sup_{x \in \bR^d} |x^mD^kf(x)|<\infty.
	\end{equation*}
This equips the space $\cS(\bR^d;\bR^n)$ with the structure of a Fr\'echet space and we write $\cS^\prime(\bR^d;\bR^n)$ for its dual, the space of tempered distributions. From now on we leave the range of such function spaces implicit and when the context is clear we will remove the explicit dependence on the domain as well.

For $p\geq 1$, we write $L^p(\bR^d)$ for the usual spaces of $p$-integrable real functions and when $p=\infty$, the space of essentially bounded functions. For $a,\,p\geq 1,$ we write $W^{a,p}(\bR^d)$ for the space of functions with $p$-integrable weak derivatives up to order $a$. For $p \in [1,\infty)$ we write $\ell^p$ for the set of sequences $\left(f_m\right)_{m=1}^\infty$ such that $\sum_{m=1}^\infty |f_m|^p <\infty$ and $\sup_{m\geq 1}|f_m|<\infty$ for $p=\infty$. We work predominantly in the scale of Besov spaces, which we denote by $\cB^{\alpha}_{p,q}(\bR^d)$ for $\alpha \in \bR$ and $p,q \in [1,\infty)$. When $p=q=\infty$ and $\alpha \in \bR\setminus \bN$, we write simplify by writing $\cC^\alpha(\bR^d)$ for the H\"older--Besov spaces and when $p=q=2$ we write $\cH^\alpha(\bR^d)$. We define these spaces properly in Appendix \ref{app:BesovSpaces}.

For $E$ a Banach space we write $\cC^\alpha(\bR_+;E)$ with $\alpha\in (0,1)$ for the space of $\alpha$-H\"older maps $X:\bR_+\rightarrow E$. For $T>0$ and maps $X:[0,T]\rightarrow E$ we write $\cC^\alpha_{T}E:=\cC^\alpha([0,T];E)$ for these spaces and $\cC^\alpha_{[s,t]}E$ for maps $X:[s,t]\rightarrow E$ with $[s,t]\subset \bR_+$. When $E=\bR^d$ we simply write $\cC^\alpha_T$ (resp. $\cC^\alpha_{[s,t]}$). For a mapping $X:[0,T]\rightarrow E$ and any $0\leq s <t\leq T$ we write $X_{s,t}:=X_t-X_s$ to denote the increment. With a slight abuse of notation, we will also denote by $X_{s,t}$ a two-parameter function $X:[0,T]^2\rightarrow E$. For $n\in \bN$ and $0\leq s<t<\infty$ we define the $n$-simplex $\Delta_n^{[s,t]}$ by 
\begin{equation*}
	\Delta_n^{[s,t]}:=\{(r_1,\ldots,r_n)\in [s,t]^n\,|\,r_1\leq\dots\leq r_n\}. 
\end{equation*}
Then for $X:[s,t]^2\rightarrow E$, when we say that $[X]_{\alpha;[s,t]}<\infty$, we mean that 
	\begin{equation*}
		[X]_{\alpha;[s,t]}:= \sup_{(u,v)\in \Delta_2^{[s,t]}} \frac{\|X_{u,v}\|_E}{|v-u|^\alpha}<\infty.  
\end{equation*}
In the case when $X$ is a one parameter path we use the shorthand notation $X_{u,v}:=X_v-X_u$.
To convert this into a proper norm we add the value of $X$ at $u=s$, defining $\|X\|_{\alpha;[s,t]} = \|X_s\|_E+ [X]_{\alpha;[s,t]}$. When $s=0$ we will simply write $\|\,\cdot\,\|_{\alpha;t}$ or $[\,\cdot\,]_{\alpha;t}$. When $X_0=0$ by convention we always measure the path in $\|\,\cdot\,\|_{\alpha}$. For $\alpha>0$ we define $\cC^{\alpha-}_T:= \cap_{\alpha'<\alpha}\cC^{\alpha'}_T$. When $\alpha=0$ we write $C_T$ (resp. $C_{[s,t]}$) for the space of continuous mappings $[0,T]\rightarrow \bR^d$ (resp. $[s,t]\rightarrow \bR^d$). 
 For a space time function $\Gamma:[0,T]\times \bR^d\rightarrow \bR^d$, when the context is clear, for $\alpha,\gamma \in \bR$, we use the notation
\begin{equation*}
	\|\Gamma\|_{\gamma,\alpha}:= \|\Gamma\|_{\cC^\gamma_T\cC^\alpha(\bR^d)}.
\end{equation*}
For $(E,\cE)$ a Hausdorff topological space which we always equip with its Borel sigma algebra we let $\cM(E)$ denote the set of real valued, signed Radon measures on $E$ and we write $\cP(E)$ for the set of probability measures on $E$. For $\mu \in \cM(E)$ there exists a pair of measures $(\mu^+,\mu^-)$ such that at least one is finite, they have disjoint support and $\mu^+(A)\geq 0$ and $\mu^-(A)\leq 0$ for any Borel set $A \subseteq E$. Then we define the total variation of $\mu \in \cM(E)$ by $|\mu|:=\mu^+-\mu^-$. Given an abstract probability space $(\Omega,\cF,\PP)$ and a measurable map $X:\Omega \rightarrow E$ we write $\cL(X)$ to designate the law of $X$ which is a probability measure on $E$ such that the identity
\begin{equation*}
	\EE\left[ f(X)\right] = \int_E f \, \dd \cL(X)
\end{equation*}
holds for all $f:E\rightarrow \bR$ continuous and bounded. Given two Banach spaces $E,F$, a Borel measurable mapping $\pi :E\rightarrow F$ and a measure $\mu \in \cM(E)$, we define the push-forward of $\mu$ by $\pi$ to be $\pi \# \mu:= \mu(\pi^{-1}\,\cdot\,)\in \cM(F)$. For any measurable mapping $g :F\rightarrow \bR$ such that $ g \circ \pi:E \rightarrow \bR$ is $\dd\mu$ integrable then the push-forward satisfies
\begin{equation*}
	\int_{F} g \,\dd \pi \#\mu = \int_E g\circ \pi \,\dd \mu.
\end{equation*}
We give more detailed definitions and discussions of some of these notions and related concepts in Section \ref{sec:MeasureFlowsRegularity}. Finally for any $p\geq 1$, and $E$ a Banach space, we write $L^p(\Omega;E)$ for the set of measurable maps $X:\Omega \rightarrow E$ such that $\EE\left[\|X\|_E^p\right]<\infty$.
We write $\cW_{p,E}(\mu,\nu)$ for the $p$-Wasserstein distance between two probability measures $\mu,\,\nu$ in $\cP_p(E)$, the space of probability measures with $p$-finite moments. We give a detailed definition of these distances and discussion of their properties in Section \ref{sec:MeasureFlowsRegularity}. For three random variables $A,\,B,\,C$ if we write $A \perp (B,C)$ we mean that $A$ is independent from $B$ and from $C$.
\section{Averaged Fields and Pathwise regularisation of ODEs}\label{sec:PathwiseRegularisation}
Let $f \in \cB_{p,q}^\beta(\bR^d)$ for $\beta \in \bR$, $p,q\in [1,\infty]$ and $Z:[0,T]\rightarrow \bR^d$ be a possibly random path on an abstract probability space $(\tilde{\Omega},\tilde{\cF},\tilde{\PP})$. Then consider the formal, integral equation,
\begin{equation}\label{eq:SimpleODE}
	\tilde{X}_t^{\xi}=\xi+\int_0^t f(\tilde{X}_r^{\xi})\dd r+Z_t. 
\end{equation}
In \cite{catellier_gubinelli_16}, the authors show that if $Z$ is sufficiently irregular (exact meaning to be explained later) then \eqref{eq:SimpleODE} can be interpreted rigorously and is pathwise well-posed, even when $f$ is only a distribution. More specifically the authors show that if $Z$ is a fractional Brownian motion with Hurst parameter $H\in (0,1)$ then for a given $f \in \cC^{-\frac{1}{2H}+2}$ there exists a full measure set $\tilde{\cN} \in \tilde{\Omega}$ (depending on $f$) such that for all $\tilde{\omega} \in \tilde{\cN}$ there exists a unique solution to \eqref{eq:SimpleODE} driven by $Z(\tilde{\omega})$. This result has recently been developed further in \cite{harang_20_cinfinity,galeati_gubinelli_22_noiseless}. In this section we will give a short introduction to the methodology and ideas behind such regularisation of ordinary differential equations, which will in subsequent sections be applied to McKean--Vlasov problems. 

The first step is to reformulate \eqref{eq:SimpleODE} by defining $X_t:=\tilde{X}_t-Z_t$, which we ask to solve,
\begin{equation}\label{eq:ModifiedODE}
	X_t^{\xi}=\xi+\int_0^t f(X_r^{\xi}+Z_r)\dd r,
\end{equation}
where again the drift is only to be understood formally for now. We then define a new distribution, for any $0\leq s<t\leq T$ and $x\in \bR^d$, setting
\begin{equation}\label{eq:GammaDef}
	\langle\Gamma_{t}f,\varphi(\,\cdot\,-x)\rangle := \int_0^t \langle f,\varphi(\,\cdot\,-x-Z_r)\rangle\,\dd r,\quad.
\end{equation}
for all $\varphi \in \cS(\bR^d)$. We refer to $\Gamma_{t}f$ as an averaged distribution and denote the time increment by $\Gamma_{s,t}f=\Gamma_{t}f-\Gamma_{s}f$.
\begin{definition}[Averaged distributions]\label{def: avg operator}
	For $\beta\in \bR$, $p,q\in [1,\infty]$, let $f\in \cB_{p,q}^\beta(\bR^d)$. We say that $\Gamma f:[0,T]\times \bR^d \rightarrow \bR^d $ defined by \eqref{eq:GammaDef}, is an {\it averaged distribution} if $ t \mapsto \Gamma_{t} f\in \cC^\gamma_{T} \cC^\alpha(\bR^d)$ for some $\gamma>1/2$ and $\alpha\geq \beta$. If $\Gamma f \in C_T^\gamma\cC^{\alpha}(\bR^d)$ for all $f\in \cB_{p,q}^\beta(\bR^d)$ then by an abuse of notation we also define the averaging operator $\Gamma :\cB^\beta_{p,q}(\bR^d)\rightarrow \cC^\gamma_T\cC^\alpha(\bR^d)$, where the evaluation is given by \eqref{eq:GammaDef}.
\end{definition}
\begin{remark}
	Note that $\Gamma$ in general depends on a given path $Z:[0,T]\rightarrow \bR^d$, however in this article we are not concerned with the properties of $\Gamma$ w.r.t $Z$, we only assume we can build a sufficiently regular averaged distribution from some set of paths $Z$. Therefore we only write $\Gamma$, and say that $\Gamma$ is associated to the path $Z$ when necessary.
\end{remark}
\begin{definition}\label{def: regularising path}
	Let $\beta\in \bR$, $p,q\in [1,\infty]$ and $\rho>0$. A path $Z:[0,T]\rightarrow \bR^d$ is called {\it $\rho$-regularising on $\cB^\beta_{p,q}(\bR^d)$} if there exists a $\gamma>1/2$ such that the averaging operator $\Gamma$ associated to $Z$ satisfies  $\Gamma f \in \cC_T^\gamma \cC^{\beta+\rho}(\bR^d)$ for every $f\in \cB^\beta_{p,q}(\bR^d)$. If the path $Z$ is such that for any $f\in \cB_{p,q}^\beta(\bR^d)$ with $\beta\in \bR$ and $p,q\in [1,\infty]$, the averaged field $\Gamma f\in \cC^\gamma_T\cC^\alpha(\bR^d)$, for any $\alpha\in \bR$, we say that $Z$ is {\em infinitely regularising}. 
\end{definition}
\begin{remark}\label{rem:interpolation}
Any continuous path $(Z_t)_{t\in [0,T]}$ is $0$-regularising, in the sense that for any function $f\in \cC^\beta$ with $\beta\in \bR$, it follows that $\Gamma f\in \cC^1_T\cC^\beta(\bR^d)$. With this knowledge, interpolation reveals that time regularity of $\Gamma f$, when $\Gamma$ is associated to a $\rho$-regularising path, can be traded for spatial regularity. To see this, since $\Gamma f\in \cC^{1/2}_T\cC^{\beta+\rho}(\bR^d)\cap \cC^1_T\cC^\beta(\bR^d)$, it follows by interpolation in Besov spaces (see e.g. \cite[Thm. 2.80]{bahouri_chemin_danchin_11}) that for any $\theta\in [0,1]$
	\begin{equation*}
		\|\Gamma_{s,t}f\|_{\cC^{\beta+\theta \rho}}\leq \|\Gamma_{s,t}f\|_{\cC^{\beta+\rho}}^\theta \|\Gamma_{s,t}f\|_{\cC^{\beta}}^{1-\theta}.
	\end{equation*}
	Thus, for any $\gamma\in [\frac{1}{2},1]$ it follows that $\Gamma f\in \cC^\gamma_T \cC^{\beta+2\rho(1-\gamma)}$.
\end{remark}
From now on, we assume that given $f$ we are able to find $Z$ sufficiently regularising that $\Gamma_{s,t} f $ is a genuine function and so we may drop the test function in \eqref{eq:GammaDef}. In this case we refer to $\Gamma f$ as an \textit{averaged field}. This assumption will be justified below.

In the case that $Z$ is a random path, the averaged field can be written as the integral of $f(x+z)$ against the occupation measure, $m_t$, of the path $t\mapsto Z_t$,
\begin{equation}\label{eq:local time formula}
	\Gamma_{0,t} f(x)= \int_0^tf(x+Z_r)\dd r=\int_{\bR^d}f(x+z) \dd m_{t}(z).
\end{equation}
Assuming the occupation measure $m_t$ has a density, $L_t \in L^1(\bR^d)$, we can re-write \eqref{eq:local time formula} as a convolution,
\begin{equation}\label{eq:avg function as local time conv}
	\Gamma_{0,t}f(x)=f\ast \bar{L}_t(x) \quad \text{ where }\quad \bar{L}_t(x):= L_t(-x). 
\end{equation}
Therefore, one approach to defining the averaged field is to first obtain regularity estimates on the occupation measure $m_t$ and then define $\Gamma f$ as in \eqref{eq:avg function as local time conv}. We outline a few known results on the regularity of the averaged fields $\Gamma f$ and the regularity of the local times, $L_t$, associated to certain Gaussian processes. For a deeper discussion on occupation measures and local times, see the survey paper \cite{geman_horowitz_80}.
\begin{example}\label{example:CatGubinelli}
	Let $Z$ be a fractional Brownian motion, on $(\tilde{\Omega},\tilde{\cF},\tilde{\PP})$, with Hurst parameter $H\in (0,1)$ and $f\in \cC^\beta$. In \cite{catellier_gubinelli_16} it is shown that there exists a set $\tilde{\cN} \subset \tilde{\Omega}$ of full measure depending on the distribution $f$ and $Z$, and a $\gamma>1/2$, such that for all $\tilde{\omega}\in \tilde{\cN}$, the averaged field $\Gamma f\in \cC^\gamma_T\cC^{\beta+\rho}_{\text{loc}}(\bR^d)(\bR^d)$ for any $\rho<1/(2H)$. Note that this result does not define an average operator in sense of Definition \ref{def: avg operator}, as the set of full measure, $\tilde{\cN}$, depends explicitly on $f$, and thus the regularising effect does not necessarily hold for all $f\in \cC^\beta(\bR^d)$ simultaneously.
\end{example}
In two recent publications \cite{galeati_gubinelli_22_noiseless,galeati_gubinelli_20_Prevalence}, Galeati and Gubinelli prove that infinitely regularising paths are prevalent in $C_T$. The concept of prevalence was earlier used by Hunt, \cite{hunt_94}, to prove that almost all continuous paths are nowhere differentiable. In \cite{galeati_gubinelli_22_noiseless} it is also shown that for any $\delta>0 $, the $\frac{1}{2\delta}$-regularising paths are prevalent in $\cC^{\delta-\varepsilon}_T$ for any $\varepsilon>0$, \cite[Thm. 1]{galeati_gubinelli_22_noiseless}. This result makes rigorous the heuristic that more irregular paths $Z$ lead to more regularising averaging operators $\Gamma$.

In the next proposition we give a concrete example of a criterion that guarantees the regularising effect of a given process. This condition has been applied to obtain regularisation results in \cite[Thm. 17]{harang_20_cinfinity} and \cite{galeati_gubinelli_20_Prevalence}.
\begin{proposition}\label{prop: reg of local time}
	Let $(Z_t)_{t\in [0,T]}$ be a continuous Gaussian process, such that for some $\zeta\in (0,\frac{1}{d})$ 
	\begin{equation*}
		\inf_{t\in [0,T]}\inf_{s\in [0,t]}\inf_{z\in \bR^d;\, |z|=1} \frac{z^T{\rm Var}(Z_t|\cF_s)z}{|t-s|^{2\zeta}}>0.
	\end{equation*}
	Then there exists a $\gamma>1/2$ such that the associated local time $L:[0,T]\times \bR^d \rightarrow \bR_+$ is contained in $\cC^\gamma_T \cH^\rho(\bR^d)$ for any $\rho<\frac{1}{2\zeta}-\frac{d}{2}$, $\PP$-a.s.. 
\end{proposition}
\begin{proof}
	See \cite[Thm. 17]{harang_20_cinfinity}.
\end{proof}
If $(Z_t)_{t\in [0,T]}$ is assumed to be a fractional Brownian motion, it is shown in \cite{harang_20_cinfinity,galeati_gubinelli_20_Prevalence} that the associated local time is $\rho$-regular in space for any $\rho \in (0,\frac{1}{2H}-\frac{d}{2})$. We summarize this in the following proposition. 
\begin{proposition}\label{prop: reg of avg op fBm}
	Let $Z:[0,T]\times \Omega \rightarrow \bR^d$ be a fractional Brownian motion with Hurst parameter $H\in (0,1)$. For a vector field $f\in \cH^\beta(\bR^d)$ with $\beta\in \bR$, let $\Gamma f:[0,T]\times \bR^d \rightarrow \bR^d$ be defined as in \eqref{eq:GammaDef}. Then there exists a set of full measure, $\tilde{\cN}\subset \tilde{\Omega}$, depending only on $Z$, such that for all $\tilde{\omega}\in \tilde{\cN}$ and any $\gamma\in (\frac{1}{2},1)$,
	$$\Gamma(\tilde{\omega}) f\in \cC^\gamma_T \cC^{\beta+\frac{1-\gamma}{H}-d(1-\gamma)}(\bR^d).$$
	Moreover, the mapping $f\mapsto \Gamma f$ defines an average operator on $\cH^\beta(\bR^d)$.
\end{proposition}
\begin{proof}
	A full proof is given in the proof of \cite[Thm. 17 and Rem. 18]{harang_20_cinfinity}, however, we give a quick sketch using the local time approach. 
	Using Proposition \ref{prop: reg of local time} we see that the local time $L$ associated to $(Z_t)_{t\in [0,T]}$ is contained in $\cC^\gamma_T\cH^\rho(\bR^d)$, $\tilde{\PP}$-a.s. for some $\rho \in (0,\frac{1}{2H}-\frac{d}{2})$ and $\gamma\in (\frac{1}{2},1)$, then an application of Young's convolution inequality for Besov spaces, \eqref{eq:PwiseBesovYoung}, gives
	\begin{equation}
		\|f\ast L_{s,t}\|_{\cC^{\beta+\alpha}}\leq \|f\|_{\cH^\beta}\|L\|_{\cC^\gamma_T \cH^\alpha}|t-s|^\gamma 
	\end{equation}
	Thus since $\Gamma f=f\ast \bar{L}$ where $\bar{L}_t(x)=L_t(-x)$, as seen in \eqref{eq:avg function as local time conv}, and using the fact that the Sobolev regularity of $\bar{L}$ is identical to that of $L$, it follows that the path $Z$ is $\rho$-regularising according to Definition \ref{def: regularising path}.
	An application of the interpolation shown in Remark \ref{rem:interpolation} completes the proof. 
\end{proof}
Note that in contrast to Example \ref{example:CatGubinelli} the full measure set, $\tilde{\cN}$, here does not depend on $f$. However, the regularity gain is lower, at almost $\frac{1}{2H}-\frac{d}{2}$, as opposed to almost $\frac{1}{2H}$.

As the concept and regularity of averaging operators as given in Definition~\ref{def: avg operator} is by now well established, and the examples of explicit paths which provide a regularising effect is vast, for the rest of this text we do not deal with particular paths but rather assume that the average operator $\Gamma$ can be built from a suitable path. In Section~\ref{sec:Applications} we provide some concrete examples with $Z$ a fractional Brownian motion to highlight the degree of roughness one might expect to require in certain cases of classical interest.

Once it is established that $\Gamma$ is an operator from $\cB_{p,q}^\beta(\bR^d) \rightarrow \cC^\gamma_T\cC^{\beta+\rho}(\bR^d)$ for some $\rho>-\beta$ and $\gamma>1/2$, we return to the ODE \eqref{eq:SimpleODE}. The idea now is to use the spatial regularity of $\Gamma f$ to ensure well-posedness of the reformed equation \eqref{eq:ModifiedODE}. To do so we employ the method of non-linear Young integrals, introduced in \cite{catellier_gubinelli_16} and also employed in \cite{harang_20_cinfinity, galeati_gubinelli_22_noiseless, galeati_harang_20}. A more general survey can be found in \cite{galeati_21_NLY}. Considering a path $Y\in \cC^{\gamma'}_T$ with $\gamma+\gamma'((\beta+\rho)\wedge 1)>1$, one defines 
\begin{equation}\label{eq:NLYIDefinition}
	\int_0^t \Gamma_{\dd r}f(Y_r):=\lim_{|\cD|\rightarrow 0} \sum_{[u,v]\in \cD} \Gamma_{u,v}f(Y_u), 
\end{equation}
where $\cD$ is any partition of the given time interval and $|\cD|$ is the maximal increment size in $\cD$. An application of the Sewing lemma, known from the theory of rough paths (cf. \cite[Lem. 4.2]{friz_hairer_14}), proves that this integral is well defined. Indeed, setting $\Xi_{s,t}=\Gamma_{s,t}f(Y_s)$ then we see that the abstract integral
\begin{equation*}
	\cI(\Xi)_t-\cI(\Xi)_s=\lim_{|\cD|\rightarrow 0} \sum_{[u,v]\in \cD} \Xi_{u,v}
\end{equation*}
is well defined, if for all $(s,t)\in \Delta_T^2$
$$
|\Xi_{s,t}|\lesssim |t-s|^{\delta_1}, \qquad {\rm and} \qquad |\delta_u\Xi_{s,t}|\lesssim |t-s|^{\delta_2}
$$
where $\delta_1\in (0,1)$, $\delta_2>1$ and for $u\in [s,t]$, $\delta_u\Xi_{s,t}:=\Xi_{s,t}-\Xi_{s,u}-\Xi_{u,t}$.
It is readily checked in our case that 
\begin{equation*}
	\delta_u\Xi_{s,t}=\Gamma_{u,t}f(Y_s)-\Gamma_{u,t}f(Y_u). 
\end{equation*}
So invoking the assumption that $\Gamma f\in \cC^\gamma_T\cC^{\beta+\rho}(\bR^d)$ it holds that for any $x,y\in \bR^d$ and $(s,t)\in \Delta_2^T$,
\begin{equation*}
	|\Gamma_{s,t}f(x)-\Gamma_{s,t}f(y)| \lesssim |x-y|^{(\beta+\rho)\wedge 1}|t-s|^\gamma,
\end{equation*}
and thus
\begin{equation*}
	|\delta_u\Xi_{s,t}|\lesssim [Y]_{\gamma'}|t-u|^\gamma|u-s|^{((\beta+\rho) \wedge 1)\gamma'}\lesssim |t-s|^{((\beta+\rho) \wedge 1)\gamma'+\gamma}. 
\end{equation*}
Since $(\rho \wedge 1)\gamma'+\gamma>1$ by assumption, we conclude that the integral \eqref{eq:NLYIDefinition} is well defined. We will refer to this construction as the non-linear Young integral (NLYI) due to the structure of the integrand.\\ \par
In the coming sections we will use the concept of the averaging operator $\Gamma$ to give meaning to McKean--Vlasov equations. This leads us to consider non-linear Young integrals constructed to coincide with integrals of the form
\begin{equation}\label{eq: integral mu star}
	\int_0^t (K\ast\mu_r)(Y_r+Z_r) \dd r,
\end{equation}
where $\mu \in \cP(\cC^{\gamma'}_T)$, $Y \in \cC^{\gamma'}_T$ and $K \in \cB_{p,q}^\beta(\bR^d)$ for $\beta\in \bR$. The regularising path $Z:[0,T]\rightarrow \bR^d$ we will take to be deterministic and sufficiently regularising such that $\Gamma K$, as defined in \eqref{eq:GammaDef}, is contained in $\cC_T^{\gamma}\cC^\alpha$ for any $T>0$, some $\gamma>1/2$ and $\alpha\geq 2$. From Proposition \ref{prop: reg of avg op fBm} we see that this assumption is not vacuous. Indeed, we can always choose a sample path of a fractional Brownian motion (on a different probability space) with Hurst parameter $H\in (0,1)$ as small as we want (this can now be seen as a deterministic path), so that $\beta+\frac{1}{2H}-\frac{d}{2}>2$. We mention that, much like in the theory of rough paths, we require $1$-degree more regularity than the spatial Lipschitz property on $\Gamma K$ in order to obtain stability of solutions.

By analogy with \eqref{eq:NLYIDefinition} our first task will be to construct the non-linear Young integral
\begin{equation}\label{eq:MeasureDependentNYLI}
	\int_0^t \Gamma_{\dd r}K\ast \mu_r(Y_r) := \lim_{|\cD|\rightarrow 0} \sum_{[u,v]\in \cD} \Gamma_{u,v}K\ast \mu_u(Y_u).
\end{equation}
We again use the sewing lemma to show that the non-linear integral on the right hand side is well defined in a Young sense, however, for this purpose we require a notion of H\"older continuity for the measure valued flow $t\mapsto \mu_t$. This is discussed in Section \ref{sec:MeasureFlowsRegularity} below. We note that in the sequel we will view $K$ as fixed for a given interacting particle system, so for notational ease we will collapse $\Gamma_{\dd r}K$ to $\Gamma_{\dd r}$.
\section{Wasserstein Distances and H\"older Regularity of Measure Flows}\label{sec:MeasureFlowsRegularity}
As we saw in the construction of the non-linear Young integral, \eqref{eq:NLYIDefinition}, it was important that the path $Y:[0,T]\rightarrow \bR^d$ was sufficiently regular. Since we are concerned with defining non-linear Young integrals with measure valued integrands, as in \eqref{eq:MeasureDependentNYLI}, we will require a notion of time regularity for measure valued flows. In this section we recap some well known material concerning the notion of Wasserstein distances between probability measures and employ them to make rigorous a notion of H\"older continuity for measure valued flows. Similar ideas were applied in \cite{cass_lyons_14}.

If $(E,d_E)$ is a metric space for any $\mu \in \cM(E)$ and $p\geq 1$ we define the $p^{\text{th}}$-moment of $\mu\in \cM(E)$ by the expression 
\begin{equation*}
	\int_{E}d_E(\xi,x)^p\,\dd |\mu|(x), \quad \text{ for some }\xi \in E.
\end{equation*}
For $p \in [1,\infty)$ we let $\cM_p(E)$ denote the set of real valued Radon measures with finite $p^{\text{th}}$-moment and we denote the subspace of zero mass Radon measures by $\cM^0(E):=\left\{\mu \in \cM(E)\,:\, \mu(E)=0\right\}$ (resp. $\cM_p^0(E):=\left\{\mu \in \cM(E)\,:\, \mu(E)=0, \mu \text{ has finite } p^{\text{th}} \text{ moment }\right\}$). We write $\cP_p(E)$ for the probability measures with finite $p^{\text{th}}$ moment on $E$. For $(E,d_E),(F,d_F)$ a pair of metric spaces and $p,q \in [1,\infty]$ we write $\cP_{p,q}(E\times F)$ for the set of probability measures, $\mu$, on $E\times F$ whose first marginals, $\mu|_{E}$, lie in $\cP_p(E)$ and whose second marginals, $\mu|_F$, lie in $\cP_q(F)$.
\begin{definition}[Wasserstein Distances]\label{def:WassDist}
	Let $(E,d_E)$ be a Polish space, and $\mathcal{P}_p(E)$ be as above. Then we may equip $\cP_p(E)$ with the distance,
	\begin{equation}\label{eq:WassersteinPMetric}
		\mathcal{W}_{p;E}(\mu,\nu):= \left(\inf_{m \in \Pi(\mu,\nu)} \iint_{E\times E} d_E(x,y)^p \dd m(x,y)\right)^{\frac{1}{p}},
	\end{equation}
	where $\Pi(\mu,\nu)\subseteq \mathcal{P}_{p,p}(E\times E)$ is the set of measures on the product space with first marginal equal to $\mu$ and second marginal equal to $\nu$.
\end{definition}
Note that the above definition makes no assumption on an underlying abstract probability space(s) giving rise to the measures $\mu,\,\nu \in \cP_p(E)$.
\begin{remark}
	For $\mu,\,\nu\in \cP_p(E)$ the metric $\cW_{p;E}$ can be equivalently characterised in terms of $E$ valued random variables on a fixed probability space. Let $(\Omega,\cF,\PP)$ be a probability space and $X,\,Y$ be any measurable mappings $X,\,Y:\Omega\rightarrow E$ such that $\EE\left[|X|^p \right]<\infty$ and $\EE\left[|Y|^p\right]<\infty$, with $\mu=\cL(X)$ and $\nu=\cL(Y)$, then,
	\begin{equation*}
		\cW_{p;E}(\mu,\nu)= \inf_{X,\,Y} \EE\left[d(X,Y)^p \right]^{\frac{1}{p}},
	\end{equation*}
	where the infimum is taken over all $X,\,Y$ as above. 
\end{remark}
The Wasserstein metrics play an important role in the study of McKean--Vlasov equations.
\begin{proposition}\label{prop:PortmanteauTheorem}Let $(E,d_E)$ be a polish space and for any $p\geq 1$ let $\cW_{p;E}$ be the distance defined by \eqref{eq:WassersteinPMetric}. Then the following all hold:
\begin{enumerate}[label=(\roman*)]
\item \label{it:WassTrueMetric}The distance $\cW_{p;E}$ satisfies the properties of a metric on $\cP_p(E)$. Furthermore $(\cP_p,\cW_{p;E})$ is itself polish.
\item \label{it:WassMinimiser} For any pair $\mu,\,\nu \in\cP_p(E)$ there exists a measure $\bar{m} \in \Pi(\mu,\nu)$ such that,
\begin{equation*}
\cW_{p;E}(\mu,\nu)= \left(\iint_{E \times E} d_E(x,y)^p\,\dd \bar{m}(x,y)\right)^{\frac{1}{p}}.
\end{equation*}
\item \label{it:Portmanteau}Let $(\mu^n)_{n\in\bN}$ be a sequence in $\cP_p(E)$, then the following are equivalent:
\begin{enumerate}[label=(\alph*)]
\item There exists a $\mu \in \cP_p(E)$ such that $\lim_{n \rightarrow \infty}\cW_{p;E}(\mu^n,\mu)=0$
\item The sequence converges weakly to $\mu\in \cP_p(E)$ and there exists an $e_0\in E$ such that
\begin{equation*}
\lim_{k\rightarrow \infty} \int_{E\setminus B_k(\xi)}d(e_0,x)^p \,\dd\mu^n(x)=0, \text{ uniformly in $n \in \bN$.}
\end{equation*}
\end{enumerate}
\end{enumerate}
\end{proposition}
We refer the reader to \cite{ambrosio_gigli_savare_08,villani_03} for more details. Point \ref{it:WassTrueMetric} in particular is proved in \cite[Ch. 1]{villani_03} and Point \ref{it:Portmanteau} is proved as \cite[Prop. 7.1.5]{ambrosio_gigli_savare_08}.\\ \par 
The $\cW_{1;E}$ metric will play a central role in our analysis. By the Kantorovich--Rubinstein duality (Theorem \ref{th:KRDuality} below) we see that $\cW_{1;E}$ can be written as the restriction of a norm on the linear space $\cM_1(E)$ to $\cP_1(E)$. This allows us to define a notion of H\"older continuous measure flows $t\mapsto \mu_t \in \cP_1(\bR^d)$, see Definition \ref{def:MeasureFlowSemiNorm} below.\\ \par 
For a complete metric space $E$ and a map $\varphi:E\rightarrow \bR$ we define the Lipschitz constant of $\varphi$ by setting
\begin{equation*}
	[\varphi]_{\lip(E)}:= \sup_{x\neq y\,\in E}\frac{|\varphi(x)-\varphi(y)|}{d_E(x,y)},
\end{equation*}
and then we define the set
\begin{equation*}
	\lip_1(E):= \left\{\varphi:E\rightarrow \bR\,:\,[\varphi]_{\lip} \leq1 \right\}.
\end{equation*}
For $\mu \in \cM^0_1(E)$ we define its Lipschitz dual norm by the expression
\begin{equation}\label{eq:KRNorm}
	\|\mu\|_{\lip^*(E)} := \sup_{\varphi \in \lip_1(E)} \int_{E} \varphi \,\dd \mu.
\end{equation}
Given $\mu,\,\nu \in \cP_1(E)$, while $\mu-\nu \notin \cP(E)$ the difference is in $\cM^0_1(E)$ and so $\|\mu-\nu\|_{\lip^*(E)}$ is well defined. The Kantorovich--Rubinstein theorem states that this quantity is equal to the $1$-Wasserstein distance.
\begin{theorem}[Kantorovich--Rubinstein Duality]\label{th:KRDuality}
	Let $(E,d_E)$ be a Polish space and ${\rm lip}_1(E)$ be as defined above. Then for all $\mu,\,\nu \in \cP_1(E)$ we have the equality
	\begin{equation}\label{eq:KRDuality}
		\cW_{1;E}(\mu,\nu) = \|\mu-\nu\|_{{\rm lip}^*(E)}.
	\end{equation}
	Furthermore, it does not affect the norm on the right hand side if we further restrict the supremum to all $\varphi \in \lip_1(E)\cap C_b(E)$.
\end{theorem}
\begin{proof}
	See the proof of \cite[Th. 1.14]{villani_03}.
\end{proof}
\begin{remark}
	Although Theorem \ref{th:KRDuality} is referred to as the Kantorovich--Rubesntein duality the dual quantity in our case is actually the Lipschitz dual norm. This discrepancy is resolved when $E$ is a compact metric space. We define the $\KR(E)$ norm on $\cM_1(E)$ by the expression
	\begin{equation*}
		\|\mu\|_{\KR(E)} = \sup_{\varphi \in \Lip_1(E)}\int_E \varphi \,\dd\mu,\qquad \Lip_1(E):= \left\{\varphi:E\rightarrow \bR\,:\,\|\varphi\|_{C}+[\varphi]_{\lip} \leq1 \right\}.
	\end{equation*}
	Then it is easily seen that when $E$ has finite diameter it is equivalent to restrict the supremum to $\varphi:E\rightarrow \bR$ such that $[\varphi]_{\lip}<1$ and $\varphi(x)=0$ for some $x \in E$. Then, since for any $\mu,\,\nu \in \cP_1(E)$ the difference $\mu-\nu$ is in $\cM^0_1(E)$ and so integrates constants to zero, one has
	\begin{equation*}
		\|\mu-\nu\|_{\lip*(E)} = \|\mu-\nu\|_{\KR(E)}.
	\end{equation*}
\end{remark}
\begin{remark}
	It is a classical result that unless $E$ is a finite space, any complete metric on $\cM(E)$ is equivalent to the total variation metric, which metrizes the topology of strong convergence. Therefore it is clear that neither $\left(\cM_1(E),\|\,\cdot\,\|_{\lip*(E)}\right)$ nor $\left(\cM_1(E),\|\,\cdot\,\|_{\KR(E)}\right)$ are complete metric spaces. However, combining Theorem \ref{th:KRDuality} and Point \ref{it:WassTrueMetric} of Proposition \ref{prop:PortmanteauTheorem} one sees that $(\cP_1(E),\|\,\cdot\,\|_{\lip*(E)})$ is complete.
\end{remark}
Since $(\cC^\beta_T(\bR^d),\|\,\cdot\,\|_{\beta})$ is a Banach space all the results of the previous section apply to the Wasserstein metrics $\cW_{p;\cC^\beta_T}$. In particular the space $(\cP_1(\cC^\beta_T),\cW_{1;\cC^\beta_T})$ is itself a Polish space when equipped with the metric,
\begin{equation*}
	\cW_{1;\cC^\beta_T}(\mu,\nu)=\|\mu-\nu\|_{\lip^*(\cC^\beta_T)}.
\end{equation*}
From now on, when $E$ is a Polish space, unless otherwise specified we always treat $\cP_p(E)$ as being equipped with the metric $\cW_{p;E}$. When the context is clear we will simply write $\|\,\cdot\,\|_{\lip^*}$, dropping the explicit dependence on $E$.
\subsection{H\"older Regularity of Measure Valued Flows}\label{subsec:HolderMeasureFlow} Let $f \in C_T$, then for every $t \in [0,T]$ we define the projection $\pi_t :C_T \rightarrow \bR^d$ to be the map such that $\pi_t f := f_t$. Then for $\mu \in \cM(C_T)$, we set $\mu_t := \pi_t\#\mu \in \cM(\bR^d)$.
Using the definition of the push-forward, we see that if $\mu \in \cP(\cC^\beta_T)$, then
\begin{equation}\label{eq:PMomentEmbedding}
	\int_{\bR^d} |x|^p\, \dd\mu_t(x)= \int_{\cC^\beta_T}|f_t|^p\, \dd\mu(f) \lesssim  \int_{\cC^\beta_T}\|f\|^p_{\beta;T}\,\dd\mu(f),\quad \forall\,t\in[0,T].
\end{equation}
So $\mu \in \cP_p(\cC^\beta_T) \Rightarrow \mu_t \in \cP_p(\bR^d)$ for every $t \in [0,T]$. In particular the $\lip^*(\bR^d)$ norm of $\mu_t$ is well defined for every $t\in [0,T]$.
We use this fact to define a notion of H\"older continuity for $\cP_1(\bR^d)$ valued measure flows.
\begin{definition}\label{def:MeasureFlowSemiNorm}
	Let $\beta \in (0,1)$, $0\leq s<t<\infty$ and $[s,t]\ni u\mapsto \mu_u \in \cM_1(\bR^d)$ be a flow of Radon measures. Then we say that $(\mu_u)_{u\in[s,t]}$ is $\beta$-H\"older continuous if 
	\begin{equation}\label{eq:MeasureFlowHolderSemiNorm}
	[\mu]_{\beta;[s,t]}:=\sup_{u \neq v\,\in [s,t]} \frac{\|\mu_v-\mu_u\|_{\lip^*(\bR^d)}}{|v-u|^\beta}<\infty.
	\end{equation}
\end{definition}
We write $\cC^\beta_{[s,t]}\cP_1(\bR^d):=\cC^\beta([s,t];\cP_{1}(\bR^d),\vertiii{\,\cdot\,;\,\cdot\,}_{\beta;[s,t]})$ for the space of $\cP_1(\bR^d)$ valued flows, equipped with the metric, 
\begin{equation*}
	\vertiii{\mu;\nu}_{\beta;[s,t]}:= \|\mu_0-\nu_0\|_{\lip^*(\bR^d)} + [\mu-\nu]_{\beta;[s,t]}.
\end{equation*}
We use the unusual notation $\vertiii{\,\cdot\,;\,\cdot\,}_{\beta;[s,t]}$ since the space of $\cP_1(\bR^d)$ valued flows is not linear. As with real valued H\"older continuous maps, we retain the convention that if $[s,t]=[0,T]$ for some $T>0$, we simply write $[\,\cdot\,]_{\beta;T},\,\vertiii{\,\cdot\,,\,\cdot\,}_{\beta;T}$ and $\cC^\eta\cP_1(\bR^d)$. For $\mu \in \cC^\beta([0,T];\cP_1(\bR^d))$ and $\beta' \in (0,\beta)$, one has $[\mu]_{\beta';[s,t]}\leq |t-s|^{\beta-\beta'}[\mu]_{\beta;[s,t]}$ for any $[s,t]\subseteq[0,T]$.
\begin{theorem}\label{th:HolderFlowEmbedding}
	The push-forward of the projection map $\pi_t$ gives a continuous embedding from $(\cP_1(\cC^\beta_T),\cW_{1;\cC^\beta_T})$ into $(\cC^\beta_T\cP_1(\bR^d),\vertiii{\,\cdot\,;\,\cdot\,}_{\beta;T})$ and for $\mu,\,\nu \in \cP_1(\cC^{\beta}_T)$,
	\begin{equation}\label{eq:HolderFlowEmbeddingNorm}
		\vertiii{\mu;\nu}_{\beta;T} \leq \cW_{1;\cC^\beta_T}(\mu,\nu).
	\end{equation}
\end{theorem}
\begin{proof}
	Let $\mu,\,\nu \in \cP_1(\cC^\beta_T)$ and we define the associated measure flows $(\mu_t)_{t \in [0,T]},\,(\nu_t)_{t\in[0,T]} \subset \cP_1(\bR^d)$ via the projection $\pi_t :\cC^\beta_T \rightarrow \bR^d$. From \eqref{eq:PMomentEmbedding} we see that for all $t\in[0,T]$, $\mu_t,\,\nu_t \in \cP_1(\bR^d)$. Then let $\varphi \in \lip_1(\bR^d)$ and using the push-forward, we have
	\begin{align*}
		\int_{\bR^d} \varphi(x) \,\dd (\mu_t(x)-\nu_t(x)-\mu_s(x)+\nu_s(x))&=\int_{\cC^\beta_T} (\varphi(f_t)-\varphi(f_s))\,\dd(\mu(f)-\nu(f))\\
		&\leq \int_{\cC^\beta_T} |f_t-f_s|\,\dd(\mu(f)-\nu(f)),
	\end{align*}
	where we used the fact that $\varphi \in \lip_1(\bR^d)$ in the last line. Dividing by $|t-s|^\beta$ and taking the supremum over $s\neq t\,\in [0,T]$, we have the bound,
	\begin{align*}
		\sup_{t\neq s \in [0,T]} \frac{1}{|t-s|^\beta} \int_{\bR^d} \varphi(x) \,\dd (\mu_t(x)-\nu_t(x)-\mu_s(x)+\nu_s(x)) &\leq \int_{\cC^\beta_T} [f]_{\beta;T}\, \dd(\mu(f)-\nu(f))\\
		&\leq \sup_{\phi \in \lip_1(\cC^\beta_T)} \int_{\cC^\beta_T}\phi(f)\,\dd (\mu(f)-\nu(f)),
	\end{align*}
	where the last inequality follows since $[\,\cdot\,]_{\beta}$ is a $\lip_1$ function on $\cC^\beta_T$. Therefore we have
	\begin{equation*}
		[\mu-\nu]_{\beta;T} \leq \|\mu-\nu\|_{\lip*(\cC^\beta_T)} = \cW_{1;\cC^\beta_T}(\mu,\nu).
	\end{equation*}
	By the same steps, we have
	\begin{align*}
		\|\mu_0-\nu_0\|_{\lip^*(\bR^d)} + [\mu-\nu]_{\beta;T} &\leq \int_{\cC^\beta_T} (|f_0| + [f]_{\beta;T} )\,\dd (\mu(f)-\nu(f)) \\
		&= \int_{\cC^\beta_T} \|f\|_{\beta;T} \,\dd (\mu(f)-\nu(f)),
	\end{align*}
	from which \eqref{eq:HolderFlowEmbeddingNorm} follows.
\end{proof}
\begin{lemma}\label{lem:HolderFlowsComplete}
	The metric space $(\cC^\beta_T\cP_1(\bR^d),\vertiii{\,\cdot\,;\,\cdot\,}_{\beta;T})$ is complete.
\end{lemma}
\begin{proof}
From Proposition \ref{prop:PortmanteauTheorem} we have that $(\cP_1(\bR^d),\|\,\cdot\,\|_{\lip^*})$ is a complete metric space. Therefore,
a minor modification of the usual proof that the space of real valued $\alpha$-H\"older functions is complete shows that $(\cC^\beta_T\cP_1(\bR^d),\vertiii{\,\cdot\,;\,\cdot\,}_{\beta;T})$ is complete.
\end{proof}
At last we mention a simple Taylor expansion type lemma for measure valued flows.
\begin{lemma}\label{lem:WassersteinTaylor}
	For any two probability measures $\mu,\nu\in \cP(\cC^\beta_T)$, and $t\in[0,T]$, we have that
	\begin{equation}\label{eq:PWiseWassDistanceBnd}
		\cW_{1,\bR^d}(\mu_t,\nu_t)\leq \cW_{1;\bR^d}(\mu_0,\nu_0)+T^\beta\cW_{1;\cC^{\beta}_T}(\mu,\nu) .
	\end{equation} 
\end{lemma}
\begin{proof}
	From \eqref{eq:KRDuality}, for any $t>0$, we see that
	\begin{align*}
		\cW_{1;\bR^d}(\mu_t,\nu_t)&= \|\mu_t-\nu_t\|_{\KR(\bR^d)}\\
		&\leq  \|\mu_0-\nu_0\|_{\KR(\bR^d)}+T^\beta [\mu-\nu]_{\beta;T}\\
		&\leq \cW_{1;\bR^d}(\mu_0,\nu_0)+T^\beta\cW_{1;\cC^{\beta}_T}(\mu,\nu).
	\end{align*}
\end{proof}
%
%
%
\section{Well-Posedness and Stability of Distribution Dependent non-linear Young Equations}\label{sec:AbstractMKVResults}
In this section we employ the results on non-linear Young integration discussed in Section \ref{sec:PathwiseRegularisation} and the notions of H\"older continuous measure flows introduced in Section \ref{sec:MeasureFlowsRegularity} to demonstrate existence, uniqueness and stability for non-linear Young equations of McKean--Vlasov type. 

For the rest of this section we fix $T>0$, $(\gamma,\eta)\in \left(\frac{1}{2},1\right)\times (0,\infty)$ such that
\begin{equation}\label{eq:EtaGammaAssumption}
  (\eta \wedge \gamma)+ \gamma>1,
\end{equation}
and $(\Omega,\cF,\PP)$ an abstract probability space. All laws of random variables will be taken with respect to $\PP$. These are the same standing assumptions as those made at the beginning of Section \ref{sec:PWiseMainResults}.

The equations we consider in this section are of the form
\begin{equation}\label{eq:AbstractNLYEMKV}
  Y_t = \xi + \int_0^t \left(\Gamma_{\dd r}\ast\mu_r \right)(Y_r) + B_t, \quad \mu = \cL(Y),
\end{equation}
 where $\Gamma \in \cC^\gamma_T\cC^\alpha(\bR^d)$ for some $\alpha \geq 1$, is a given function, and the drift term is rigorously defined in Subsection \ref{subsec:NLYI} below. Throughout this section $\Gamma$ will be assumed to be a given space-time function, not necessarily an averaged field of any particular kernel. In Section \ref{sec:MainResultsProofs} we show how such $\Gamma$ can be built from a wide range of distributions $K\in \cS^\prime(\bR^d)$ and local times associated to regularising paths $Z \in C_T$. We fix a solution concept for \eqref{eq:AbstractNLYEMKV}.

\begin{definition}\label{def:AbstractNLYE}
Let $\alpha\geq 1$, $\Gamma:[0,T]\times \bR^d\rightarrow \bR^d$ be such that $\Gamma \in \cC^\gamma_T\cC^\alpha(\bR^d)$ and $(\xi,B)\in L^1(\Omega;\bR^d \times\cC^\eta_T)$, with $B_0=0$. Then we say that a random variable $Y:\Omega \rightarrow \cC^{\eta\wedge\gamma}_T$ is a solution to the non-linear Young equation of McKean--Vlasov type if for any $t\in(0,T]$ the identity,
\begin{equation}\label{eq:AbstractNLYEFrom0}
  Y_t(\omega) = \xi(\omega) + \int_0^t (\Gamma_{\dd r}\ast\mu_r)(Y_r(\omega)) + B_t(\omega),\quad \mu = \cL(Y)
\end{equation}
holds for $\PP$-a.a. $\omega \in \Omega$, where the integral is understood as a non-linear Young integral with measure dependence, properly defined in Lemma \ref{lem:YoungIntegration}.

For any $s\in (0,T)$ and $h\in(0,T-s]$, given the interval $[s,s+h]\subset (0,T]$ and data $(x_s,B)\in L^1(\Omega;\bR^d\times \cC^\eta_T)_{[s,s+h]}$ we say that $Y:\Omega\rightarrow \cC^{\eta \wedge \gamma}_T$ is a solution to the non-linear Young equation on $[s,s+h]$ if for any $t\in[s,s+h]$ the identity,
\begin{equation}\label{eq:AbstractNLYEFromS}
  Y_t(\omega) = x_s(\omega) + \int_s^t (\Gamma_{\dd r}\ast\mu_r)(Y_r(\omega)) + B_t(\omega) -B_s(\omega), \quad \mu= \cL(Y)\big|_{[s,s+h]}, 
\end{equation}
holds for $\PP$-a.a. $\omega \in \Omega$.
\end{definition}
It is immediate from the definition that any solution to the generalised McKean--Vlasov problem \eqref{eq:AbstractNLYEMKV} satisfy the semi-group property. More precisely, if for any $s\leq t\in [0,T]$ and $(\xi,B)\in \bR^d\times \cC^\eta_T$ as above, we let
\begin{equation*}
  S_{[s,t]}(\xi,B):= \xi + \int_s^t \Gamma_{\dd r}\ast \mu_r(Y_r) + B_t-B_s, \quad \mu = \cL(Y)
\end{equation*}
then the identity
\begin{equation*}
  S_{[0,t]}(\xi,B) = S_{[s,t]}\left(S_{[0,s]}(\xi,B),B\right),
\end{equation*}
holds $\PP$-almost surely.
\begin{remark}
It follows that if $B$ is a Markov process on a filtered space $(\Omega,\cF,(\cF_t)_{t\in [0,T]})$ then any solution $Y:\Omega \times [0,T]\rightarrow \bR^d$ will be too. However, it is also easily seen that in this setting the solution $Y$ cannot satisfy the strong Markov property, since the law of the stopped process is not equal to the law of the un-stopped process evaluated at the random time.
\end{remark}
The main result of this section is the following abstract equivalent of Theorem \ref{th:IntroGeneralMainTheorem}.
\begin{theorem}\label{th:AbstractNLETheorem}
Let $\alpha\geq 2$, $\Gamma \in \cC^\gamma_T\cC^\alpha(\bR^d)$, satisfying the assumptions of Lemma \ref{lem:MeasureSewingStability} below, and $(\xi,B)\in L^1(\Omega;\bR^d)\times L^{p}(\Omega;\cC^\eta_T)$, such that $B_0=0$. Then there exists a unique solution $Y\in L^1(\Omega;\cC^{\eta\wedge \gamma}_T)$ to the non-linear Young equation of McKean--Vlasov type, \eqref{eq:AbstractNLYEMKV}, in the sense of Definition \ref{def:AbstractNLYE}.

Furthermore, if $(\xi^1,B^1),\,(\xi^2,B^2)\in L^1(\Omega;\bR^d)\times L^{p}(\Omega;\cC^\eta_T)$, for any $p\geq 1$, are two pairs of data, then given the corresponding solutions $Y^1,\,Y^2 \in L^1(\Omega;\cC^{\eta\wedge \gamma}_T)$, defining $\mu^1=\cL(Y^1),\,\mu^2=\cL(Y^2)$ and for any $\beta \in (1-\gamma,\eta \wedge \gamma)$ choose $q=\frac{2-\beta}{\gamma-\beta}$, there exists a constant
$$C:=C\left(T,\EE\left[[B^1]^{2q}_{\eta;T}\right]\vee \EE\left[[B^2]^{2q}_{\eta;T}\right],\Gamma,\gamma,\eta,\beta\right)>0$$
such that
\begin{equation}\label{eq:AbstractNLYEDataStable}
  \cW_{1;\cC^\beta_T}(\mu^1,\mu^2)\leq C\left( \cW_{1;\bR^d}(\cL(\xi^1),\cL(\xi^2)) +\cW_{2;\cC^\eta_T}(\cL(B^1),\cL(B^2))\right).
\end{equation}
\end{theorem}
\begin{remark}\label{rem:MomentRem2}
    Note here that while the constant $C$ depends only on the $2q$ moments of $B$ for $q = \frac{2-\beta}{\gamma-\beta}$, since $\beta \in (1-\gamma,\eta\wedge \gamma)$ is arbitrary $q$ can in fact be arbitrarily large hence the requirement for for $B,\,B^1,\,B^2$ to have all finite moments.
\end{remark}
In the remainder of this section we first extend the definition of the non-linear Young integral to integrands involving functions convolved with the time marginals of a measure flow. We also obtain stability estimates on the non-linear Young integral (NLYI) with respect to the measure and spatial trajectory. This is all done in Subsection~\ref{subsec:NLYI}. Then we prove Theorem \ref{th:AbstractNLETheorem} in two stages; firstly in Subsection~\ref{subsec:FrozenFlowWP} we freeze a path-measure $\mu \in \cP_1\left(\cC^{\eta\wedge \gamma}_T\right)$ and demonstrate existence and uniqueness of solutions $Y^\mu$ to the dynamics of \eqref{eq:AbstractNLYEMKV} with $\mu$ fixed. Then using the stability in measure of the NLYI we show by a fixed point argument the existence of unique solutions to the full McKean--Vlasov type non-linear Young equation and the associated stability bound \eqref{eq:AbstractNLYEDataStable}.
\subsection{Non-Linear Young Integration for Measure Dependent Integrands}\label{subsec:NLYI}
We extend the notion of non-linear Young integration to include measure dependent integrands. We make use of the results presented in Section \ref{sec:MeasureFlowsRegularity}. For completeness we include proofs of many results even if they closely reflect those already obtained in the literature for NLYI without measure dependence.
\begin{lemma}\label{lem:YoungIntegration}
Let $\alpha\geq 1$, $\Gamma:[0,T]\times \bR^d$ be in $\cC^\gamma_T\cC^\alpha(\bR^d)$ and be such that for all $ s<t\in [0,T] $ and $x,y\in \bR^d$
 \begin{equation}\label{eq:itegrationCond.}
  \begin{aligned}
  {\rm (i)}& \qquad |\Gamma_{s,t}(x)|+|\nabla \Gamma_{s,t}(x)|\lesssim |t-s|^\gamma 
  \\
  {\rm (ii)}& \qquad |\Gamma_{s,t}(x)-\Gamma_{s,t}(y)|\lesssim |t-s|^\gamma |x-y|.
  \end{aligned}
\end{equation}
Let $\beta >0$ be such that $\gamma+\beta>1$ and assume we are given $\mu\in \cC^\beta_T\cP_1(\bR^d)$ and $Y\in \cC^{\beta}_T$. Then there exists a unique path
\begin{equation*}
  t\mapsto \int_{0}^{t}(\Gamma_{\dd r}\ast\mu_r)(Y_r)\in \mathcal{C}^{\gamma}([0,T],\bR^d)
\end{equation*} constructed as
\begin{equation}\label{eq:non-linearYoungIntegral}
   \int_{0}^{t}\Gamma_{\dd r}\ast \mu_r(Y_r):= \lim_{|\mathcal{D}|\rightarrow 0}\sum_{[u,v]\in \mathcal{D}}( \Gamma_{u,v}\ast \mu_u)(Y_{u}),
\end{equation}
where $\cD$ is a partition of $[0,t]$ with maximal resolution $|\cD|$. Moreover, there exists a constant $C>0$ such that for all $s<t\in [0,T]$
\begin{equation}\label{eq:SewingLemmaError}
  \left|  \int_{s}^{t}\Gamma_{\dd r}\ast \mu_r(Y_r)-\Gamma_{s,t}\ast \mu_s(Y_s)\right|\leq C |s-t|^{\gamma+\beta} \|\Gamma\|_{\gamma,\alpha}\left([Y]_{\beta;[s,t]}+[\mu]_{\beta;[s,t]}\right).
\end{equation}
\end{lemma}
\begin{remark}
For $\beta<\gamma$ the condition $\gamma+\beta>1$ required by the statement of Lemma~\ref{lem:YoungIntegration} can be relaxed to the condition $\gamma+\beta(\alpha\wedge 1)>1$, for any $\alpha >0$, where $\Gamma \in \cC^\gamma_T\cC^\alpha(\bR^d)$, see e.g. \cite{harang_20_cinfinity}. However, in subsequent sections we require $\alpha\geq 2$ in order to obtain the necessary stability estimates, see Lemma~\ref{lem:MeasureSewingStability}, so we directly impose the simpler requirement above.
\end{remark}

\begin{proof}
We define $\delta_u f_{s,t}:=f_{s,t}-f_{s,u}-f_{u,t}$ and $\Xi_{u,v}:=\Gamma_{u,v}\ast \mu_u(Y_{u})$, and recall from the sewing lemma, \cite[Lem. 4.2]{friz_hairer_14}, that if 
\begin{equation}\label{two bounds}
  |\Xi_{u,v}|\lesssim |v-u|^{\delta_1} \qquad {\rm and}\qquad |\delta_{z}\Xi_{u,v}|\lesssim |v-u|^{\delta_2},
\end{equation}
for $\delta_1\in (0,1)$ and $\delta_2>1$, uniformly in $z\in[u,v]$, then there exists a unique limit of the Riemann sums $\sum_{[u,v]\in \mathcal{D}}\Xi_{u,v}$, along a decreasing sequence of partitions $\cD$ of $[0,t]$. In this case there exists a unique function $\mathcal{I}(\Xi)(t)$ such that 
\begin{equation*}
  \mathcal{I}(\Xi)(t):=\lim_{|\mathcal{D}|\rightarrow 0 }\sum_{[u,v]\in \mathcal{D}}\Xi_{u,v}.
\end{equation*}
We begin by showing the first inequality in \eqref{two bounds}. From condition {\rm (i)} of \eqref{eq:itegrationCond.},  $\Gamma$ is globally bounded in space and is $\gamma$-regular in time, therefore we have that
\begin{equation}\label{eq:XiSewing1}
  |\Xi_{u,v}|=|\Gamma_{u,v}\ast \mu_u (Y_u)| \,\leq\, \int_{\bR^d} |\Gamma_{u,v}(Y_u-y)|\mu_u(\dd y) \leq \|\Gamma\|_{\gamma,\alpha}|v-u|^\gamma,
\end{equation}
where we used that $\mu_u(\bR^d)=1$. Thus the first bound in \eqref{two bounds} holds. 
To prove the second inequality in \eqref{two bounds}, using the additivity of $t\mapsto \Gamma_t$, we observe that for $u\leq z\leq v$
\begin{equation}\label{eq:DelGamma}
  \delta_{z}\left(\Gamma_{u,v}\ast \mu_u(Y_u)\right)=\Gamma_{z,v}\ast (\mu_u-\mu_z) (Y_{u})+(\Gamma_{z,v}\ast \mu_z (Y_{u})-\Gamma_{z,v}\ast \mu_z (Y_{z})).
\end{equation}
Considering the second term of \eqref{eq:DelGamma}, we again use the fact that $\mu_t(\bR^d)=1$ for all $t\in [0,T]$ to obtain by application of \eqref{eq:itegrationCond.} that
\begin{align}
|\Gamma_{z,v}\ast \mu_z (Y_{u})-\Gamma_{z,v}\ast \mu_z (Y_{z})|&\leq\int_{\bR^d}|\Gamma_{z,v}(Y_u-y)-\Gamma_{z,v}(Y_z-y)|\mu_z(\dd y) \notag \\
&\leq \sup_{y\in \bR^d} |\Gamma_{z,v}(Y_u-y)-\Gamma_{z,v}(Y_z-y)| \notag\\
&\leq \|\Gamma\|_{\gamma,\alpha}[Y]_{\beta;[s,t]}|v-u|^{\gamma+\beta},\label{eq: diff T star mu}
\end{align}
where we have used that $|v-z|\vee|z-u|\leq|v-u|$ in the last line. For the first term of \eqref{eq:DelGamma} we first argue that the function,
\begin{equation*}
    \bR^d \ni y\mapsto \varphi(y) := \frac{1}{|v-u|^\gamma\|\Gamma\|_{\gamma,\alpha}}\Gamma_{z,v}(Y_u-y),
\end{equation*}
is $1$-Lipschitz continuous. Using (\rm{ii}) of \eqref{eq:itegrationCond.} we directly find, for $x\neq y\in \bR^d$
\begin{align*}
    |\varphi(x)-\varphi(y)| \leq \frac{|v-z|^\gamma}{|v-u|^\gamma} \leq 1,
\end{align*}
where we again used that $|v-z|\vee |u-z|\leq |v-u|$. Hence, using \eqref{eq:MeasureFlowHolderSemiNorm},
\begin{equation}\label{eq:DiffMuReg}
    | \Gamma_{z,v}\ast (\mu_u-\mu_z) (Y_{u})| =  |v-u|^\gamma\|\Gamma\|_{\gamma,\alpha} \left|\int_{\bR^{d}} \varphi(y)\,\dd(\mu_u-\mu_z)(y)\,\right| \leq |v-u|^{\gamma+\beta}\|\Gamma\|_{\gamma,\alpha}[\mu]_{\beta;[s,t]}.
\end{equation}
Thus since $\gamma+\beta>1$ by assumption, we conclude that the integral in \eqref{eq:non-linearYoungIntegral} is well defined. Moreover, again using \cite[Lem. 4.2]{friz_hairer_14}, we directly obtain the inequality \eqref{eq:SewingLemmaError}. 
\end{proof}
In addition to its construction it will be useful to have estimates on the stability of the non-linear Young integral constructed in \eqref{eq:non-linearYoungIntegral} with respect to the path $Y$ and the measure flow $\mu$. The next lemma establishes these bounds under an additional regularity assumption.
\begin{lemma}\label{lem:MeasureSewingStability}
Let $\alpha\geq 2$, $\Gamma:[0,T]\times \bR^d$ be in $\cC^\gamma_T\cC^\alpha(\bR^d)$ be such that for all $s< t\in [0,T] $ and $x,y\in \bR^d$
\begin{equation}\label{eq:GammaSpaceTimeStabillity}
  \begin{aligned}
  {\rm (i)}& \qquad |\Gamma_{s,t}(x)|+|\nabla \Gamma_{s,t}(x)|\lesssim |s-t|^\gamma 
  \\
  {\rm (ii)}& \qquad |\Gamma_{s,t}(x)-\Gamma_{s,t}(y)|\lesssim |s-t|^\gamma |x-y|
   \\
  {\rm (iii)}& \qquad |\nabla \Gamma_{s,t}(x)-\nabla \Gamma_{s,t}(y)|\lesssim |s-t|^\gamma |x-y|.
  \end{aligned}
\end{equation}
Let $\beta>0$ be such that $\gamma+\beta>1$ and $M_1,\,M_2>0$. Then assuming we are given two measure flows $\mu,\tilde{\mu} \in \cC^\beta_T\cP_1(\bR^d)$, and two paths $Y,\tilde{Y}\in \cC^{\beta}_T$ such that for any $0\leq s<t \leq T$, $[\mu]_{\beta;[s,t]}\vee [\tilde{\mu}]_{\beta;[s,t]}\leq M_1$ and $[Y]_{\beta;[s,t]}\vee[\tilde{Y}]_{\beta;[s,t]}\leq M_2$, then
\begin{equation}\label{eq:non-linearYoungStabillity}
\begin{aligned}
\bigg|\int_s^t \big[\Gamma_{\dd r}\ast \mu_r(Y_r)-\Gamma_{\dd r}\ast \tilde{\mu}_r(\tilde{Y}_r)\big] &-\Gamma_{s,t}\ast\mu_s(Y_s)+\Gamma_{s,t}\ast\tilde{\mu}_s(\tilde{Y}_s)
   \bigg| 
   \\
  &\lesssim |s-t|^{\gamma+\beta}\|\Gamma\|_{\gamma,\alpha}\left(1+M_1+M_2\right) \|Y-\tilde{Y}\|_{\beta;[s,t]}\\
&\quad + |s-t|^{\gamma+\beta}\|\Gamma\|_{\gamma,\alpha}\left(1+M_2\right) \vertiii{\mu;\tilde{\mu}}_{\beta;[s,t]}
  \end{aligned}
  \end{equation}
  where the integral is interpreted as the non-linear Young integral given in Lemma \ref{lem:YoungIntegration}.
\end{lemma}
\begin{remark}
It is crucial that the pre-factors on the right hand side that do not depend on the differences $\mu-\tilde{\mu}$ and $Y-\tilde{Y}$ only depend on the semi-norms in the relevant quantities. Later this is important in proving existence and uniqueness results for the non-linear Young equations as it allows us to obtain contraction bounds independently of the initial data.
\end{remark}
\begin{proof}[Proof of Lemma \ref{lem:MeasureSewingStability}]
Set $\Xi_{s,t}:=\Gamma_{s,t}\ast\mu_s(Y_s)-\Gamma_{s,t}\ast\tilde{\mu}_s(\tilde{Y}_s)$. To prove \eqref{eq:non-linearYoungStabillity} we follow the same strategy as in the proof of Lemma \ref{lem:YoungIntegration}, by invoking the sewing lemma. Recall from \cite[Lem. 4.2]{friz_hairer_14}, that if for any $(s,u,t)\in \Delta_3^T$ the following inequalities are satisfied 
\begin{equation}\label{eq:FrizHairerYoungStabillity}
  |\Xi_{s,t}|\,\lesssim |s-t|^{\delta_1} \qquad {\rm and} \qquad |\delta _u\Xi_{s,t}|\,\lesssim |s-t|^{\delta_2},
\end{equation}
for some $\delta_1\in (0,1)$ and $\delta_2>1$, then
\begin{equation*}
\begin{aligned}
\bigg|\int_s^t \left[\Gamma_{\dd r}\ast \mu_r(Y_r)-\Gamma_{\dd r}\ast \tilde{\mu}_r(\tilde{Y}_r)\right] -\Gamma_{s,t}\ast\mu_s(Y_s)+\Gamma_{s,t}\ast\tilde{\mu}_s(\tilde{Y}_s)
   \bigg| 
\lesssim [\delta\Xi]_{\delta_2;[s,t]}|s-t|^{\delta_2},
  \end{aligned}
  \end{equation*}
where we have used the notation
\begin{equation*}
    [\delta\Xi]_{\delta_2;[s,t]} := \sup_{(r,u,v)\,\in\, \Delta^{[s,t]}_3} \frac{|\delta_u\Xi_{r,v}|}{|r-v|^{\delta_2}}.
\end{equation*}
We begin by splitting $\Xi$ into two functions, setting
\begin{align*}
  \Xi^1_{s,t}:=\Gamma_{s,t}\ast\mu_s(Y_s)-\Gamma_{s,t}\ast\mu_s(\tilde{Y}_s),\quad 
  \Xi^2_{s,t}:=\Gamma_{s,t}\ast(\mu_s-\tilde{\mu}_s)(\tilde{Y}_s). 
\end{align*}
Similar steps as in the proof of the bound \eqref{eq:XiSewing1} in Lemma \ref{lem:YoungIntegration} show that the first bound of \eqref{eq:FrizHairerYoungStabillity} holds for both $\Xi^1,\,\Xi^2$, with $\delta_1=\gamma$. Therefore we concentrate on showing that $|\delta_u \Xi^1_{s,t}|$ and $|\delta_u \Xi^2_{s,t}|$ are both controlled by $|s-t|^{\gamma+\beta}$.
\\ \par
By the fundamental theorem of calculus we can write $\Xi^1$ as 
\begin{equation*}
  \Xi^1_{s,t}=\int_0^1 \nabla \Gamma_{s,t}\ast \mu_s(\rho Y_s +(1-\rho)\tilde{Y}_s)\dd \rho\cdot(Y_s-\tilde{Y}_s). 
\end{equation*}
It is readily checked that for $0\leq s\leq u\leq t\leq T$
\begin{equation*}
  \delta_u \Xi^1_{s,t}=\int_0^1 \nabla \Gamma_{s,u}\ast \mu_u(\rho Y_u +(1-\rho)\tilde{Y}_u)\dd \rho\cdot(Y_u-\tilde{Y}_u)-\int_0^1 \nabla \Gamma_{s,u}\ast \mu_t(\rho Y_t +(1-\rho)\tilde{Y}_t)\dd \rho\cdot (Y_t-\tilde{Y}_t). 
\end{equation*}
By adding and subtracting $\int_0^1 \nabla \Gamma_{s,u}\ast \mu_u(\rho Y_u +(1-\rho)\tilde{Y}_u)d\rho\cdot(Y_t-\tilde{Y}_t)$ in the equality above, we obtain the two differences 
\begin{align*}
  \fD^1_{s,u,t}&= \left(\int_0^1 \nabla \Gamma_{s,u}\ast \mu_t(\rho Y_t +(1-\rho)\tilde{Y}_t)\dd \rho-\int_0^1 \nabla \Gamma_{s,u}\ast \mu_u(\rho Y_u +(1-\rho)\tilde{Y}_u)\dd \rho\right)\cdot(Y_t-\tilde{Y}_t)
  \\
  \fD^2_{s,u,t}&= \int_0^1 \nabla \Gamma_{s,u}\ast \mu_u(\rho Y_u +(1-\rho)\tilde{Y}_u)\dd \rho\cdot (Y_t-\tilde{Y}_t-Y_u+\tilde{Y}_u). 
\end{align*}
Considering $\fD^1$, we add and subtract the term $\int_0^1 \nabla \Gamma_{s,u}\ast \mu_u(\rho Y_t +(1-\rho)\tilde{Y}_t)\dd \rho $, and define,
\begin{align*}
  \fD^{1,1}_{s,u,t} &:=\int_0^1 \nabla \Gamma_{s,u}\ast (\mu_t-\mu_u)(\rho Y_t+(1-\rho)\tilde{Y}_t)\dd \rho \cdot(Y_t-\tilde{Y}_t),\\
  \fD^{1,2}_{s,u,t}&:=\int_0^1 \left(\nabla \Gamma_{s,u}\ast \mu_u(\rho Y_t +(1-\rho)\tilde{Y}_t)- \nabla \Gamma_{s,u}\ast \mu_u(\rho Y_u +(1-\rho)\tilde{Y}_u)\,\right)\dd \rho\,\cdot (Y_t-\tilde{Y}_t)
\end{align*}
In order to bound the term $\fD^{1,1}_{s,u,t}$ we use a similar argument as we used to obtain \eqref{eq:DiffMuReg}. Using (\rm{iii}) of \eqref{eq:YoungIntegralStabillity} we first see that for any $\rho \in [0,1]$, the map
\begin{equation*}
     \bR^d\ni y\mapsto \varphi_\rho(y):=\frac{1}{|t-s|^\gamma\|\Gamma\|_{\gamma,\alpha} }\nabla \Gamma_{s,u}(\rho Y_t + (1-\rho)\tilde{Y}_t-y),
\end{equation*}
is $1$-Lipschitz continuous. It then follows, again from \eqref{eq:MeasureFlowHolderSemiNorm}, that
\begin{align*}
  |\fD^{1,1}_{s,u,t}| &\leq |t-s|^\gamma\|\Gamma\|_{\gamma,\alpha} \left| \int_0^1\int_{\bR^d} \varphi_\rho(y)\,\dd(\mu_u-\mu_t)(y)\,\dd\rho\, \right|\,|Y_t-\tilde{Y}_t|\\
  &\leq |t-s|^{\gamma+\beta}\|\Gamma\|_{\gamma,\alpha}\left[\mu\right]_{\beta;[s,t]}\|Y-\tilde{Y}\|_{\beta;[s,t]}.
\end{align*}
We now bound the $\fD^2$ term; again using similar steps as in the proof of \eqref{eq: diff T star mu}, we have
\begin{equation*}
  \|\fD^2_{s,u,t}\|\lesssim |t-s|^{\gamma+\beta} \|\Gamma\|_{\gamma,\alpha} \|Y-\tilde{Y}\|_{\beta;[s,t]}. 
\end{equation*}
Combining the bounds for $\fD^1 $ and $\fD^2$ gives
\begin{equation}\label{eq:deltaUXi1}
  |\delta_u\Xi^1_{s,t}|\lesssim |t-s|^{\gamma+\beta} \|\Gamma\|_{\gamma,\alpha} \left(1+M_1+M_2\right)\|Y-\tilde{Y}\|_{\beta;[s,t]} .
\end{equation}
Concerning the bound on $|\delta_u\Xi_{s,t}^2|$, we first divide the expression into two parts and then invoking inequalities similar to \eqref{eq: diff T star mu} and \eqref{eq:DiffMuReg} we obtain,
\begin{equation}\label{eq:deltaUXi2}
\begin{aligned}
  |\delta_u \Xi^2_{s,t} | &\lesssim |\Gamma_{u,t}\ast(\mu_s-\tilde{\mu}_s-\mu_u+\tilde{\mu}_u)(\tilde{Y}_s)|+|\Gamma_{u,t}\ast(\mu_u-\tilde{\mu}_u)(\tilde{Y}_t)-\Gamma_{u,t}\ast(\mu_u-\tilde{\mu}_u)(\tilde{Y}_u)|\\
  &\lesssim |t-s|^{\gamma+\beta} \|\Gamma\|_{\gamma,\alpha} \left(1+[\tilde{Y}]_{\beta;[s,t]}\right)\vertiii{\mu;\tilde{\mu}}_{\beta;[s,t]}.
\end{aligned}
\end{equation}
So combining \eqref{eq:deltaUXi1} and \eqref{eq:deltaUXi2}, we see that
\begin{equation*}
\begin{aligned}
| \delta_u \Xi_{s,t}|&\lesssim |t-s|^{\gamma+\beta}\|\Gamma\|_{\gamma,\alpha}\left(1+M_1+M_2\right) \|Y-\tilde{Y}\|_{\beta;[s,t]}+ |t-s|^{\gamma+\beta}\|\Gamma\|_{\gamma,\alpha}\left(1+M_2\right) \vertiii{\mu;\tilde{\mu}}_{\beta;{[s,t]}}.
\end{aligned}
\end{equation*}
This shows that the second of the two conditions in \eqref{eq:FrizHairerYoungStabillity} is satisfied, with $\delta_2=\gamma+\beta>1$. The estimate \eqref{eq:non-linearYoungStabillity} then also follows from the sewing lemma, \cite[Lem. 4.2]{friz_hairer_14}.
\end{proof}
\begin{corollary}\label{cor:non-linearYoungStabillity}
Let $\alpha\geq 2$, $\beta>0$, be such that $\gamma+\beta>1$, $M_1,\,M_2>0$ and $\Gamma\in \cC^\gamma_T\cC^\alpha(\bR^d)$, satisfying the assumptions of Lemma \ref{lem:MeasureSewingStability}. Then, given $\mu,\,\tilde{\mu} \in \cC^\beta_T\cP_1(\bR^d)$ and $Y,\,\tilde{Y}\in \cC^{\beta}_T$, such that for any $0\leq s<t \leq T$, 
we have $[\mu]_{\beta;[s,t]}\vee [\tilde{\mu}]_{\beta;[s,t]}\leq M_1$ and $[Y]_{\beta;[s,t]}\vee[\tilde{Y}]_{\beta;[s,t]}\leq M_2$,
  \begin{equation}\label{eq:YoungIntegralStabillity}
\begin{aligned}
\bigg|\int_s^t \Gamma_{\dd r}\ast \mu_r(Y_r)-(\Gamma_{\dd r}\ast \tilde{\mu}_r)(\tilde{Y}_r)
   \bigg|
   &\lesssim_{\gamma,\beta} |t-s|^{\gamma+\beta}\|\Gamma\|_{\gamma,\alpha}(1+M_1+M_2) \|Y-\tilde{Y}\|_{\beta;[s,t]}\\
& \quad + |t-s|^{\gamma+\beta}\|\Gamma\|_{\gamma,\alpha}(1+M_2) \vertiii{\mu;\tilde{\mu}}_{\beta;[s,t]} \\
   & \quad + |t-s|^\gamma\|\Gamma\|_{\gamma,\alpha}\left(|Y_s-\tilde{Y}_s|+\|\mu_s-\tilde{\mu}_s\|_{\KR}\right).
   \end{aligned}
  \end{equation}
\end{corollary}
\begin{proof}
This follows directly from Lemma \ref{lem:MeasureSewingStability} in combination with the triangle inequality, where we observe that
\begin{multline*}
  \bigg|\int_s^t \Gamma_{\dd r}\ast \mu_r(Y_r)-\Gamma_{\dd r}\ast \tilde{\mu}_r(\tilde{Y}_r)
   \bigg| \leq \left|\Gamma_{s,t}\ast \mu_s(Y_s)-\Gamma_{s,t}\ast \tilde{\mu}_s(\tilde{Y}_s) \right|
   \\
   + \bigg|\int_s^t \left[\Gamma_{\dd r}\ast \mu_r(Y_r)-\Gamma_{\dd r}\ast \tilde{\mu}_r(\tilde{Y}_r)\right] -\Gamma_{s,t}\ast\mu_s(Y_s)+\Gamma_{s,t}\ast\tilde{\mu}_s(\tilde{Y}_s)
   \bigg|,
\end{multline*}
where the estimate for the first term on the right hand side is found by similar procedures as done in the proof of Lemma \ref{lem:MeasureSewingStability}, and a bound for the second is given in \eqref{eq:non-linearYoungStabillity}. 
\end{proof}
\subsection{Existence and Uniqueness under Frozen Measure Flow}\label{subsec:FrozenFlowWP} 

The next theorem provides pathwise existence and uniqueness of \eqref{eq:AbstractNLYEFrom0} in the presence of a frozen measure flow.
\begin{theorem}\label{th:RDEWellPosedFixedMeasure}
Let $\alpha\geq 2$, $\Gamma \in \cC^\gamma_T\cC^\alpha(\bR^d)$, satisfy the assumptions of Lemma \ref{lem:MeasureSewingStability}, $(\xi,B)\in L^1(\Omega;\bR^d)\times L^1(\Omega;\cC_T^\eta)$, with $B_0=0$ and $\mu\in \cP_1(\cC^{\eta\wedge \gamma}_T)$. Then for $\PP$-a.a. $\omega\in \Omega$ there exists a unique solution $Y^\mu(\omega) \in \cC^{\eta\wedge \gamma}_T(\bR^d)$ to the equation
\begin{equation}\label{eq:RDEFixedMeasure}
  Y^\mu_t(\omega) = \xi(\omega) +\int_0^t\Gamma_{\dd r}\ast \mu_r(Y^\mu_r(\omega)) + B_t(\omega).
\end{equation}
 
\end{theorem}
\begin{proof}
First, let $(B_t)_{t\in [0,T]}:= (B_t(\omega))_{t \in [0,T]}$ be a realisation of $B$, finite in $\cC^\eta_T$ and we define the measure valued flow $t\mapsto \mu_t \in \cC^{\gamma\wedge \eta}_T\cP_1(\bR^d)$ by setting $\mu_t = \pi_t \# \mu$. Then for any $\beta \in (1-\gamma,\eta \wedge \gamma)$, $x\in \bR^d$ and $\bar{T} \in [0,T]$ we define the ball in $\cC^\beta_T(\bR^d)$,
\begin{equation*}
  \fB_{\bar{T} ;x} := \left\{Y \in \cC^{\beta}_T(\bR^d)\,:\, Y_0 =\xi,\,[Y]_{\beta;\bar{T} }\leq 1 \right\} ,
\end{equation*}
We equip $\fB_{\bar{T} ;x}$ with the structure of a complete metric space via the H\"older semi-norm $[\,\cdot\,]_{\beta;[0,\bar{T}] }$. Then we define the solution map, $\Phi_{\bar{T}} (Y)$, by setting, for every $Y \in \fB_{\bar{T};\xi}$,
\begin{equation*}
  \Phi_{\bar{T} }(Y)_t := \xi +\int_0^t \left(\Gamma_{\dd r}\ast \mu_r\right)(Y_r) + B_t,\quad \text{ for all }t \in (0,\bar{T} ].
\end{equation*}
We first check that there exists a $T_0>0$ such that $\Phi_{T_0}$ leaves the ball $\fB_{T_0;\xi}$ invariant. Adding and subtracting the term $\Gamma_{s,t}\ast \mu_s(Y_s)$, and then using the fact that we have $\gamma+\beta>1$ to invoke the bounds on the non-linear Young integral from Lemma \ref{lem:YoungIntegration}, for any $0\leq s<t\leq \bar{T} $ we have
\begin{align}
\left|\int_s^{t}\Gamma_{\dd r}\ast \mu_r(Y_r)\,\right|&\leq \left|\int_s^{t}\Gamma_{\dd r}\ast \mu_r(Y_r)-\Gamma_{s,t}\ast \mu_s(Y_s)\right|+\left|\Gamma_{s,t}\ast \mu_s(Y_s)\right| \notag\\
&\lesssim_{\gamma,\beta} |t-s|^{\gamma+\beta} \|\Gamma\|_{\gamma,\alpha}\left([Y]_{\beta;\bar{T}}+[\mu]_{\beta;T}\right) + |t-s|^{\gamma} \|\Gamma\|_{\gamma,\alpha},\label{eq:bound on integral}
\end{align}
so that for any $Y\in \fB_{\bar{T};x}$ we have
\begin{equation}\label{eq:PhiBoundedness}
  [\Phi_{\bar{T}} (Y)]_{\beta;\bar{T}}\lesssim_{\gamma,\beta}
  \bar{T}^{\gamma} \|\Gamma\|_{\gamma,\alpha}\left(1+[\mu]_{\beta;T}\right)+ \bar{T}^{\gamma-\beta} \|\Gamma\|_{\gamma,\alpha} + \bar{T}^{\eta-\beta}[B]_{\eta;T}.
\end{equation}
Therefore we see that choosing $\bar{T}:=T_0>0$ sufficiently small we ensure that $\Phi(\fB_{T_0;x}) \subseteq \fB_{T_0;x}$. The initial condition is satisfied due to the positive regularity of the integral and $B$.\\ \par
The next step is to show that $\Phi_{\bar{T}}$ is a contraction on $\fB_{\bar{T};x}$ for some $\bar{T}\leq T_0$. For $Y,\tilde{Y}\in \fB_{T_0;x}$,
\begin{equation*}
  \left[\Phi_T(Y)-\Phi_T(\tilde{Y})\right]_{\beta;{\bar{T}}} = \left[\int_0^\cdot [\Gamma_{\dd r}\ast \mu_r(Y_r)-\Gamma_{\dd r}\ast \mu_r(\tilde{Y}_r)]\right]_{\beta;\bar{T}}. 
\end{equation*}
For any $0\leq s<t\leq \bar{T}$, from Corollary \ref{cor:non-linearYoungStabillity}, using the fact that $Y_0-\tilde{Y_0}=0$ to replace the norm with the semi-norm, we have
\begin{equation*}
	\begin{aligned}
		\bigg|\int_s^t [\Gamma_{\dd r}\ast \mu_r(Y_r)-\Gamma_{\dd r}\ast \mu_r(\tilde{Y}_r)]
		\bigg| &\lesssim_{\gamma,\beta} |s-t|^{\gamma+\beta}\|\Gamma\|_{\gamma,\alpha}(2+\left[\mu\right]_{\eta\wedge \gamma;T}) [Y-\tilde{Y}]_{\beta;\bar{T}}
		\\
		&\quad + |s-t|^\gamma\|\Gamma\|_{\gamma,\alpha}|Y_t-\tilde{Y}_t|,
	\end{aligned}
\end{equation*}
furthermore, for $t\in [0,\bar{T}]$ we have
\begin{equation*}
	|Y_t-\tilde{Y}_t| \leq |Y_0-\tilde{Y}_0| + t^\beta[Y-\tilde{Y}]_{\beta;t},
\end{equation*}
so, using that $Y_0=\tilde{Y}_0=x$, we have
\begin{equation*}
	\begin{aligned}
		\bigg|\int_s^t [\Gamma_{\dd r}\ast \mu_r(Y_r)-\Gamma_{\dd r}\ast \mu_r(\tilde{Y}_r)]
		\bigg| &\lesssim_{\gamma,\beta} |t-s|^{\gamma+\beta}\|\Gamma\|_{\gamma,\alpha}(2+\left[\mu\right]_{\eta\wedge \gamma;T}) [Y-\tilde{Y}]_{\beta;\bar{T}}
		\\
		&\quad + |t-s|^\gamma \bar{T}^\beta\|\Gamma\|_{\gamma,\alpha}[Y-\tilde{Y}]_{\beta;\bar{T}},
	\end{aligned}
\end{equation*}
Taking the $\beta -$H\"older semi-norm of both sides we have that
  \begin{equation}\label{eq:PhiContractionBound}
     [\Phi_{\bar{T}}(Y)-\Phi_{\bar{T}}(\tilde{Y})]_{\beta;\bar{T}} \lesssim_{\gamma,\beta} \bar{T}^{\gamma}\|\Gamma\|_{\gamma,\alpha} \left(1+ \left[\mu\right]_{\eta\wedge \gamma;T} \right)[Y-\tilde{Y}]_{\beta;\bar{T}}.
  \end{equation}
 So choosing $\bar{T}=T_1\in (0,T_0]$ sufficiently small, $\Phi$ is a contraction on $\fB_{T_1;x}$. It follows by the Banach fixed point theorem, that there exists a unique solution to \eqref{eq:RDEFixedMeasure} contained in $\fB_{T_1;x}$. Furthermore, since the bounds \eqref{eq:PhiBoundedness} and \eqref{eq:PhiContractionBound} do not depend on $|x|$, we may extend the solution to a further time interval by defining the new ball $\fB_{T_1+T;Y_{T_1}}$ for $\bar{T}>0$ and then repeating the above arguments. Thus this solution can be extended to any interval $[0,T] \subset \bR_+$. Finally, we observe that by extending \eqref{eq:PhiBoundedness} all the way up to $\beta=\eta\wedge \gamma$ we see that the solution constructed above a is $\PP$-measurable mapping $\Omega \ni \omega \mapsto Y^\mu(\omega) \in \cC^{\eta \wedge \gamma}_T$ and furthermore that $\cL(Y^\mu_0)=\cL(\xi)$. So the proof is complete.
  \end{proof}
For later analysis it will be useful to obtain control on the growth of this solution, $Y$ in $\cC^{\eta \wedge \gamma}_T$ and its law in $\cC^{\eta\wedge \gamma}_T\cP_1(\bR^d)$. The following lemma collects these controls. 

\begin{lemma}[Growth Control]\label{lem:SolGrowthBounds}
 Let $\alpha\geq 2$, $\Gamma \in \cC^\gamma_T\cC^\alpha(\bR^d)$, satisfy the assumptions of Lemma \ref{lem:MeasureSewingStability}, $\mu \in \cP_1(\cC^{\eta \wedge \gamma}_T)$ and $(\xi,B) \in L^1(\Omega;\bR^d)\times L^{1}(\Omega;\cC^\eta_T)$ with $B_0=0$. Then writing $Y^\mu$ for the associated solution to \eqref{eq:RDEFixedMeasure}, for any $\beta\in (1-\gamma,\eta\wedge \gamma)$, there exists a deterministic constant $C:=C(\gamma,\eta,\beta)>0$ and a $\theta:=\theta(\gamma,\eta,\beta)>0$ such that for $\PP$-a.a. $\omega\in \Omega$,
 \begin{equation}\label{eq:SolGrowthBound}
    [Y^\mu]_{\beta;T}(\omega) \leq CT^\theta\left(1+ [\mu]_{\beta;T}+[B]_{\eta;T}(\omega)\right)\left(1\vee \|\Gamma\|_{\gamma,\alpha}^2\right).
 \end{equation}
In addition, for a new, deterministic constant $C:=C(\gamma,\eta,\beta)>0$, we have that
 \begin{equation}\label{eq:SolLawGrowthBound}
   [\cL(Y^\mu)]_{\beta;T} \leq CT^\theta\left(1+ [\mu]_{\beta;T}+\EE[[B]_{\eta;T}]\right)\left(1\vee \|\Gamma\|_{\gamma,\alpha}^2\right).
 \end{equation}
\end{lemma}
\begin{proof}
 We begin by proving \eqref{eq:SolGrowthBound}. Let $h \in (0,1\wedge T)$ and $t \in [0,T-h]$ then for any $s \in [t,t+h]$, we have
 \begin{equation*}
   |Y_{t,s}^\mu|\leq \left|\int_s^t (\Gamma_{\dd r} \ast \mu_r)(Y^\mu_r)\,\right|+|B_{s,t}|. 
 \end{equation*}
 From \eqref{eq:bound on integral} we have that 
 \begin{equation*}
   |Y_{t,s}^\mu|\, \lesssim_{\gamma,\beta} |t-s|^{\gamma+\beta} \|\Gamma\|_{\gamma,\alpha}\left([Y^\mu]_{\beta;[s,t]}+[\mu]_{\beta;T}\right) + |t-s|^{\gamma} \|\Gamma\|_{\gamma,\alpha}+|B_{s,t}|.
 \end{equation*}
Taking the $\beta$-H\"older semi-norm, over $[t,t+h]$, on both sides we obtain that
 \begin{equation*}
   [Y^\mu]_{\beta;[t,t+h]}\lesssim_{\gamma,\beta} h^{\gamma}\|\Gamma\|_{\gamma,\alpha}[Y^\mu]_{\beta;[t,t+h]}+T^{\gamma}\|\Gamma\|_{\gamma,\alpha}[\mu]_{\beta;T}+T^{\gamma-\beta}\|\Gamma\|_{\gamma,\alpha}+T^{\eta-\beta}[B]_{\eta;T}. 
 \end{equation*}
 We define $\bar{h}>0$ by the formula
 \begin{equation}\label{eq:barh}
   \bar{h}:=\left(\frac{1}{2C\|\Gamma\|_{\gamma,\alpha}}\right)^{\frac{1}{\gamma}}\wedge 1\wedge T, 
 \end{equation}
 with $C>0$ the implied proportionality constant above and so it follows that, 
 \begin{align*}
   [Y^\mu]_{\beta;[t,t+\bar{h}]}   &\,\lesssim_{\gamma,\beta} T^{\gamma}\|\Gamma\|_{\gamma,\alpha}[\mu]_{\beta;T}+T^{\gamma-\beta}\|\Gamma\|_{\gamma,\alpha}+T^{\eta-\beta}[B]_{\eta;T}.
 \end{align*}
Applying Lemma \ref{lem:LocalToGlobalHolderBound} there exists a $C:=C(\gamma,\eta,\beta)>0$ and a $\theta:=\theta(\gamma,\eta,\beta)>0$ such that
 \begin{equation*}
   [Y^\mu]_{\beta,T} \leq CT^\theta\left(1+ [\mu]_{\beta;T}+[B]_{\eta;T}\right)\left(1\vee \|\Gamma\|_{\gamma,\alpha}^{1+\frac{1-\beta}{\gamma}}\right).
 \end{equation*}
Since $0<\frac{1-\beta}{\gamma}<1$, \eqref{eq:SolGrowthBound} follows. To prove \eqref{eq:SolLawGrowthBound}, first we observe that 
 \begin{align}
   \left[\cL(Y^\mu) \right]_{\beta;T} &= \sup_{\varphi \in \lip_1(\bR^d)} \sup_{t\,\neq \,s \in [0,T]}\frac{1}{|t-s|^{\beta}}\left|\int_{\bR^d} \varphi(y)\,\dd \left( \cL(Y^\mu_t)-\cL(Y^\mu_s)\right)\,(y) \right|\notag\\
   &= \sup_{\varphi \in \lip_1(\bR^d)} \sup_{t\,\neq \,s \in [0,T]}\frac{1}{|t-s|^{\beta}}|\EE\left[\varphi(Y^\mu_t)-\varphi(Y^\mu_s) \right]|\notag\\
   &\leq \EE\left[[Y^\mu]_{\beta;T} \right]\label{eq:law to expectation},
 \end{align}
 where in the penultimate line we used that $\varphi \in \lip_1(\bR^d)$ and Jensen's inequality. Applying \eqref{eq:SolGrowthBound} inside the expectation gives \eqref{eq:SolLawGrowthBound}.
\end{proof}

\begin{remark}\label{rem:SolMomentBounds}
Taking $\theta=0$ and $\beta=\eta \wedge \gamma$ in \eqref{eq:SolGrowthBound} shows that $Y^\mu \in L^1(\Omega;\cC^{\eta \wedge \gamma}_T)$. Obtaining higher moment bounds follows in the same vein, only requiring us to assume that $(\xi,B)\in L^p(\Omega;\bR^d)\times L^p(\Omega; \cC^\eta_T)$ and $\mu \in \cP_p(\cC^{\eta \wedge \gamma}_T)$ for the same $p\geq 1$. 
\end{remark}

\subsection{Stability of Frozen Measure Flow Solutions}
We define the solution map for the frozen measure flow equation, \eqref{eq:RDEFixedMeasure}, for every $\mu \in \cP_1(\cC^{\eta \wedge \gamma}_T)$, setting
\begin{equation}\label{eq:SolMapDefine}
	\begin{aligned}
		S^\mu_T : \bR^d \times \cC^\eta_T\times &\rightarrow \cC^{\eta\wedge \gamma}_T\\
		(\xi,B)&\mapsto Y^\mu, 
	\end{aligned}
\end{equation}
where $Y^\mu$ is the solution to the NLYE, \eqref{eq:RDEFixedMeasure}, constructed in Theorem~\ref{th:RDEWellPosedFixedMeasure}.
\begin{lemma}[Stability Control]\label{lem:SolStabilityBounds}
Let $\alpha \geq 2$, $\Gamma\in \cC^\gamma_T\cC^\alpha(\bR^d)$, satisfy the assumptions of Lemma~\ref{lem:MeasureSewingStability}, $\beta \in(1-\gamma,\eta \wedge \gamma)$ and $p= \frac{2-\beta}{\gamma-\beta}$. Then let $(\xi,B,\mu),\,(\tilde{\xi},\tilde{B},\tilde{\mu}) \in L^{1}(\Omega;\bR^d) \times L^{2p}(\Omega;\cC^\eta_T)\times \cP_1(\cC^{\eta\wedge \gamma}_T)$ be two pairs of input triples, such that $(\xi-\tilde{\xi})\perp (B,\tilde{B})$, 
$B_0=\tilde{B}_0=0$ and suppose there exists a constant $M>0$ such that
 \begin{equation*}
   [\mu]_{\beta;T}\vee[\tilde{\mu}]_{\beta;T} \leq M.
 \end{equation*}
Then setting $Y := S^\mu_T(\xi,B)$, $\tilde{Y}:=S^{\tilde{\mu}}_T(\tilde{x}_0,\tilde{B})$ and defining the strictly positive random variable
\begin{equation*}
  \fG:= (1+M+[B]_{\eta;T}\vee[\tilde{B}]_{\eta;T}),
\end{equation*}
there exists a constant $C:= C(\gamma,\beta$,$\|\Gamma\|_{\gamma,\alpha})>0$ such that for  $\PP$-a.e. $\omega\in \Omega$,
 \begin{equation}\label{eq:SolStabilityBound}
  	\|Y-\tilde{Y}\|_{\beta;T} \leq C\fG^{p}T^{1-\beta}\left([B-\tilde{B}]_{\beta;T} +\vertiii{\mu;\tilde{\mu}}_{\beta;T}+|\xi-\tilde{\xi}|\right)
 \end{equation}
 Furthermore, we have that
 \begin{equation}\label{eq:SolLawStabilityBound}
 \begin{aligned}
   \vertiii{\cL(Y);\cL(\tilde{Y})}_{\beta;T} \leq \,&\,CT^{1-\beta}\left( \EE\left[\fG^p\right] \vertiii{\mu;\tilde{\mu}}_{\beta;T} +\EE\left[\fG^{2p}\right]^{\frac{1}{2}}\EE\left[ [B-\tilde{B}]^2_{\beta;T}\right]^{\frac{1}{2}}\right) \\
   & +(1+CT^{2-\beta}\EE\left[\fG^{p}\right] )\EE\left[|\xi-\tilde{\xi}|\right].
   \end{aligned}
 \end{equation}
\end{lemma}
\begin{proof}
Let  $h\in (0,1)$,  $\tau >0$ and $s<t \in [\tau,\tau+h]$. Then observe that
	 \begin{equation}\label{eq:YDiffTriangle1}
		|Y_{s,t}-\tilde{Y}_{s,t}| \,\leq  \left|\,\int_s^t\Gamma_{\dd r}\ast \mu_r(Y_r)-\Gamma_{\dd r}\ast \tilde{\mu}_r(\tilde{Y}_r)\,\right|+|B_{s,t}-\tilde{B}_{s,t}|.
	\end{equation}
Applying Corollary \ref{cor:non-linearYoungStabillity}, and using the definition of $\fG$, and $\theta>0$ from Lemma \ref{lem:SolGrowthBounds},
\begin{equation*}
\begin{aligned}
\bigg[\int \Gamma_{\dd r}\ast \mu_r(Y_r)-\Gamma_{\dd r}\ast \tilde{\mu}_r(\tilde{Y}_r)
   \bigg]_{\beta;[\tau,\tau+h]}
   &\lesssim_{\gamma,\beta} h^{\gamma}(1\vee \|\Gamma\|_{\gamma,\alpha}^3)\fG \left( \|Y-\tilde{Y}\|_{\beta;[\tau,\tau+h]}+\vertiii{\mu;\tilde{\mu}}_{\beta;[\tau,\tau+h]}\right)\\
   & \quad + h^{\gamma-\beta}\|\Gamma\|_{\gamma,\alpha}\left(\|Y-\tilde{Y}\|_{\infty;[\tau,\tau+h]}+\sup_{s\in [\tau,\tau+h]}\|\mu_s-\tilde{\mu}_s\|_{\KR}\right).
   \end{aligned}
\end{equation*}
Using the standard bound for supremum and H\"older norms, we have that 
\begin{equation*}
    \|Y-\tilde{Y}\|_{\infty;[\tau,\tau+h]} \lesssim |Y_\tau-\tilde{Y}_\tau|+[Y-\tilde{Y}]_{\beta;[\tau,\tau+h]}=\|Y-\tilde{Y}\|_{\beta;[\tau,\tau+h]}, 
\end{equation*}
with a similar bound holding for the term $\sup_{s\in [\tau,\tau+h]}\|\mu_s-\tilde{\mu}_s\|_{\KR}$. Thus, using that $\beta<\gamma$, we obtain 
\begin{equation*}
\bigg[\int \Gamma_{\dd r}\ast \mu_r(Y_r)-\Gamma_{\dd r}\ast \tilde{\mu}_r(\tilde{Y}_r)
   \bigg]_{\beta;[\tau,\tau+h]}
   \lesssim_{\gamma,\beta} h^{\gamma-\beta}(1\vee \|\Gamma\|_{\gamma,\alpha}^3)\fG \left( \|Y-\tilde{Y}\|_{\beta;[\tau,\tau+h]}+\vertiii{\mu;\tilde{\mu}}_{\beta;[\tau,\tau+h]}\right),
\end{equation*}
where we have used that $h<1$. 
Inserting this bound in \eqref{eq:YDiffTriangle1} and using that $\eta> \beta$,  it follows that there exists a constant $C$ depending on $\gamma $ and $\beta$ such that
\begin{equation}\label{eq:boundsmallh}
    [Y-\tilde{Y}]_{\beta;[\tau,\tau+h]} 
    \\
    \leq  [B-\tilde{B}]_{\beta;[\tau,\tau+h]} +Ch^{\gamma-\beta}(1\vee \|\Gamma\|_{\gamma,\alpha}^3)\fG \left( \|Y-\tilde{Y}\|_{\beta;[\tau,\tau+h]}+\vertiii{\mu;\tilde{\mu}}_{\beta;[\tau,\tau+h]}\right).
\end{equation}
Define the random variable 
\begin{equation}\label{eq:rv}
    C(\Gamma,\fG) := C(1\vee \|\Gamma\|_{\gamma,\alpha}^3)\fG
\end{equation}
Then using the fact that $\|Y-\tilde{Y}\|_{\beta;[\tau,\tau+h]} = |Y_\tau-\tilde{Y}_\tau| +[Y-\tilde{Y}]_{\beta;[\tau,\tau+h]}$ we have,
\begin{equation*}
     [Y-\tilde{Y}]_{\beta;[\tau,\tau+h]} 
    \\
    \leq  [B-\tilde{B}]_{\beta;[\tau,\tau+h]} +h^{\gamma-\beta}C(\Gamma,\fG) \left( |Y_\tau-\tilde{Y}_\tau| +[Y-\tilde{Y}]_{\beta;[\tau,\tau+h]}+\vertiii{\mu;\tilde{\mu}}_{\beta;[\tau,\tau+h]}\right).
\end{equation*}
Choosing $h=\bar{h}$ specifically such that,
\begin{equation*}
    \bar{h} = \left(\frac{1}{4C(\Gamma,\fG) } \right)^{\frac{1}{\gamma-\beta}}\wedge 1\wedge T
\end{equation*}
it is readily seen that for $\tau \in [0,T-\bar{h}]$,
\begin{equation}\label{eq:small time contraction}
     [Y-\tilde{Y}]_{\beta;[\tau,\tau+\bar{h}]} 
    \\
    \leq  \frac{4}{3}[B-\tilde{B}]_{\beta;T} +\frac{1}{3}|Y_\tau-\tilde{Y}_\tau| +\frac{1}{3}\vertiii{\mu;\tilde{\mu}}_{\beta;T}.
\end{equation}
where we have used that $\|\cdot\|_{\beta;[\tau,\tau+\bar{h}]}\leq \|\cdot\|_{\beta;T}$ for any $[\tau,\tau+\bar{h}]\subset[0,T]$. 
Note that if $\bar{h}=T$, it would follow from this inequality that $\mu\mapsto Y$ is a contraction on $[0,T]$. As $\bar{h}$ is a random variable we must also investigate the case when $\bar{h}<T$, combining with certain iteration techniques to obtain a global bound. 
We use the identity $|X_{\tau}| \leq |X_{\tau-\bar{h}}|+ \bar{h}^\beta[X]_{\beta;[\tau-\bar{h},\tau]}$ to find that,
\begin{align*}
    |Y_{\tau}-\tilde{Y}_{\tau}| &\leq |Y_{\tau-\bar{h}}-\tilde{Y}_{\tau-\bar{h}}|+ \bar{h}^\beta[Y-\tilde{Y}]_{\beta;[\tau-\bar{h},\tau]}
    \\
    &\leq (1+\frac{\bar{h}^\beta}{3})|Y_{\tau-\bar{h}}-\tilde{Y}_{\tau-\bar{h}}| + \bar{h}^\beta \left( \frac{4}{3}[B-\tilde{B}]_{\beta;T}  +\frac{1}{3}\vertiii{\mu;\tilde{\mu}}_{\beta;T}\right)
\end{align*}
Iterating this inequality, for any $\tau\in [0,T-\bar{h}]$ one need at most to iterate $\lceil \frac{T}{\bar{h}} \rceil$ times, and thus we obtain the bound
\begin{equation*}
    |Y_{\tau}-\tilde{Y}_{\tau}| \leq    \left(1+\frac{\bar{h}^\beta}{3}\right)^{\lceil\frac{T}{\bar{h}}\rceil}|\xi-\tilde{\xi}|+ \left(1+\frac{\bar{h}^\beta}{3}\right)^{\lceil\frac{T}{\bar{h}}\rceil-1}\bar{h}^\beta \left( \frac{4}{3}[B-\tilde{B}]_{\beta;T}  +\frac{1}{3}\vertiii{\mu;\tilde{\mu}}_{\beta;T}\right).
\end{equation*}

Combining this bound with \eqref{eq:small time contraction} we see that 
\begin{align*}
    \|Y-\tilde{Y}\|_{\beta;[\tau,\tau+\bar{h}]}&\leq \frac{1}{3}( \left(1+\frac{\bar{h}^\beta}{3}\right)^{\lceil\frac{T}{\bar{h}}\rceil}|\xi-\tilde{\xi}|+ \left(1+\frac{\bar{h}^\beta}{3}\right)^{\lceil\frac{T}{\bar{h}}\rceil-1}\bar{h}^\beta \left( \frac{4}{3}[B-\tilde{B}]_{\beta;T}  +\frac{1}{3}\vertiii{\mu;\tilde{\mu}}_{\beta;T}\right))
    \\
    &+\frac{4}{3}[B-\tilde{B}]_{\beta;T}  +\frac{1}{3}\vertiii{\mu;\tilde{\mu}}_{\beta;T}
    \\
    &\leq2\left(1+\frac{\bar{h}^\beta}{3}\right)^{\lceil\frac{T}{\bar{h}}\rceil}\left(|\xi-\tilde{\xi}|+ [B-\tilde{B}]_{\beta;T}+\vertiii{\mu;\tilde{\mu}}_{\beta;T}\right).
\end{align*}

An application of Lemma \ref{lem:LocalToGlobalHolderBound} then reveals that 

\begin{equation*}
    	\|Y-\tilde{Y}\|_{\beta;T} \leq  2\left(1+\frac{\bar{h}^\beta}{3}\right)^{\lceil\frac{T}{\bar{h}}\rceil}\left(|\xi-\tilde{\xi}|+ [B-\tilde{B}]_{\beta;T}+\vertiii{\mu;\tilde{\mu}}_{\beta;T}\right)(1\vee 2\bar{h}^{\beta-1})T^{1-\beta}
\end{equation*}
%
%
Note that  for any $k\geq 1$ we have  $(1+\frac{\bar{h}^\beta}{3})^{k}\leq k(1+(\frac{\bar{h}^\beta}{3})^k)$, and furthermore note that $\bar{h}\leq 1$. Thus 
by definition of  $\bar{h}$ we obtain that 
\begin{equation*}
    \|Y-\tilde{Y}\|_{\beta;T} \leq  4  (1+T 4C(\Gamma,\fG))^{\frac{2-\beta}{\gamma-\beta}}  \left(|\xi-\tilde{\xi}|+ [B-\tilde{B}]_{\beta;T}+\vertiii{\mu;\tilde{\mu}}_{\beta;T}\right)T^{1-\beta},
\end{equation*}
which provides us with the desirable pathwise bound. 
Recall that due to our assumptions on $B,\,\tilde{B}$, we have $\fG\in L^{2p}(\Omega;\cC^\eta_T)$.

With the pathwise estimate in hand, we show stability in the law of $Y$.
By the same argument as used in the derivation of inequality \eqref{eq:law to expectation}, we have that
\begin{equation}\label{eq:stability law first bound}
  \vertiii{\cL(Y);\cL(\tilde{Y})}_{\beta;T}\leq \EE\left[|\xi-\tilde{\xi}| \right] + \EE\left[\|Y-\tilde{Y}\|_{\beta;T} \right].
\end{equation}
Then applying \eqref{eq:SolStabilityBound} inside the second expectation, using that $(x-\tilde{x})\perp \fG$ due to the fact that $x,\,\tilde{x}$ are independent from $(B,\,\tilde{B})$, and applying H\"older's inequality, we obtain \eqref{eq:SolLawStabilityBound}.
\end{proof}

\subsection{McKean--Vlasov Fixed Point} \label{subsec:MKVFixedPoint}
We now show that we can close the fixed point $\mu=\cL(Y^\mu)$ and in doing so obtain a solution $(Y,\mu)$ to the full abstract McKean--Vlasov problem \eqref{eq:AbstractNLYEMKV}. For $p\geq 1$ and any $\nu:=\cL(\xi,B)\in \cP_{1,p}(\bR^d\times\cC^\eta_T)$ with $B_0=0$, we define the map 
\begin{equation}\label{eq:PsiDefine}
\begin{aligned}
  \Psi(\nu,\,\cdot\,): \cP_1\left(\cC^{\eta\wedge \gamma}_T\right)&\rightarrow \cP_1\left(\cC^{\eta\wedge \gamma}_T\right)\\
  \mu &\mapsto \cL(Y^\mu):= S^\mu_T\#\cL(\xi,B),
  \end{aligned}
\end{equation}
for $S^\mu_T$ as defined in \eqref{eq:SolMapDefine}. Given this set up, we prove the following theorem.

\begin{theorem}\label{th:AbstractMKVSol}
Let $\alpha\geq 2$, $\Gamma \in \cC^\gamma_T\cC^\alpha(\bR^d)$, satisfy the assumptions of Lemma \ref{lem:MeasureSewingStability}, $\beta \in (1-\gamma,\eta \wedge \gamma)$, $p= \frac{2-\beta}{\gamma-\beta}$ and $\nu \in \cP_{1,p}(\bR^d\times \cC^\eta_T)$ be as above. Then there exists a unique $\bar{\mu} \in \cP_1(\cC^{\eta \wedge \gamma}_T)$ such that $\Psi(\nu,\bar{\mu})=\bar{\mu}$. As a result $\bar{\mu}$ solves the McKean--Vlasov problem
\begin{equation}\label{eq:McKeanVlasov fixed point}
  Y_t = \xi + \int_0^t (\Gamma_{\dd r}\ast \bar{\mu}_r)(Y_r)+ B_t,\quad \bar{\mu}=\cL(Y).
\end{equation}
\end{theorem}
\begin{proof}
For any $t \in [0,T]$, $\nu \in \cP_{1,p}(\bR^d\times \cC^\eta_T)$ we define the set
\begin{equation*}
  \fB_t := \left\{ \mu \in \cP_1(\cC^\beta_T)|\,\mu_{0} = \nu|_{\bR^d},\,[\mu]_{\beta;t}\leq 1 \right\}
\end{equation*}
which is a complete metric space under $[\,\cdot\,]_{\beta;t}$ (due to the fact that all elements start in $\mu_{0} = \nu|_{\bR^d}$). We first show that there exists a $T_0\in [0,T]$ such that $\Psi(\nu,\,\cdot\,)$ leaves $\fB_{T_0}$ invariant. We let $(\xi,B)\sim \nu$ and $Y= S^\mu_T(\xi,B)$. From Lemma \ref{lem:SolGrowthBounds}, specifically \eqref{eq:SolLawGrowthBound} we have
\begin{equation*}
 [\cL(Y)]_{\beta;t} \leq Ct^\theta\left(1+\EE\left[[B]_{\eta;T}\right]\right)\left(1\vee \|\Gamma\|^2_{\gamma,\alpha}\right),
\end{equation*}
for some $\theta:=\theta(\gamma,\eta,\beta)>0$ and where we used that $\mu\in \fB_t$ ensures $[\mu]_{\beta;t}\leq 1$. So choosing $T_0>0$ sufficiently small we see that $[\cL(Y)]_{\beta;T_0} \leq 1$. Furthermore, it is immediate from the proof of Theorem \ref{th:RDEWellPosedFixedMeasure} that $\cL(Y_0) =\nu |_{\bR^d}$ so we conclude that $\Psi(\nu,\fB_{T_0})\subseteq \fB_{T_0}$.\\ \par 
Now we show that there exists some $T_1 \in (0,T_0]$ such that $\Psi(\nu,\,\cdot\,)$ is a contraction on $\fB_{T_1}$. Let $\mu^1,\,\mu^2 \in \fB_{T_0}$ and $Y^1=S^{\mu^1}(\xi,B),\,Y^2=S^{\mu^2}(\xi,B)$ be distinct. Then, following the exact same procedure as leading to \eqref{eq:small time contraction}  in Lemma \ref{lem:SolStabilityBounds} and combining with the estimate in \eqref{eq:stability law first bound} using $\fG := (2+[B]_{\eta;T})$ (note that now $Y^1$ and $Y^2$ both starts in $\xi$ with same random noise $B$, and so the independence condition $(\xi-\xi)\perp (B,B)$ of Lemma \ref{lem:SolStabilityBounds} is trivially satisfied), we have that for all $t\in (0,T_0]$,
\begin{align*}
   \vertiii{\cL(Y);\cL(\tilde{Y})}_{\beta;t} &\lesssim_{\gamma,\eta,\beta}\, t^{\gamma-\beta} \EE\left[C(\Gamma,\fG)^{p}\right] \vertiii{\mu;\tilde{\mu}}_{\beta;t},
\end{align*}
where $C(\Gamma,\fG)$ is defined as in \eqref{eq:rv}. 
Since we assume $B$ is $p$-integrable we can choose $t:=T_1\in (0,T_0]$ sufficiently small to obtain that $\Psi(\nu,\,\cdot\,)$ is a contraction on $\fB_{T_1}$. Applying the Banach fixed point theorem it follows that there exists a unique fixed point $
\bar{\mu} \in \fB_{T_1}$ such that $\Psi(\nu,\bar{\mu})=\bar{\mu}$ on $[0,T_1]$. \\ \par 
We now show that we can extend this solution to the whole interval $[0,T]$. Note that both $T_0$ and $T_1$ were chosen independently of $\nu|_{\bR^d}$, the initial distribution.
From Lemma \ref{lem:SolGrowthBounds} we have that $\bar{\mu}_{T_1} \in \cP_1(\bR^d)$ and so we can define a new family of sets in $\cP_1(\cC^{\beta})$ by setting
\begin{equation*}
  \fB^1_{[T_1,T_1+t]}:= \left\{\mu \in \cP_1\left(\cC^\beta_{[T_1,T_1+T]}\right)\,:\,\mu_{T_1} = \bar{\mu}_{T_1},\,[\mu]_{\beta;[T_1,T_1+t]}\leq 1 \right\},
\end{equation*}
for any $t\in [0,T-T_1]$. We may then repeat the same argument as above, now considering the solution map $\Psi(\bar{\mu}_{T_1} \otimes \nu|_{\cC^\eta_T},\mu):= S_{[T_1,T_1+T]}^\mu(\cL(\bar{\mu}_{T_1},B))$. 
Since $T_1$ was chosen independently of the initial distribution, by the same arguments we obtain a new fixed point $\bar{\mu}\in \fB^1_{[T_1,2T_1]}$. Note that
\begin{equation*}
  \cL(Y^{\bar{\mu}_{2T_1}})=S_{[T_1,2T_1]}^{\bar{\mu}_{2T_1}}\#\bar{\mu}_{T_1}=S^{\bar{\mu}_{2T_2}}_{[T_1,2T_2]}\#S_{[0,T_1]}^{\mu_{T_1}}\#\nu,
\end{equation*}
and thus there exists a unique solution to \eqref{eq:McKeanVlasov fixed point} on $[0,2T_2]$. This procedure can be iterated to any interval $[kT_1,(k+1)T_1\wedge T]\subset [0,T]$, and so we conclude that there exists a unique solution to \eqref{eq:McKeanVlasov fixed point} on $[0,T]$. \\ \par
Finally, using Theorem \ref{th:RDEWellPosedFixedMeasure} we see that $Y^{\bar{\mu}}\in \cC^{\eta \wedge \gamma}_T$ and so from Lemma \ref{lem:SolGrowthBounds} and Remark \ref{rem:SolMomentBounds} we have that $\bar{\mu}\in \cP_1(\cC^{\eta \wedge \gamma}_T)$. This concludes the proof. 
\end{proof}
We now define the fixed point map
\begin{equation}\label{eq:PsiBarDefine}
	\begin{aligned}
		\bar{\Psi}: \cP_{1,2p}(\bR^d\times \cC^\eta_T)  &\rightarrow \cP_1\left(\cC^{\eta\wedge \gamma}_T\right)\\
		\nu &\mapsto \bar{\mu} := \bar{\Psi}(\nu,\bar{\mu}), 
	\end{aligned}
\end{equation}
for $p\geq 1$, where $\bar{\mu} = \Psi(\nu,\bar{\nu})$, with $\Psi$ defined by \eqref{eq:PsiDefine}.
\subsection{Stability of the Fixed Point Law}
In this section we investigate the stability of the solution to \eqref{eq:McKeanVlasov fixed point} with respect to the joint law of the initial data and the driving noise.\\ \par 
First we introduce some notation. Given two probability spaces, $(\Omega^1,\cF^1,\PP^1),\,(\Omega^2,\cF^2,\PP^2)$ we write $\EE^1$ (resp. $\EE^2$) for the expectation over $\Omega^1$ (resp. $\Omega^2$) with respect to $\PP^1$ (resp. $\PP^2$) and $\EE^{1,2}$ for the expectation over $\Omega^1\times \Omega^2$ with respect to $\PP^1\times \PP^2$.
\begin{theorem}\label{th:AbstractMKVSolStable}
  Let $\alpha\geq 2$, $\Gamma \in \cC^\gamma_T\cC^\alpha(\bR^d)$, satisfy the assumptions of Lemma \ref{lem:MeasureSewingStability}, $\beta \in (1-\gamma,\eta \wedge \gamma)$, $p= \frac{2-\beta}{\gamma-\beta}$ and two possibly different probability spaces $(\Omega^1,\cF^1,\PP^1),\,(\Omega^2,\cF^2,\PP^2)$. Assume that we are given $(\xi^1,B^1) \in L^1(\Omega^1;\bR^d)\times L^{2p}(\Omega^1;\cC^\eta_T)$ and $(\xi^2,B^2) \in L^1(\Omega^2;\bR^d)\times L^{2p}(\Omega^2;\cC^\eta_T)$
  such that $B^1_0=B^2_0=0$ and $(\xi^1-\xi^2)\perp (B^1,B^2)$, with $\nu^i= \cL(\xi^i,B^i) \in \cP_{1,2p}(\bR^d\times\cC^{\eta}_T)$ for $i=1,2$, and letting $\mu^1,\mu^2 \in \cP_1(\cC^{\eta\wedge \gamma}_T)$ be the associated solutions to \eqref{eq:AbstractNLYEMKV}, then there exists a constant $C:=C\left(\EE^{1}\left[[B^1]^{2p}_{\eta;T}\right]\vee \EE^{2}\left[[B^2]^{2p}_{\eta;T}\right],\Gamma,\gamma,\eta,\beta\right)>0$ 
  such that
  \begin{equation}\label{eq:LawFixedPointGlobalStable}
    \cW_{1;\cC^\beta_{T}}(\mu^1,\mu^2) \leq C\left( \cW_{1;\bR^d}(\cL(\xi^1),\cL(\xi^2))+\cW_{2; \cC^\eta_T}(\cL(B^1),\cL(B^2))\right).
   \end{equation}
\end{theorem}
\begin{proof}
 We let $\mu^i = \bar{\Psi}^i(\nu^i)$, where $\bar{\Psi}$ is defined by \eqref{eq:PsiBarDefine}. By Point \ref{it:WassMinimiser} of Proposition \ref{prop:PortmanteauTheorem} there exist optimal transport plans $(m_0,m)\in \Pi(\nu^1|_{\bR^d},\nu^2|_{\bR^d})\times \Pi(\nu^1|_{\cC^\eta_T},\nu^2|_{\cC^{\eta}_T})$ such that
 \begin{align*}
   \cW_{1;\bR^d}(\nu^1|_{\bR^d},\nu^2|_{\bR^d})+ \cW_{2;\cC^\eta_T}(\nu^1|_{\cC^\eta_T},\nu^2|_{\cC^{\eta}_T})=\EE_{m_0}\left[|\xi^1-\xi^2|\right] +\EE_{m}\left[ [B^1-B^2]^2_{\eta;T}\right]^{\frac{1}{2}}.
 \end{align*}
Defining
$$\bar{m}:= m_0 \otimes m \in \cP_1\left((\bR^d \times \bR^d)\times(\cC^\eta_T\times \cC^\eta_T) \right),$$
using the definition of $1$-Wasserstein distance on $\cC^\beta_t$, and since $\cP_1(\cC^\eta_t)\subset\cP_1(\cC^\beta_t)$, we have that
 \begin{equation}\label{eq:LawWassersteinDistance1}
   \cW_{1;\cC^{\beta}_t}(\cL(Y^1),\cL(Y^2)) \leq \EE_{\bar{m}}\left[\|Y^1-Y^2\|_{\beta;t} \right], \quad \text{ for any } t \in [0,T]
 \end{equation}
 Using Lemma \ref{lem:SolGrowthBounds}, specifically \eqref{eq:SolLawGrowthBound}, for $i=1,2$, any $t\in (0,T]$, and some $\theta:=\theta(\gamma,\eta,\beta)>0$, we have that
 \begin{equation*}
 	[\mu^i]_{\beta;t} \leq Ct^\theta\left(1+ [\mu^i]_{\beta;t}+\EE^i\left[[B^i]_{\eta;T}\right]\right)\left(1\vee \|\Gamma\|^2_{\gamma,\alpha}\right).
 \end{equation*}
 So now, choosing $t=T_0\in (0,T]$ defined by
 \begin{equation*}
   T_0:= \left(\frac{1}{2C\left(1\vee \|\Gamma\|_{\gamma,\alpha}\right)}\right)^{\frac{1}{\theta}},
 \end{equation*}
 we see that, for a new constant $C:=C(\Gamma,\gamma,\beta)>0$
 \begin{equation}\label{eq:LawFixedPointGrowth}
   [\mu^i]_{\beta;T_0} \leq C\left(1+\EE^i\left[\left[B^i\right]_{\eta;T}\right]\right),\quad \text{for } i =1,2.
 \end{equation}
 We now define the strictly positive random variable on $(\Omega^1,\cF^1,\PP^1)\times (\Omega^2,\cF^2,\PP^2)$,
\begin{equation*}
	\fG:= \left(1+\EE^1\left[\left[B^1\right]_{\eta;T}\right] \vee\EE^2\left[\left[B^2\right]_{\eta;T}\right]+[B^1]_{\beta;T}\vee[B^2]_{\beta;T}\right)\geq 1.
\end{equation*}
Note in particular that both $\fG$ and $T_0$ are independent of initial data $\xi^1,\,\xi^2$. So, using that $(\xi^1-\xi^2)\perp_m \fG$, and applying \eqref{eq:SolStabilityBound} of Lemma \ref{lem:SolStabilityBounds}, for any $t\in [0,T_0]$, we have that
\begin{align*}
	 \EE_{\bar{m}}\left[		\|Y^1-Y^2\|_{\beta;t}\right] \lesssim_{\gamma,\eta,\beta,\alpha, \Gamma} \,&\, t^{2-\beta} \EE_{m}\left[\fG^{p}\right] \vertiii{\mu^1;\mu^2}_{\beta;t}
 +T^{2-\beta}\EE_{m}\left[\fG^{2p}\right]^{\frac{1}{2}} \EE_{m}\left[	 [B^1-B^2]^2_{\eta;T}\right]^{\frac{1}{2}}\\
	 &+ \EE_{m}\left[\fG^{p}\right] \EE_{m_0}\left[	|\xi^1-\xi^2|\right].
\end{align*}
Note that $\EE_m$ denotes an integration over the product space $\Omega^1\times \Omega^2$. From Theorem \ref{th:HolderFlowEmbedding}, we have that
 \begin{align*}
  \vertiii{\mu^1;\mu^2\,}_{\beta;t}\leq \cW_{1;\cC^\beta_t}(\mu^1,\mu^2),
 \end{align*}
so in turn
 \begin{align*}
   \EE_{\bar{m}}\left[\|Y^1-Y^2\|_{\beta;t}\right] &\lesssim_{\gamma,\eta,\beta,\alpha,\Gamma} t^{2-\beta} \EE_{m}\left[\fG^{p}\right] \cW_{1;\cC^\beta_t}(\mu^1,\mu^2)
 +T^{2-\beta} \EE_{m}\left[\fG^{2p}\right]^{\frac{1}{2}} \EE_{m}\left[	 [B^1-B^2]^2_{\eta;T}\right]^{\frac{1}{2}}\\
   &\hspace{4em} + \EE_{m}\left[\fG^{p}\right] \EE_{m_0}\left[	|\xi^1-\xi^2|\right] ,
 \end{align*}
 where $\sigma =\sigma(\gamma,\eta,\beta)>0$ is the same as in Lemma \ref{lem:SolStabilityBounds}. Now choose $t=T_1 \in (0, 1\wedge T_0)$ according to
\begin{equation*}
  T_1 := \left(\frac{1}{2C(1\vee\|\Gamma\|^4_{\gamma,\alpha}) \EE_{m}\left[\fG^{p}\right]}\right)^{\frac{1}{2-\beta}}\wedge 1\wedge T_0,
\end{equation*}
where $C:=C(\gamma,\eta,\beta)>0$ is the proportionality constant above. So then using \eqref{eq:LawWassersteinDistance1}, we have, for a new constant $C:=C(T,\Gamma,\gamma,\eta,\beta)>0$,
\begin{align*}
  \cW_{1;\cC^\beta_{T_1}}(\mu^1,\mu^2) &\leq C \left( \EE_{m}\left[\fG^{2p}\right]^{\frac{1}{2}} \EE_{m}\left[	 [B^1-B^2]^2_{\eta;T}\right]^{\frac{1}{2}}  + \EE_{m}\left[\fG^{p}\right] \EE_{m_0}\left[	|\xi^1-\xi^2|\right]\right).
\end{align*}
So then we chose $\bar{m}= m_0 \otimes m$ to be the optimal transport for $\nu^1,\,\nu^2$, we have that
\begin{equation}\label{eq:FixedPointLawStable1}
	\begin{aligned}
		 \cW_{1;\cC^\beta_{T_1}}(\mu^1,\mu^2) &\leq C \EE_{m}\left[\fG^{2p}\right]^{\frac{1}{2}}\bigg( \cW_{1;\bR^d}(\nu^1|_{\bR^d},\nu^2|_{\bR^d})+\cW_{2; \cC^\eta_T}(\nu^1|_{\cC^\eta_T},\nu^2|_{\cC^\eta_T})\bigg),
	\end{aligned}
\end{equation}
where we used the ordering of moments for the second expectation. Since $T_0,\,T_1$ were chosen independently of $\nu^1|_{\bR^d}$ and $\nu^2|_{\bR^d}$, we can iterate this procedure to find that on the interval $[T_1,2T_1]$, where now $\nu^1:=\cL(Y_{T_1}^1,B^1)\in \cP_{1,2p}(\bR^d\times \cC^\eta_T)$ and $\nu^2:=\cL(Y_{T_1}^2,B^2)\in \cP_{1,2p}(\bR^d\times \cC^\eta_T)$ we have
\begin{equation}\label{eq:stability over small time}
\cW_{1;\cC^\beta_{[T_1,2T_1]}}(\mu^1,\mu^2) \leq C\EE_{m}\left[\fG^{2p}\right]^{\frac{1}{2}} \left( \cW_{1;\bR^d}(\cL(Y^1_{T_1}),\cL(Y^2_{T_1}))+\cW_{2; \cC^\eta_T}(\nu^1|_{\cC^\eta_T},\nu^2|_{\cC^\eta_2})\right).
\end{equation}
From Lemma \ref{lem:WassersteinTaylor} we have that 
\begin{equation*}
  \cW_{1;\bR^d}(\cL(Y^1_{T_1}),\cL(Y^2_{T_1}))\leq \cW_{1;\bR^d}(\cL(\xi^1),\cL(\xi^2))+T_1^\beta\cW_{1;\cC^{\beta}_{T_1}}(\mu^1,\mu^2) 
\end{equation*}
 so inserting \eqref{eq:FixedPointLawStable1} in the above inequality yields that, for a new constant $C:=C(T,T_1,\Gamma,\gamma,\eta,\beta)$, we have
\begin{equation*}
   \cW_{1;\cC^\beta_{[T_1,2T_1]}}(\mu^1,\mu^2) \leq C \EE_m[\fG^{2p}]\left( \cW_{1;\bR^d}(\cL(\xi^1),\cL(\xi^2))+\cW_{2; \cC^\eta_T}(\nu^1|_{\cC^\eta_T},\nu^2|_{\cC^\eta_2})\right). 
\end{equation*}
We can repeat this procedure for any interval $[kT_1,(k+1)T_1]\subset [0,T]$, to give,
\begin{equation*}
   \cW_{1;\cC^\beta_{[kT_1,(k+1)T_1]}}(\mu^1,\mu^2) \leq C \EE_m[\fG^{2p}]^{\frac{k+1}{2}}\left( \cW_{1;\bR^d}(\cL(\xi^1),\cL(\xi^2))+\cW_{2; \cC^\eta_T}(\nu^1|_{\cC^\eta_T},\nu^2|_{\cC^\eta_2})\right). 
\end{equation*}
Since there are only finitely many intervals of this kind inside $[0,T]$ this estimate can be made uniform in $k$ and using the continuity of the measure flow, also to any sub-interval of $[0,T]$ of length $T_1$. Therefore, Lemma \ref{lem:WassLocalToGlobal} implies that there exists a constant $C=C(T,T_1,\EE_m[\fG^{2p}],\Gamma,\gamma,\eta,\beta)>0$ such that 
\begin{equation*}
  \cW_{1;\cC^\beta_{T}}(\mu^1,\mu^2) \leq C\left( \cW_{1;\bR^d}(\cL(\xi^1),\cL(\xi^2))+\cW_{2; \cC^\eta_T}(\cL(B^1),\cL(B^2))\right),
\end{equation*}
which concludes the proof. 
\end{proof}
\section{Mean Field Limit of the Abstract Particle System}\label{sec:AbstractMeanFieldLimit}
We apply the results of the previous section to show convergence of the non-linear Young particle system,
\begin{equation}\label{eq:GenParticleSystem}
  Y^i_t = \xi^i + \frac{1}{N}\sum_{j=1}^N\int_0^t \ \Gamma_{\dd r}(Y^i_r-Y^j_r) + B^i_t, 
\end{equation}
for $i=1,\ldots,N$ to \eqref{eq:AbstractNLYEMKV}. As in \cite{coghi_deuschel_friz_maurelli_20} we only assume convergence in law of the idiosyncratic noise vectors $(B_t^i)_{i=1,\ldots,N}$ to some $B_t$. That is we do not require any independence or exchangeability of the vectors $(B_t^i)_{i=1,\ldots,N}$, although of course either or both could be ingredients to showing the convergence. Our approach makes use of a trick of Tanaka, \cite{tanaka_84}, which was also employed in \cite{coghi_deuschel_friz_maurelli_20}. The idea is to re-cast the mean field approximation as a stability result by a transformation of the underlying probability space.\\ \par
As in the introduction to \cite{cass_lyons_14} and Section~3 of \cite{coghi_deuschel_friz_maurelli_20}, we begin by building, for any $N\geq 1$, the probability space $(\Omega_N,\cF_N,\PP_N)$, by setting,
\begin{align*}
  \Omega_N :=\{1,\ldots,N\},\quad \cF_N:= 2^{\Omega_N},\quad\PP_N:=\frac{1}{N}\sum_{i=1}^N\delta_i.
\end{align*}
where $2^{\Omega_N}$ is the power set of $\Omega_N$ and $\delta_i$ is the Kronecker delta. So we can easily identify any $N$-tuple, $(Y^i)_{i=1,\ldots,N} \subset E^N$ with a random variable $Y^{(N)}:\Omega_N \rightarrow E$ defined such that $Y^{(N)}(i)= Y^i$. Furthermore, the law of $Y^{(N)}$ as an $E$ valued random variable, on $(\Omega_N,\cF_N,\PP_N)$ is given by the empirical measure,
\begin{equation*}
  \cL_N\left(Y^{(N)}\right) := \frac{1}{N}\sum_{i=1}^N\delta_{Y_i},
\end{equation*}
where the delta is now a Dirac mass on $E$. Using this construction we can associate to the random vectors $$(\xi^{(N)},B^{(N)}) = ((\xi^1,B^1),\ldots,(x^N,B^N))\in (\bR^d\times \cC^\eta_T)^N \text{ and } Y^{(N)}= (Y^1,\ldots,Y^N)\in (\cC^{\gamma\wedge \eta}_T)^N$$
the empirical measures $$\cL_N(\xi^{(N)},B^{(N)})\in \cP_{1}(\bR^d\times \cC^\eta_T) \text{ and } \cL_N(Y^{(N)})\in \cP_1(\cC^{\eta\wedge \gamma}_T)$$
which define their laws on $(\Omega_N,\cF_N,\PP_N)$. From this point of view we can rewrite the particle system \eqref{eq:GenParticleSystem} in the more familiar form
\begin{equation}\label{eq:EmpiricalMKV}
  Y^{(N)}_t = \xi^{(N)}_0 + \int_0^t \left(\Gamma_{\dd r}\ast \cL_N(Y^{(N)})_r\right)(Y^{(N)}_r)+B^{(N)}_t.
\end{equation}
We refer to equation \eqref{eq:EmpiricalMKV} as the empirical McKean--Vlasov problem. Now we state and prove the following theorem, which is essentially Theorem~21 of \cite{coghi_deuschel_friz_maurelli_20} modified to our setting.\\ \par 
We fix $T>0$, $(\Omega,\cF,\PP)$ an abstract probability space, $\gamma,\eta$ as in \eqref{eq:EtaGammaAssumption}, $\alpha\geq 2$ and $\Gamma \in \cC^\gamma_T\cC^\alpha(\bR^d)$ satisfying the assumptions of Lemma \ref{lem:MeasureSewingStability}. Then we have the following mean field approximation result.
\begin{theorem}\label{th:AbstractMKVMeanField}
Let $\beta \in (1-\gamma,\eta \wedge \gamma)$, $p= \frac{2-\beta}{\gamma-\beta}$ and $\nu \:= \cL(\xi,B)\in \cP_{1,2p}(\bR^d\times \cC^{\eta \wedge \gamma}_T)$ be such that $B_0=0$ and $\xi\perp B$. For any $N\in \bN$ also assume we have $\left(\xi^{(N)}_0,B^{(N)}\right)\in L^1(\Omega;\bR^{Nd})\times L^{2p}(\Omega;(\cC^\eta_T)^N)$ a family of random variables with $B^{(N)}_0=0$. Then the following statements hold:
\begin{enumerate}[label=(\roman*)]
  \item\label{it:NParticleWellPosed} For every $N \in \bN$ and $\PP$-a.a. $\omega \in \Omega$ there exists a unique solution $Y^{(N)}(\omega)\in (\cC^{\eta\wedge \gamma}_T)^N$ to the empirical McKean--Vlasov problem, \eqref{eq:EmpiricalMKV}. Furthermore the mapping $\omega \mapsto Y^{(N)}(\omega)$ is $\cF$ measurable.
  \item\label{it:MeanFieldWellPosed} There exists a $Y \in \cP_1(\cC^{\eta\wedge \gamma}_T)$ such that for $\PP$-a.a. $\omega \in \Omega$, $Y(\omega)$ solves the dynamics of \eqref{eq:McKeanVlasov fixed point}.
  \item\label{it:MeanFieldApproxWithB} There exists a constant $C:=C\left(T,\Gamma,\EE\left[[B]^{2p}_{\eta;T}\right]\vee \EE\left[\EE_N\left[[B^{(N)}]^{2p}_{\eta;T}\right]\right],\gamma,\eta,\beta\right) >0$, such that for all $N\geq 1$, $\PP$-a.s. we have the bound
  \begin{equation}\label{eq:NParticleStabillityWithB}
\begin{aligned}
     \cW_{1;\cC^\beta_{T}}\left(\cL_N(Y^{(N)})(\omega),\cL(Y)\right) \leq C\bigg(& \cW_{1;\bR^d}\bigg(\cL_N\big(\xi^{(N)}_0\big)(\omega),\cL(\xi)\bigg) \\
     &\, +\cW_{2; \cC^\eta_T}\left(\cL_N\left(B^{(N)}\right)(\omega),\cL(B)\right)\bigg).
     \end{aligned}
  \end{equation}
\end{enumerate}
\end{theorem}
\begin{remark}
  The independence condition $\xi\perp B$ imposed in Theorem  \ref{th:AbstractMKVMeanField} is a consequence of the independence condition $(\zeta-\xi)\perp (\tilde{B},B)$ required by Theorem \ref{th:AbstractMKVSolStable}, where $\zeta\sim \cL(\xi^{(N)})(\omega)$ and $\tilde{B}\sim \cL(B^{(N)})(\omega)$. Since we fix $\omega\in \Omega$ for the random variables $(\xi^{(N)},B^{(N)})$, treating them as random variables on the space $\Omega_N$ defined below,  amounts to requiring $\xi\perp B$ w.r.t. $\PP$.
\end{remark}
\begin{proof}
The existence and uniqueness statements of points \ref{it:NParticleWellPosed} and \ref{it:MeanFieldWellPosed} are direct consequences of Theorem~\ref{th:AbstractMKVSol} with inputs $(\xi^{(N)}_0,B^{(N)})$ and $(\xi,B)$ on the probability spaces $(\Omega_N,\cF_N,\PP_N)$ and $(\Omega,\cF,\PP)$ respectively. The requirement that $\cL_N(\xi^{(N)}_0,B^{(N)}) \in \cP_{1,2}(\Omega_N;(\bR^d\times\cC^\eta_T)^N)$ is seen to be satisfied since for $\PP$-a.a. $\omega \in \Omega$
\begin{equation*}
  \EE_N\left[|\xi^{(N)}_0(\omega)|\right]+\EE_N\left[[B^{(N)}(\omega)]^2_{\eta;T} \right] = \frac{1}{N}\sum_{i=1}^N |\xi^i(\omega)|+ \frac{1}{N}\sum_{i=1}^N [B^i(\omega)]^2_{\eta;T} <\infty.
\end{equation*}
The measurabillity assertion of point \ref{it:NParticleWellPosed} follows from the continuity of the solution map $\bar{\Psi}_N:\cP^N_1(\bR^d\times\cC^\eta_T) \rightarrow \cP_1^N(\cC^\eta_T)$, so that $\omega \mapsto \cL_N(Y^{(N)}(\omega))$ is $\cF$ measurable so that in turn $\omega\mapsto Y^{(N)}(\omega) = S^{\cL_N(Y^{(N)}(\omega))}(\xi^{(N)}_0(\omega),B^{(N)}(\omega))$ is also $\cF$ measurable.\\ \par 
The mean field approximation result of \ref{it:MeanFieldApproxWithB} now follows directly from Theorem~\ref{th:AbstractMKVSolStable}.
\end{proof}
\section{Proofs of Main Results}\label{sec:MainResultsProofs}
We collect the proofs of Theorem \ref{th:IntroGeneralMainTheorem}, Theorem \ref{th:IntroMeanFieldApprox} and Corollary \ref{cor:RegularizeClassicalMKV}. We combine the results of, for example \cite{harang_20_cinfinity} or \cite{galeati_gubinelli_22_noiseless}, which ensure the existence of sufficiently regularising paths, with the abstract results of, Theorem \ref{th:AbstractNLETheorem} and Theorem \ref{th:AbstractMKVMeanField}, along with the preliminary results from Section \ref{sec:MeasureFlowsRegularity}.
\begin{proof}[Proof of Theorem \ref{th:IntroGeneralMainTheorem}]
In order to prove Theorem \ref{th:IntroGeneralMainTheorem} we need to ensure that given $K \in \cB^\sigma_{q,r}(\bR^d)$ there exists a $Z_t \in C([0,T];\bR^d)$ such that
\begin{equation*}
  \Gamma_{s,t} := K\ast L_{s,t},
\end{equation*}
satisfies the requirements of Theorem \ref{th:AbstractNLETheorem}. This can be done using the results of \cite{harang_20_cinfinity}, or \cite{galeati_gubinelli_22_noiseless}. More specifically, from Proposition \ref{prop: reg of avg op fBm} we see that we can choose $t\mapsto Z_t$ to be an fBm with Hurst parameter sufficiently low.
\end{proof}
\begin{proof}[Proof of Theorem \ref{th:IntroMeanFieldApprox}]
In order to prove Theorem \ref{th:IntroMeanFieldApprox}, with $\Gamma_{s,t}:= K\ast L_{s,t}$ as above, it suffices to apply Theorem \ref{th:AbstractMKVMeanField}. Since by assumption 
$$\lim_{N\rightarrow \infty}\cW_{1;\bR^d\times \cC^\eta_T}\left(\cL_N(\xi^{(N)}_0, B^{(N)}),\cL(\xi,B)\right)= 0,$$
it follows from \eqref{eq:NParticleStabillityWithB} that for any $\beta \in (1-\gamma,\gamma\wedge \eta)$ we have $\cW_{1;\cC^{\beta}_T}(\cL_N(Y^{(N)}),\cL(Y))\rightarrow 0$ as $N\rightarrow \infty$. Then, applying Proposition \ref{prop:PortmanteauTheorem} we obtain weak convergence of $\cL_N(Y^{(N)})$ to $\cL(Y)$.
\end{proof}
\begin{proof}[Proof of Corollary \ref{cor:RegularizeClassicalMKV}]
Let $K \in \cS^\prime(\bR^d)$ be a homogeneous distribution of order $\sigma<0$ (see Def. \ref{eq:HomogeneousDefinition}) and $Z \in C([0,T];\bR^d)$ distributed according to the law of an fBm with Hurst parameter $H \in (0,1)$ on a separate probability space $(\tilde{\Omega},\tilde{\cF},\tilde{\PP})$. So then from Proposition \ref{prop:HomogeneousDistBesovControl} and Remark \ref{rem:HomogeneousDistBpq}, for any $\varepsilon>0$ we have that $K \in \cB^{\sigma + \frac{d}{2}-\varepsilon}_{2,2}= \cH^{\sigma+\frac{d}{2}-\varepsilon}$. From Proposition \ref{prop: reg of avg op fBm} we see that for $\Gamma$ the averaging operator associated to $Z$, there exists a set of full measure $\tilde{\cN} \subseteq \tilde{\Omega}$ such that for all $\tilde{\omega}\in \tilde{\cN}$, $\Gamma K(\tilde{\omega}) \in \cC^\gamma_T\cC^{\sigma +\frac{1-\gamma}{H} -d\left(\frac{1}{2}-\gamma\right)}$ for any $\gamma \in \left(\frac{1}{2},1\right)$. Since we assume that $B$ takes values in $\cC^{1/2-\varepsilon}_T$ for any $\varepsilon>0$, almost surely, we have $1/2-\varepsilon+\gamma>0$ for any $\gamma\in \left(\frac{1}{2}+2\varepsilon,1\right)$. So then for all $H<\frac{1}{4-2\sigma}$ we have that $\Gamma K \in \cC^\gamma_T\cC^{\alpha}$ for some $\alpha >2$, $\gamma\in (1/2,1)$ and so the results of Theorems \ref{th:IntroGeneralMainTheorem} and \ref{th:IntroMeanFieldApprox} both apply.
\end{proof}

\section{Applications to McKean--Vlasov Equations with Homogeneous Interaction Kernels}\label{sec:Applications}
In the preceding sections we demonstrated a regularisation by noise result for generalised McKean--Vlasov equations. Many physically relevant particle systems and McKean--Vlasov models involve interaction kernels given by homogeneous distributions of negative order and particles driven either by idiosyncratic Brownian motions or with no additional forcing. Using Corollary \ref{cor:RegularizeClassicalMKV} we discuss some classical examples to which our method applies in this context, along with the necessary upper bound on the Hurst parameter of the regularising path. Note that when $B^{(N)}$ and $B$ are almost $1/2$-H\"older continuous the restriction on the Hurst parameter is entirely governed by the singularity of $K$.
\vspace{0.5em}
	\begin{enumerate}[label=(\roman*)]
	\item \textbf{Power Law Potentials:} In general, $K(x)\sim |x|^\sigma$ for $\sigma<0$. Arise in models of chemotaxis, plasma dynamics, galactic dynamics, flocking models, Landau models and Ginibre ensembles \cite{fournier_jourdain_17,biler_tadeeusz_93,serfaty_20,duerinckx_16,friesen_kutoviy_20,cattiaux_delebecque_pedesches_18}. In \cite{serfaty_20} the mean field approximation result is obtained for $(B^i)_{i=1}^N$ i.i.d Brownian motions and with repulsive kernels $K(x)\sim |x|^{-\sigma}$ for $\sigma \in (0,d)$. In \cite{bresch_jabin_wang_19} the same result is shown for $K(x)\sim -\chi |x|^{-d+1}$ for $\chi$ in a suitable region.\\ \par
	\noindent
	For $(B^i)_{i=1}^N$ at least as regular as the Brownian motion and converging in mean field scaling to $B$, then our results apply with $H <\frac{1}{4+2\sigma}$ for any $\sigma\in \bR$. \\
	\item \textbf{Biot--Savart Law:} Applied in the vorticity formulation of Euler and Navier--Stokes equations in $d=2$.
	\begin{equation*}
		K(x)\sim \frac{x^\perp}{|x|^2}, \quad x^\perp := (-\xi^2,\xi^1).
	\end{equation*}
	Since the kernel scales like the Coulomb potential in $2$-dimensions our results hold for $H<\frac{1}{6}$. In this instance, due to the rotational structure of the kernel more is known in the un-regularised case. Well-posedness of the limiting equation and propagation of chaos is known in both the viscous and inviscid cases, cf. \cite{fournier_hauray_mischler_14} and \cite{marchioro_pulvirenti_94}. A quantitative propagation of chaos result is also obtained in the viscous setting in \cite{jabin_wang_18}.
	\\
	\item \textbf{The Dirac:} Setting
	\begin{equation*}
		K(x)\sim \delta_0(x),
	\end{equation*}
	our results apply with $H<\frac{1}{4+2d}$. In \cite{sznitman_89}, Sznitman studied a particle approximation of one dimensional Burgers equation with the Dirac as the interaction kernel. Propagation of chaos and well-posedness results were shown in this case without additional regularisation but assuming $(B^i)_{i=1}^N$ to be i.i.d Brownian motions.\\
	\item \textbf{The Lennard--Jones Potentials:} Applied in particle simulations of crystallisation, the family of interactions, \cite{theil_05}
	\begin{equation*}
		K_{p,2p}(x)\sim |x|^{-2p}-2|x|^{-p}, \quad p>0,
	\end{equation*}
are known as Leonard--Jones potentials, cf. \cite{deluca_friesecke_18, theil_05}. Formally, these kernels converge to the Heitmann--Radin kernel
	\begin{equation*}
		K_{HR}(x):= \begin{cases}
			\infty, & |x|< 1,\\
			-1, & |x|=1,\\
			0, & |x|>1.
		\end{cases}
	\end{equation*}
Intuitively speaking, $K_{\text{HR}}$ acts to separate particles at distance $1$ from each other. The typical approach to theories of crystallisation is to study static minimizers of the free energy associated to the interaction, cf. \cite{dobrushin_kotecky_shlosman_92,deluca_friesecke_18, theil_05}. Therefore these models do not directly fit into our framework, however, we propose it would be interesting to consider dynamic approximations to crystalline structures, using our regularisation by noise approach. Using the Leonard--Jones potential, for fixed $p>0$ our results apply with $H<\frac{1}{4+2p}$ and $Z$ independent of $K$. In this setting we may vary $p \in [p_{\text{min}},p_{\text{max}}]$ without changing $Z$ provided we choose $H<\frac{1}{4+p_{\text{max}}}$. 
\end{enumerate}
\appendix
\section{Besov Spaces}\label{app:BesovSpaces}
We recall some definitions and standard results regarding the scale of Besov spaces on $\bR^d$. We refer to \cite{bahouri_chemin_danchin_11} for more details. We define the Fourier transform, and its inverse on $L^1(\bR^d)$ by setting,
\begin{equation*}
	\cF f(z) := \int_{\bR^d}f(x)e^{-ix\cdot z}\,\dd x,\quad 	\cF^{-1}f(x) := \frac{1}{(2\pi)^d} \int_{\bR^d}f(z)e^{i x\cdot z}\, \dd z.
\end{equation*}
It is standard that $\cF$ fixes $\cS(\bR^d)$ so we may extend both definitions to the tempered distributions by duality. We also recall the definition of a Fourier multiplier. For any $f \in \cS^\prime(\bR^d)$ and $\varphi:\bR^d\rightarrow \bR$ measurable and with at most polynomial growth, we define,
\begin{equation*}
	\varphi(D) f := \cF^{-1} \left(\varphi \cF f \right) = \left(\cF^{-1}\varphi\right)\ast f,
\end{equation*}
where $D$ is a placeholder for the derivative operator on $\bR^d$.
\subsection{Dyadic Partition of Unity and Littlewood-Payley Blocks}
We let $\tilde{\chi},\,\chi \in C^\infty_c(\bR^d)$ be such that
\begin{enumerate}
	\item $\supp\, \tilde{\chi}\subset B_{\frac{4}{3}}(0)$ and $\supp\, \chi \subset B_{\frac{8}{3}}(0)\setminus B_{\frac{3}{4}}(0)$,
	\item $\tilde{\chi}(\zeta) + \sum_{k=0}^\infty \chi(2^{-k}\zeta)=1$, for all $\zeta \in \bR$.
\end{enumerate}
The existence of such a dyadic partition of unity is shown in \cite[Prop. 2.10]{bahouri_chemin_danchin_11}. For $k\geq 0$ we define $\chi_k(\,\cdot\,) := \chi(2^{-k}\,\cdot\,)$ and set $\chi_k=0$ for all $k<-1$.\\ \par 
For any $f\in \cS^\prime(\bR^d)$ we define the inhomogeneous Littlewood--Paley blocks by setting,
\begin{equation}\label{eq:LPDecompose}
	\begin{aligned}
		\Delta_{-1} f &:= \tilde{\chi}(D)f= \tilde{h}\ast f,\\
		\Delta_k f &:=\chi_k(D)f=h(2^{k}\,\cdot\,)\ast f, \quad \forall\,k\geq 0,
	\end{aligned}
\end{equation}
where $\tilde{h}= \cF^{-1}\tilde{\chi}$ and $h = \cF^{-1}\chi$. Since $\tilde{h},\,h \in \cS(\bR^d)$, the operators $\Delta_k$ map $L^p$ to $L^p$ for any $p\in [1,\infty]$ with norms independent of $p$ and $k$.
\begin{definition}[Inhomogeneous Besov Spaces]\label{def:InHomBesovSpace}
	For $\alpha\in \mathbb{R}$ and $p,q\in[1,\infty]$, the inhomogeneous Besov space $\cB^\alpha_{p,q}(\bR^d)$ is defined by 
	\begin{equation*}
		\cB_{p,q}^\alpha(\bR^d) =\left\{f\in \mathcal{S}^\prime(\bR^d)\, :\,\|f\|_{\cB_{p,q}^\alpha(\bR^d)}:= \left(\sum _{k\geq-1} 2^{kq\alpha}\|\Delta_k f\|_{L^p}^q\right)^{\frac{1}{q}} <\infty \right\}.
	\end{equation*}
\end{definition}
For $p=q=\infty$ we use the notation
\begin{equation*}
	\cC^{\alpha}(\bR^d):=\Big\{f\in \mathcal{S}^\prime \,:\,\|f\|_{\cB_{\infty,\infty}^\alpha(\bR^d)}:= \sup_{j\geq -1} 2^{j\alpha}\|\Delta_j f\|_{L^\infty(\bR^d)} <\infty \Big\}.
\end{equation*}
For $\alpha >0$ and not an integer these spaces agree with the usual H\"older spaces, however for $a \in \bN$, $\cC^a(\bR^d)$ coincides with $W^{a,\infty}(\bR^d)$. When $p=q=2$ we use the special notation $\cH^\alpha(\bR^d) = \cB^\alpha_{2,2}(\bR^d)$ to denote the scale of Hilbertian Sobolev spaces, on which an equivalent norm is given by the expression
\begin{equation*}
	\|f\|_{\cH^\alpha} := \|(1+|\,\cdot\,|)^\alpha\cF(f) \|_{L^2}.
\end{equation*}
The Besov spaces enjoy a number of useful properties which we list below. Proofs of the following statements can be found in \cite{bahouri_chemin_danchin_11}. 
\begin{enumerate}[label=(\roman*)]
	\item Embeddings: for $\alpha \in \bR$, $1\leq p_1\leq p_2\leq \infty$ and $1\leq q_1\leq q_2 \leq \infty$ one has,
	\begin{equation}\label{eq:PwiseBesovEmbedding}
		\|\,\cdot\,\|_{\cB^{\alpha}_{p_2,q_2}} \lesssim \|\,\cdot\,\|_{\cB^{\alpha+d\left(\frac{1}{p_1}-\frac{1}{p_2}\right)}_{p_1,q_1}}.
	\end{equation}
	We also have the following, continuous embeddings,
	\begin{align}
		\|f\|_{\cB^\alpha_{p,q}}&\lesssim \|f\|_{\cB^{\alpha'}_{p,q}}\quad \alpha<\alpha' \in \bR,\label{eq:PwiseHolderBesovCompactEmbedding}\\
		\|f\|_{\cB^\alpha_{p,q}}&\lesssim \|f\|_{\cB^{\alpha}_{p,q'}}\quad q>q' \in [1,\infty],\label{eq:PwiseBesovEmbeddingQQ'}\\
		\|f\|_{\cB^\alpha_{p,q}}&\lesssim \|f\|_{\cB^{\alpha'}_{p,q'}}\quad \alpha<\alpha' \in \bR, \forall \, q\leq q'\in [1,\infty].\label{eq:PwiseBesovEmbeddingQQ'Reg}
	\end{align}
and the embedding $\cB^{\alpha'}_{p,q} \hookrightarrow\cB^{\alpha}_{p,q}$ of \eqref{eq:PwiseHolderBesovCompactEmbedding} is compact.
	\item Relations to $L^p$ spaces: For $p\in [1,\infty]$ one has,
	\begin{align*}
		\|f\|_{L^p}\lesssim \|f\|_{\cB^{0}_{p,1}},\quad
		\|f\|_{\cB^{0}_{p,\infty}}\lesssim \|f\|_{L^p}.
	\end{align*}
\end{enumerate}
A version of Young's convolution inequality holds in Besov spaces.
\begin{lemma}\label{lem: Youngs Convolution}
	(Young's convolution inequality) For $\alpha,\beta\in \mathbb{R}$, let $f\in \cB^\beta_{p,\infty}(\bR^d)$ and $g\in \cB^\alpha_{q,\infty}(\bR^d)$, and let $r\in[1,\infty]$  be defined through the relation $\frac{1}{r}+1=\frac{1}{p}+\frac{1}{q}$. Then
	\begin{equation}\label{eq:PwiseBesovYoung}
		\|f\ast g\|_{\cB^{\alpha+\beta}_{r,\infty}} \lesssim \|f\|_{\cB^{\beta}_{p,\infty}} \|g\|_{\cB^{\alpha}_{q,\infty}}.
	\end{equation}
\end{lemma}
\subsection{Functions of H\"older Continuity on Intervals of the Real Line}\label{subsec:HomogeneousDist}

The next lemma gives a useful criteria for extending a local control on the H\"older continuity of a path to a global one. The statement is based on \cite[Exercise 4.24]{friz_hairer_14}, and we include a short proof similar to the one given there. 
\begin{lemma}\label{lem:LocalToGlobalHolderBound}
	Let $E$ be a Banach space, $\alpha\in (0,1)$, $T>0$ and $X:[0,T]\rightarrow E$. Suppose that there exists a constant $M>0$ and $h\in (0,T]$ such that for any $t \in [0,T-h]$ we have $[X]_{\alpha;[t,t+h]} \leq M$. Then $X$ is $\alpha$-H\"older continuous on $[0,T]$ and
	\begin{equation*}
		[X]_{\alpha;T} \leq M(1\vee 2h^{\alpha-1})T^{1-\alpha}. 
	\end{equation*}
\end{lemma}
\begin{proof}
	We need to show that for any $0\leq s\leq t\leq T $ then $\frac{\|X_{s,t}\|_E}{|t-s|^\alpha}\leq M(1\vee 2h^{\alpha-1})T^{1-\alpha}$. In the case when $|t-s|\leq h$ there is nothing to prove, so let $|t-s|\geq h$. Define $t_i=(s+ih)\wedge t$ for $i\in \bN$. Note that for $N\geq (t-s)/h$, $t_N=t$, and that $t_{i+1}-t_i\leq h$ for all $i\in \bN$. Therefore, we have
	\begin{equation*}
		\|X_{s,t}\|_E\leq \sum_{0 \leq i < (t-s)/h} \|X_{t_{i},t_{i+1}}\|_E \leq M h^\alpha \left(1+\frac{t-s}{h}\right) \leq 2Mh^
		\alpha \frac{t-s}{h} \leq 2Mh^{\alpha-1}|t-s|^\alpha T^{1-\alpha},
	\end{equation*}
which concludes the proof.
\end{proof}
Working with the Wasserstein distance for probability measures on H\"older spaces for the contraction arguments in the proofs of existence and uniqueness, we need a similar property to that in Lemma \ref{lem:LocalToGlobalHolderBound} for these distances. The next lemma is therefore a variation of Lemma \ref{lem:LocalToGlobalHolderBound} adapted to this special case.  
\begin{lemma}\label{lem:WassLocalToGlobal}
    Let $T>0$, $\alpha \in (0,1)$ and $\mu^1,\,\mu^2 \in \cP(C_T^\alpha)$. Suppose there exists a constant $M>0$ and $h\in (0,T]$ such that for any $t\in[0,T-h]$,
    \begin{equation}\label{eq:WassScalingAssumption}
        \cW_{1;C^\alpha_{[t,t+h]}}(\mu^1,\mu^2) \leq M.
    \end{equation}
    Then
    \begin{equation}\label{eq:WassScalingConclusion}
        \cW_{1;C_T^\alpha}(\mu^1,\mu^2)\leq \cW_{1;\bR^d}(\mu^1_0,\mu^2_0) + M(1\vee 2h^{-1}T).
    \end{equation}
\end{lemma}
\begin{proof}
We first note that for $\mu^1,\,\mu^2\in \cP(\cC^{\alpha}_T)$ there exists an optimal coupling given by random variables $Y^1,\,Y^2\in \cC^\alpha_T$ on a common probability space such that
\begin{equation*}
    \cW_{1;\cC^{\alpha}_{[t,t+h]}}(\mu^1,\mu^2) = \EE\left[\|Y^1-Y^2\|_{\alpha;[t,t+h]}\right] = \EE\left[|Y_t-Y_t|_{\bR^d}\right] + \EE\left[[Y^1-Y^2]_{\alpha;[t,t+h]}\right],
\end{equation*}
and in particular,
\begin{equation*}
     \cW_{1;\cC^{\alpha}_T}(\mu^1,\mu^2) = \EE\left[\|Y^1-Y^2\|_{\alpha;T}\right] = \cW_{1;\bR^d}(\mu^1_0,\mu^2_0) + \EE\left[[Y^1-Y^2]_{\alpha;T}\right].
\end{equation*}
From \eqref{eq:WassScalingAssumption} it follows that for any $t\in [0,T-h]$,
\begin{equation*}
    \EE\left[[Y^1-Y^2]_{\alpha;{[t,t+h]}}\right] \leq M
\end{equation*}
and so in order to conclude it suffices to show that for a given random variable $X$, if
    $$\EE\left[[X]_{\alpha;[t,t+h]}\right]\leq M,$$
    uniformly over $t\in [0,T-h]$, then
    $$\EE\left[[X]_{\alpha;T}\right]\leq M\left(1\vee 2h^{-1}T\right).$$
    To this end, we will follow a similar procedure as used in the proof of Lemma \ref{lem:LocalToGlobalHolderBound}. We begin to observe that 
    \begin{equation*}
        \EE\left[[X]_{\alpha;T} \right] = \EE\left[\sup_{\substack{0\leq s<t\leq T\\
       |t-s|<h }} \frac{|X_{s,t}|}{|t-s|^\alpha}\right] + \EE\left[\sup_{\substack{0\leq s<t\leq T\\
       |t-s|\geq h }} \frac{|X_{s,t}|}{|t-s|^\alpha}\right] \leq M+ \EE\left[\sup_{\substack{0\leq s<t\leq T\\
       |t-s|\geq h }} \frac{|X_{s,t}|}{|t-s|^\alpha}\right].
    \end{equation*}
    To control the second term let us define $t_i = (s+ih)\wedge t$ for $i\in \bN$ so that for $N\geq (t-s)/h$, $t_N=t$ and $|t_{i+1}-t_i|\leq h$ for all $i =\{1,\ldots,N-1\}$. Also note that since $\sup_{t\in [0,T-h]}\EE[[X]_{\alpha;[t,t+h]}]\leq M$ the random variables $[X]_{\alpha;[t,t+h]}$ is $\PP$-a.s. finite. We therefore have,
    \begin{align*}
        \EE\left[\sup_{\substack{0\leq s<t\leq T\\
       |t-s|\geq h }} \frac{|X_{s,t}|}{|t-s|^\alpha}\right]
       &\leq \EE\left[\sup_{\substack{0\leq s<t\leq T\\
       |t-s|\geq h }} \frac{1}{|t-s|^\alpha} \sum_{0\leq i <(t-s)/h}[X]_{t_i,t_i+h} h^{\alpha}  \right]\\
       &\leq \sum_{0\leq i \leq T/h}\EE\left[[X]_{t_i,t_i+h}\right]\\
       & \leq 2h^{-1}T M,
    \end{align*}
    where in the last passage we have used that $(1+\frac{T}{h})\leq 2Th^{-1}$. This concludes the proof. 
\end{proof}
\subsection{Besov Regularity of Homogeneous Distributions}\label{subsec:HomogeneousDistributions}
In Section \ref{sec:MainResultsProofs} we discuss applications of our general result \eqref{th:IntroGeneralMainTheorem} to some specific McKean--Vlasov problems where $K$ is a given homogeneous distribution. In this subsection we discuss the regularity of these distributions in Besov spaces. \\ \par
For $\varphi \in \cS(\bR^d)$ and $\lambda>0$ we define the dilation by
\begin{equation}
	\varphi_\lambda(x):= \lambda^{-d}\varphi\left(\lambda^{-1}x_1,\ldots,\lambda^{-1}x_d\right).
\end{equation}
\begin{definition}[Homogeneous Distribution]\label{sec:HomogeneousDistributions}
	We say that $K\in \cS^\prime(\bR^d)$ is homogeneous of degree $\sigma\in \bR$ if for any $\varphi\in \cS(\bR^d)$ and $\lambda>0$ one has,
	\begin{equation}\label{eq:HomogeneousDefinition}
		\langle K, \varphi_\lambda\rangle = \lambda^{\sigma}\langle K,\varphi\rangle.
	\end{equation}
\end{definition}
Replacing $\bR^d$ with the punctured domain $\bR^d\setminus\{0\}$ we can instead define the notion of homogeneous distributions in $\cS^\prime(\bR^d\setminus\{0\})$. For any $\sigma\in \bR\setminus \bZ_{\leq-d}$ all homogeneous distributions on $\cS^\prime(\bR^d\setminus\{0\})$ of order $\sigma$ are given by functions of the form
\begin{equation}\label{eq:HomogeneousDistributionForm}
	\tilde{K}_\sigma(x)=f\left(\frac{x}{|x|} \right)|x|^\sigma,
\end{equation}
where $f \in \cS^\prime(\bS^{d-1})$ is a distribution on the $d$-dimensional unit sphere. If $\sigma >-d$ then $\tilde{K}_\sigma(x)$ extends uniquely to a homogeneous distribution $K_\sigma \in \cS^\prime(\bR^d)$ without modification. However, for $\sigma \leq -d$ the question of extending $\tilde{K}_\sigma$ to a distribution on the un-punctured plane is more complicated. For a full discussion see \cite[Sec. 3.2]{hormander_03}, while we simply state the main results here.\\ \par 
When $\sigma<-d$ and not an integer there exists a unique extension $K_\sigma$ defined by
\begin{equation}\label{eq:HomogeneousDistPVExtension}
	\langle K_\sigma,\varphi\rangle = \int_{\bR^d}\tilde{K}_\sigma(x) \left(\varphi(x)-P^\varphi_{\floor{-\sigma-d};0}(x)\right)\,\dd x
\end{equation}
where $P^\varphi_{k;0}$ is the Taylor polynomial to order $k-1$ of $\varphi$ at $0$. This is proved as \cite[Theorem 3.2.3]{hormander_03}. We refer to \eqref{eq:HomogeneousDistPVExtension} as the principle value extension of $\tilde{K}_\sigma$.
\\ \par 
For $\sigma \in \bZ_{\leq -d}$ the formula \eqref{eq:HomogeneousDistPVExtension} does define an extension of $\tilde{K}_\sigma$ but it fails to be unique. One can always add any linear combination of sufficiently high derivatives of the Dirac delta, cf. \cite[Thm. 3.2.4]{hormander_03}. For $\sigma=-n$ with $n\in \bZ_{\geq d}$ we choose to define the extension of $\tilde{K}_n$ by the formula
\begin{equation*}
	\langle K_n,\varphi\rangle = \int_{\bR^d}\tilde{K}_n(x) \left(\varphi(x)-P^\varphi_{n-d;0}(x)\right)\,\dd x + \sum_{|a|=n-d}D^{a}\varphi(0).
\end{equation*}
Defining the convolution of a distribution in the usual way, we see that for all $\sigma <0$ we have
\begin{equation}\label{eq:RenormKernelConv}
	\begin{aligned}
	(K_{\sigma}\ast\varphi)(x) :&=\int_{\bR^d}\tilde{K}_\sigma(x-y) \left(\varphi(y)-P^\varphi_{\floor{-\sigma-d};0}(y)\right)\,\dd y \\
	&\quad + \mathds{1}_{\{\sigma \in \bZ_{\leq -d}\}}\sum_{|a|=\sigma-d}D^{a}\varphi(0),
	\end{aligned}
\end{equation}
for all $\varphi \in \cS(\bR^d)$. The Fourier transform of a homogeneous distribution is nicely described by the following theorem.
\begin{proposition}\label{prop:FTofHomogeneousDist}
	Let $K\in \cS^\prime(\bR^d)$ be a homogeneous distribution of order $\sigma \in \bR$, then $\cF K \in \cS^\prime(\bR^d)$ and is a homogeneous distribution of order $-(\sigma+d)$.
\end{proposition}
\begin{proof}
	See the proof of \cite[Theorem 7.1.16]{hormander_03}
\end{proof}
This result suggests that $K_\sigma$ should be controlled in a suitable space of negative regularity, in fact it almost immediately follows that $K_\sigma \in \cH^{\sigma+\frac{d}{2}}$. Below we give a self contained proof that for $\sigma<0$, $K_\sigma \in \cB^{\sigma+\frac{d}{p}}_{p,\infty}$ for any $p\in [1,\infty]$.
\begin{theorem}\label{prop:HomogeneousDistBesovControl}
		Let $\sigma <0$ and $\tilde{K}_\sigma\in \cS^\prime(\bR^d\setminus\{0\})$ be the distribution described by \eqref{eq:HomogeneousDistributionForm} and let $K_\sigma$ be its principle value extension to $\cS^\prime(\bR^d)$ defined in \eqref{eq:HomogeneousDistPVExtension}. Then $K_\sigma \in \cB^{\sigma+d/p}_{p,\infty}(\bR^d)$ for any $p,\,q\in [1,\infty]$.
\end{theorem}
\begin{proof}
	From the Besov embeddings \eqref{eq:PwiseBesovEmbedding}, for any $p,q\in [1,\infty]$ we have that
	\begin{equation*}
		\|K_\sigma\|_{\cB^{\sigma+d/p}_{p,\infty}} \lesssim \|K_\sigma\|_{\cB^{\sigma+d}_{1,\infty}},
	\end{equation*}
	so we concentrate on showing that $K_\sigma \in \cB^{\sigma+d}_{1,\infty}$. From Proposition \ref{prop:FTofHomogeneousDist} in combination with the definition of the Littlewood--Paley blocks, \eqref{eq:LPDecompose}, and using \eqref{eq:HomogeneousDefinition}, for any $k \geq 0$ we have
	\begin{equation*}
		\Delta_{k}K_{\sigma}= 2^{-(\sigma +d)k}\Delta_0 K_\sigma.
	\end{equation*}
Therefore, we have
\begin{equation}
	\sup_{k\geq 0} 2^{(\sigma+d)k}\|\Delta_k K_\sigma\|_{L^1}= \|\Delta_0 K_\sigma\|_{L^1}.
\end{equation}
So it suffices to show that $\|\Delta_{-1}K_\sigma\|_{L^1},\,\|\Delta_0 K_\sigma\|_{L^1}$ are both finite. We choose a smooth, cut-off function $\psi\in C^\infty_c(\bR^d)$ such that $\supp(\psi)= B_1(0)$ and $\|\psi\|_{C^\infty_c(\bR^d)}\leq 1$. Then we write,
	\begin{equation*}
		\tilde{K}_\sigma = \psi \tilde{K}_\sigma+(1-\psi)\tilde{K}_\sigma := \tilde{K}_{\sigma,0}+ \tilde{K}_{\sigma,1}
	\end{equation*}
	and define the principle value extensions $K_{\sigma,0}$ and $K_{\sigma,1}$ analogously. Then we divide the proof into two cases, $-d<\sigma <0$ and $\sigma\leq -d$.
\\ \par 
	First consider the case $-d<\sigma<0$. We directly have that $K_{\sigma,0}\in L^1(\bR^d)$ so since $\Delta_{-1}$ and $\Delta_0$ are both bounded maps $L^p\rightarrow L^p$ we have $\Delta_{-1}K_{\sigma,0},\,\Delta_{0}K_{\sigma,0}\in L^1(\bR^d)$. Regarding the part supported away from the origin, using that the functions $h,\,\tilde{h}$ from \eqref{eq:LPDecompose} decay faster than any polynomial we also have $\Delta_{-1}K_{\sigma,1},\,\Delta_0 K_{\sigma,1} \in L^1(\bR^d)$.\\\par 
	
	When $\sigma\leq -d$ the situation is reversed. In this case we see that $K_{\sigma,1}\in L^1(\bR^d)$ directly and so by the boundedness of $\Delta_0$ and $\Delta_{-1}$ as mappings $L^p\rightarrow L^p$ we have $\Delta_{-1}K_{\sigma,1},\,\Delta_{0}K_{\sigma,1}\in L^1(\bR^d)$. 
	Regarding the compactly supported term the proofs for $\Delta_{-1}K_{\sigma,0}$ and $\Delta_{0}K_{\sigma,0}$ are very similar, and so we only present the $-1$ block. Using Taylor's theorem we have,
	\begin{align*}
		\Delta_{-1}K_{\sigma,0}(x)
		&\leq \int_{B_1(x)}|x-y|^\sigma\left|\tilde{h}(y)-P^{\tilde{h}}_{k;y}(x)\right|\,\dd y + \mathds{1}_{\sigma\in \bZ_{\leq -d}}D^{|\sigma|-d}\tilde{h}(x)\\
		&\leq \|D^{k+1}\tilde{h}\|_{L^\infty(B_1(x))}\left(\int_{B_1(x)}|x-y|^{\sigma+k+1}\,\dd y +1\right)\\
		&\lesssim (1+|x|^{k+1})^{-1}\|\tilde{h}\|_{k+1,\cS},
	\end{align*}
	where we used the fact that $\sigma+k+1>-d$ to evaluate the integral. The last line is integrable over $\bR^d$ and so we have $\Delta_{-1}K_{\sigma,0} \in L^1(\bR^d)$. Applying the same argument to $\Delta_{0}K_{\sigma,0}$ we have $\|\Delta_{0}K_{\sigma,0}\|_{L^1}<\infty$.\\ \par
	In conclusion, for any $\sigma <0$ we have
	\begin{equation*}
		\|K_\sigma\|_{\cB^{\sigma+d}_{1,\infty}} = \sup_{k\geq -1}2^{(\sigma +d)k} \|\Delta_k K_\sigma\|_{L^1} \leq \sup_{k \in \{0,1\}} \|\Delta_{k}K_\sigma\|_{L^1} <\infty.
	\end{equation*}
\end{proof}
\begin{remark}\label{rem:HomogeneousDistBpq}
	Using the Besov embedding \eqref{eq:PwiseBesovEmbeddingQQ'Reg}, for any $\varepsilon>0$ we also have that $K_\sigma \in \cB_{p,q}^{\sigma+d/2-\varepsilon}(\bR^d)$ for any $q \in [1,\infty)$.
\end{remark}

\end{document}